\documentclass{amsart}
\usepackage[english]{babel}
\usepackage{amssymb}
\usepackage{amsmath}
\usepackage{amsfonts}
\usepackage{amsthm}
\usepackage{enumerate}
\usepackage{color}
\usepackage{todonotes}

\newtheorem{hypothesis}{Hypothesis}
\newtheorem{theorem}{Theorem}[section]
\newtheorem{lemma}[theorem]{Lemma}
\newtheorem{proposition}[theorem]{Proposition}
\newtheorem{corollary}[theorem]{Corollary}

\theoremstyle{definition}

\theoremstyle{remark}
\newtheorem{remark}[theorem]{Remark}
\def\R{{\mathbb R}}

\def\K{{\mathbb K}}
\def\D{{\mathbb D}}
\newcommand{\T}{\mathbb{T}}
\newcommand{\Z}{\mathbb{Z}}
\newcommand{\N}{\mathbb{N}}

\def\F{\mathcal{F}}
\numberwithin{equation}{section}

\def\supp{\mathop{\rm supp}\nolimits}
\def\sign{\mathop{\rm sign}\nolimits}
\newcommand\cro[1]{\langle #1 \rangle}

\begin{document}

\title[Well-posedness for dispersive Burgers equation]{Low regularity well-posedness of nonlocal dispersive perturbations of Burgers' equation}

\author[L. Molinet, D. Pilod and S. Vento]{Luc Molinet, Didier Pilod  and St\'ephane Vento}

\address{Luc Molinet,  Institut Denis Poisson, Universit\'e de Tours, Universit\'e d'Orl\'eans, CNRS (UMR 7013), Parc Grandmont, 37200 Tours, France.}
\email{Luc.Molinet@univ-tours.fr}

\address{Didier Pilod, Department of Mathematics, University of Bergen, PO Box 7803, 5020 Bergen, Norway. }
\email{Didier.Pilod@uib.no}

\address{St\'ephane Vento, Universit\'e Sorbonne Paris Nord, Laboratoire Analyse, G\'eom\'etrie et Applications, LAGA, CNRS, UMR 7539, F-93430, Villetaneuse, France.}
\email{vento@math.univ-paris13.fr}

\date{\today}

\subjclass[2020]{35A01, 35A02, 35B45, 35E15, 35Q53}
\keywords{Nonlocal nonlinear dispersive equations, Local and global well-posedness, Modified energy}
\maketitle

\begin{abstract} 
We consider the Cauchy problem associated to a class of dispersive perturbations of Burgers' equations, which contains the low dispersion Benjamin-Ono equation, (also known as low dispersion fractional KdV equation),
$$ \partial_tu-D_x^{\alpha}\partial_xu=\partial_x(u^2) \, ,$$ 
and prove that it is locally well-posed in $H^s(\mathbb K)$, $\mathbb K=\mathbb R$ or $\mathbb T$, for $s>s_{\alpha}$, where 
\begin{equation*}
s_\alpha=\begin{cases}
 1-\frac{3\alpha}4 & \text{for} \quad \frac23 \le \alpha \le 1; \\ 
  \frac 32(1-\alpha) &  \text{for}  \quad \frac13 \le \alpha \le \frac23; \\ 
 \frac 32-\frac{\alpha}{1-\alpha}  & \text{for} \quad 0 < \alpha  \le \frac13 .
 \end{cases}
\end{equation*} 
The uniqueness is unconditional in $H^s(\mathbb K)$ for $s>\max\{\frac12,s_{\alpha}\}$. Moreover, we obtain \emph{a priori} estimates for the solutions at the lower regularity threshold $s>\widetilde{s}_\alpha$ where
\begin{equation*}
\widetilde{s}_\alpha=\begin{cases}
 \frac 12-\frac \alpha 4 & \text{for} \quad \frac23 \le \alpha \le 1; \\ 
  1-\alpha &  \text{for}  \quad \frac12 \le \alpha \le \frac23; \\ 
 \frac 32-\frac{\alpha}{1-\alpha}  & \text{for} \quad 0 < \alpha  \le \frac12 .
 \end{cases}
\end{equation*}
As a consequence of these results and of the Hamiltonian structure of the equation, we deduce global well-posedness in $H^s(\mathbb K)$ for $s>s_{\alpha}$ when $\alpha>\frac23$, and in the energy space $H^{\frac{\alpha}2}(\mathbb K)$ when  $\alpha>\frac45$.

The proof combines an energy method for strongly nonresonant dispersive equations, introduced by the first and third authors, with refined Strichartz estimates and modified energies. Moreover, we use a full symmetrisation of the modified energy both for the \textit{a priori} estimate and for the estimate of the difference of two solutions. This symmetrisation allows some cancellations of the resulting symbol which are crucial to close the estimates. 
\end{abstract}

\section{Introduction}

\subsection{Dispersive perturbations of Burgers' equation}
In this paper we continue our investigation of the well-posedness at low regularity of nonlocal dispersive perturbations of the Burgers' equation. More specifically, we consider the class of initial value problems (IVP)
\begin{equation}\label{fKdV}
 \begin{cases} 
 \partial_tu -L_{\alpha+1}u = \partial_x(u^2), \\ 
 u(0,\cdot)=u_0,
 \end{cases}
\end{equation}
where $u=u(t,x)$ is a real valued function, $t \in \mathbb R$, $x\in \K=\R$ or $\T(=\R/2\pi \Z)$ and $L_{\alpha+1}$ is a Fourier multiplier of symbol $i\omega_{\alpha+1}$, with $\omega_{\alpha+1}$ behaving in absolute value at high frequency as $|\xi|^{\alpha+1}$ for $0<\alpha \le 1$. We refer to Hypothesis \ref{hyp1} below for more precise assumptions about $\omega_{\alpha+1}$. 

This class of equation contains in particular the fractional KdV equation (fKdV), also known as dispersion generalized Benjamin-Ono equation, 
\begin{equation} \label{fKdVbis}
\partial_tu -D^{\alpha}\partial_xu = \partial_x(u^2) ,
\end{equation}
where $D^\alpha$ is the Fourier multiplier of symbol $|\xi|^{\alpha}$.
In the cases where $\alpha=2$, respectively $\alpha=1$, equation \eqref{fKdVbis} corresponds to the well-known Korteweg-de Vries (KdV), respectively Benjamin-Ono (BO) equations, while in the case $\alpha=0$, it corresponds to the inviscid Burgers equation. 

Also of physical interest with connections to fluid mechanics are the cases $\alpha=\pm \frac12$: the Whitham equation 
\begin{equation} \label{Whitham}
\partial_tu -w(D)\partial_xu = \partial_x(u^2) ,
\end{equation}
where $w(D)$ is the Fourier multiplier of symbol $w(\xi)= \left( \frac{\tanh (\xi)}{\xi}\right)^{\frac12} (1+\tau \xi^2)^{\frac12}$ is included in the class of equations \eqref{fKdV}. It was obtained by Whitham \cite{Whitham67} by replacing the dispersion $\partial_x^3$ in the KdV equation by that of the linearized water wave equations around the equilibrium. Here $\tau \ge 0$ is a parameter related to surface tension, so that the case of pure gravity waves ($\tau=0$) corresponds to $\alpha=-\frac12$, while the case of capillary-gravity waves ($\tau>0$) corresponds to $\alpha=\frac12$ in our setting. We also refer to \cite{KLPS18} for a detailed survey on the Whitham equation.

While both the KdV and the BO equation are completely integrable, no integrable structures are known  for the fKdV equations with non-integer values of $\alpha$. However, equation \eqref{fKdVbis} enjoys an Hamiltonian structure and the quantities 
\begin{equation} \label{Mass}
M(u)=\int_{\mathbb K} u^2(t,x) dx
\end{equation}
and
\begin{equation} \label{MassEnergy}
H(u)=\frac12 \int_{\mathbb K} (D^{\frac{\alpha}2}u)^2(t,x)dx+\frac13 \int_{\mathbb K} u^3(t,x) dx 
\end{equation}
are conserved by its flow.
Moreover, the solutions of fKdV are preserved under the scaling transformation $u_{\lambda}(x,t)=\lambda^{\alpha}u(\lambda x,\lambda^{1+\alpha}t)$,  $\lambda>0$, which leaves the $\dot{H}^s$ invariant for $s=\frac12-\alpha$. In particular, equation \eqref{fKdV} is \emph{mass-critical} for $\alpha=\frac12$ and \emph{energy critical} for $\alpha=\frac13$. As explained above, the case $\alpha=\frac12$ is of particular interest because of its connection with fluid mechanics. 

\subsection{Well-posedness review on fractional KdV equations}

We are interested in the local well-posedness of the IVP associated to \eqref{fKdV} in the $L^2$-based Sobolev spaces $H^s$ at low regularity. The notion of \emph{local well-posedness} (LWP) is to be understood in the sense of Hadamard and Kato, meaning that for every initial datum $u_0 \in H^s$, there exist a positive time $T=T(\|u_0\|_{H^s})$ and a solution $u \in C([0,T] :  H^s)$ of \eqref{fKdV} satisfying $u(0,\cdot)=u_0$. This solution is unique, at least in the resolution space in which it is constructed. Moreover, the flow map data-solution which to the initial datum $u_0$ maps the solution $u$ is continuous from $H^s$ to $C([0,T'] : H^s)$, for any $T'<T$. If the solution map is smooth, the flow is said to be \emph{semi-linear}, otherwise, if the solution map is only continuous, the flow is said to be \emph{quasi-linear}. Finally, when $T$ can be chosen arbitrarily large in the definition above, the IVP is said to be \emph{globally well-posed} (GWP) in $H^s$. 

Based on the numerical evidences obtained by Klein and Saut in \cite{KlSa15}, it is conjectured that the IVP associated to \eqref{fKdVbis} is globally well-posed  in the sub-critical case $\alpha>\frac12$, whereas blow-up may occur for $ \alpha\le 1/2$. In the previous work \cite{MPV}, we proved that all "sufficiently  smooth" solutions are global  for $\alpha>\frac67$. Here, we extend that result to $\alpha>\frac23$.

We now carry out a review of the well-posedness of the IVP associated with the classes \eqref{fKdV} and \eqref{fKdVbis} in $H^s$. Using classical energy estimates and compactness arguments, one obtains local well-posedness in $H^s$ for $s > \frac{3}{2}$, both on the real line and the torus, for the IVP \eqref{fKdV}, under the sole assumption that the operator $L_{\alpha+1}$ is skew-adjoint (see, for instance, \cite{ABFS89}). This regularity threshold is optimal in the absence of dispersion: the Burgers' equation, as noted in Remark 1.6 of \cite{LiPiSa14}, fails to be well-posed in $H^{\frac32}(\mathbb R)$. 

When $\alpha > 0$, the dispersive nature of the equation provides additional smoothing, which can be exploited to lower the regularity requirement. This was achieved for the KdV equation in the 90's in seminal works by Kenig, Ponce, and Vega \cite{KPV93,KPV96}, and by Bourgain \cite{Bo93}, by applying a fixed point argument on the associated Duhamel formulation in some suitable resolution spaces. The resolution space of Kenig, Ponce, and Vega is constructed using the linear smoothing effects of the equation, namely Strichartz estimates, local Kato smoothing, and maximal function estimates, while Bourgain's approach is based on the Fourier restriction norm method. As a consequence of these results, the flow map of KdV is semi-linear in any Sobolev space $H^s$, with enough regularity ($s>-\frac34$ to be precise). 

However, this result breaks down, at least on the line, as soon as $\alpha<2$. Indeed it was proved by Molinet, Saut and Tzvetkov in \cite{MoSaTz01} that the flow map is quasi-linear in any Sobolev space $H^s(\mathbb R)$, $s\in \mathbb R$, so it is not possible to solve the IVP by using a fixed-point argument on the associated Duhamel formulation. One has then to come back to compactness methods and energy estimates to prove well-posedness below the threshold $s=\frac32$.

This was done for the Benjamin-Ono equation by Ponce \cite{Po91} who established global well-posedness in $H^\frac32(\mathbb R)$ by using the Strichartz estimate $\|e^{tD\partial_x}u_0\|_{L^4_tL^{\infty}_x} \lesssim \|u_0\|_{L^2}$. Later on, Koch and Tzvetkov \cite{KoTz03}, and thereafter Kenig and Koenig \cite{KK03}, pushed down this result up to $s>\frac54$ and $s>\frac98$, respectively. These improvements were achieved by estimating the norm $\int_0^T\|\partial_xu(t)\|_{L^{\infty}_x}dt $ of a solution  localised around a dyadic frequency $N$.  The key idea was to partition the time interval $[0, T]$ into subintervals of length approximately $N^{-\delta}$, for some suitable $\delta > 0$, apply Strichartz estimates on each subinterval, and then sum over the contributions from all intervals. Similar ideas were previously used  in the wave equation setting \cite{BC99,Ta00} and in the Schr\"odinger equation setting \cite{BGT04}. We also refer to a short-time Strichartz estimate in the periodic setting, established by Kenig, Ponce, and Vega \cite{KPV}, which will be useful in this work. 

In a different direction, Tao introduced in \cite{Tao04} a complex version of the Cole-Hopf transformation, localized in frequency, to cancel the problematic high-low frequency interactions in the nonlinearity of the BO equation.  This \emph{gauge} transform, combined with Strichartz estimates, enabled him to prove global well-posedness for BO in $H^1(\mathbb R)$. Building on this idea and using the Fourier restriction norm method, Ionescu and Kenig \cite{IK07}, and independently Molinet \cite{Mo07,Mo08}, proved global well-posedness in $L^2$ on the real line and on the torus, respectively (see also Burq and Planchon \cite{BuPl08} for a result in  $H^s(\mathbb R)$, $s>\frac14$). We also refer to the works \cite{MoPi12, IfTa19} for alternative proofs of global well-posedness in $L^2$. The approach of Molinet and Pilod \cite{MoPi12}, which still relies on the Fourier restriction norm method, applies both on the real line and the torus, and also establishes unconditional uniqueness (see \cite{Kishi22,MoPi23} for further results in this direction). In contrast, the proof by Ifrim and Tataru \cite{IfTa19} does not use the Fourier restriction norm method and additionally provides long-time asymptotics for small data. 
Recently, by exploiting the complete integrability, the global well-posedness threshold for the BO equation was lowered to $H^s$, $s>-\frac12$, first on the torus by G\'erard, Kappeler and Topalov in \cite{GKT23}, and then on the line by Killip, Laurens and Visan \cite{KiLaVi24}.

Although the gauge transform is somewhat related to the integrability of the equation (see \cite{GKT24}), the method appears more robust and adaptable to dispersive equations beyond the BO equation. We refer to \cite{Ifsa25,CLOP24} for well-posedness results in $L^2$ on $\mathbb R$ and $\mathbb T$ for the Intermediate Long Wave equation relying on the BO gauge, (see also \cite{GaLa25} for an improvement to $H^{-\frac12}(\mathbb T)$ using the integrable structure of BO). By using a paradifferential version of the gauge Herr, Ionescu, Kenig and Koch \cite{HIKK10} proved global well-posedness in $L^2(\mathbb R)$ for the IVP associated with \eqref{fKdVbis}, for $1<\alpha<2$. This result was extended recently to negative regularities by Ai and Liu \cite{AiLiu24} who employed a pseudo-differential version of the gauge transform. At present, it remains an open question whether these techniques can be extended to lower dispersions  $0<\alpha<1$. Nevertheless, we would like to mention the recent work of Said \cite{Said23} in this direction, where a paralinearized version of \eqref{fKdVbis} is considered. 

Other types of energy methods have been employed to study the IVP for \eqref{fKdV} in the case $1<\alpha<2$. In \cite{Guo12}, Guo used the short-time Fourier restriction norm method, in the framework introduced by Ionescu, Kenig, Tataru for KPI in \cite{IKT08}, to prove local well-posedness in $H^s(\mathbb R)$, $s>2-\alpha$. This result was improved by Molinet and Vento \cite{MV}, who established unconditional well-posedness in $H^s$, $s>1-\frac{\alpha}2$ both on the line and on the torus, by introducing a new energy method for strongly resonant dispersive equations.  More recently, Schippa \cite{Schip20} combined the short-time Fourier restriction norm method with modified energies to further lower the threshold in the periodic setting, proving well-posedness in $L^2$ when $\frac32 \le \alpha \le 2$ and in $H^s$ for $s>\frac32-\alpha$ when $1 \le \alpha \le \frac32$. 

Only few results are currently available in the literature for \eqref{fKdVbis} in the case $0<\alpha<1$. On the line, Linares, Pilod and Saut \cite{LiPiSa14} established local well-posedness in $H^s(\mathbb R)$ for $s>\frac32-\frac{3\alpha}8$ by using the refined Strichartz estimate as in \cite{KoTz03,KK03}. This result was subsequently improved by the authors of the present paper in \cite{MPV} to $s>\frac32-\frac{5\alpha}4$, which in turn yields global well-posedness in the energy space $H^{\frac{\alpha}2}(\mathbb R)$ for $\alpha>\frac67$ by relying on the conserved quantities in \eqref{MassEnergy}. The proof in \cite{MPV} combines the energy method in \cite{MV} with modified energies. On the torus, J\"ockel proved recently in \cite{Joc24} local well-posedness in $H^s(\mathbb T)$ for $s>\frac32-\alpha$ by combining the short time restriction norm method with modified energies. In this work, we further improve these results, both on the line and on the torus. 

Finally, we would like to mention some interesting related results concerning the small data, long-time existence problem for fKdV equations. In the case of the Burgers-Hilbert equation (corresponding to $\alpha=-1$ in \eqref{fKdVbis}), Hunter and Ifrim~\cite{HI}, as well as Hunter, Ifrim, Tataru, and Wong~\cite{HITW}, proved that solutions with initial data of size $\epsilon \ll 1$ exist on a time interval of length at least $\epsilon^{-2}$. These results were later extended by Ehrnstr\"om and Wang to the fKdV equations \eqref{fKdVbis} for $-1<\alpha<1$, $\alpha \neq 0$ in~\cite{EhWa19}, and to more general Whitham-type symbols in~\cite{EhWa22}. The proofs of all these results rely on modified energies and normal form transformations.

\subsection{Statements of the results}

Here, we work on $\K\in\{\T, \R\}$. As explained above, we focus on the case $0<\alpha\le 1$. Our regularity threshold $s_{\alpha}$ is given by 
\begin{equation}\label{cond-s}
s_\alpha=\begin{cases}
 1-\frac{3\alpha}4 & \text{for} \quad \frac23 \le \alpha \le 1; \\ 
  \frac 32(1-\alpha) &  \text{for}  \quad \frac13 \le \alpha \le \frac23; \\ 
 \frac 32-\frac{\alpha}{1-\alpha}  & \text{for} \quad 0 < \alpha  \le \frac13 .
 \end{cases}
\end{equation}
We also define the regularity threshold $\widetilde{s}_\alpha$ for the \emph{a priori} estimate  by
\begin{equation}\label{cond-stilde}
\widetilde{s}_\alpha=\begin{cases}
 \frac 12-\frac \alpha 4 & \text{for} \quad \frac23 \le \alpha \le 1; \\ 
 1-\alpha &  \text{for}  \quad \frac12 \le \alpha \le \frac23; \\ 
 \frac 32-\frac{\alpha}{1-\alpha}  & \text{for} \quad 0 < \alpha  \le \frac12 .
 \end{cases}
\end{equation}
Observe in particular that $s_{\alpha} \ge \widetilde{s}_\alpha$.

Moreover, we assume that the symbol $\omega_{\alpha+1}$ of the Fourier multiplier $L_{\alpha+1}$ satisfies the following assumption.
\begin{hypothesis}\label{hyp1}
$\omega_{\alpha+1}\in C^{1}(\R)\cap C^3(\R\backslash\{0\})$ is a real-valued odd function satisfying: there exists $\xi_0>0$ such that, for all $|\xi| \ge \xi_0$ and $\beta\in\{0,1,2,3\}$,
\begin{equation} \label{hyp1.1}
\left| \partial^\beta\omega_{\alpha+1}(\xi) \right| \sim |\xi|^{\alpha+1-\beta}
\end{equation} 
\end{hypothesis}

\begin{remark} \label{Rem:symb} It follows from the mean value theorem and the fact that $\omega_{\alpha+1}(0)=0$ ($\omega_{\alpha+1}$ is odd) that the symbol $\omega_{\alpha+1}$ also satisfies, for $|\xi| \le \xi_0$, 
\begin{equation} \label{hyp1.2}
\left| \omega_{\alpha+1}(\xi) \right| \lesssim |\xi| .
\end{equation} 
\end{remark}

\begin{remark}The following symbols satisfy Hypothesis \ref{hyp1}:
\begin{enumerate}
\item for $0<\alpha \le 1$: the homogeneous symbol $\omega_{\alpha+1}(\xi)=|\xi|^{\alpha}\xi$  corresponding to the dispersive operator $D^{\alpha}\partial_x$ in the fKdV equation \eqref{fKdVbis}; 
\item for $\alpha=\frac12$: the Whitham symbol $\omega_W(\xi)= \left( \frac{\tanh (\xi)}{\xi}\right)^{\frac12} (1+\tau \xi^2)^{\frac12}$, with $\tau>0$ corresponding to the Whitham equation with surface tension \eqref{Whitham}; 
\item for $\alpha=1$: the symbol $\omega_{ILW}(\xi)=\xi^2\coth(\delta \xi)$ corresponding to the Intermediate Long Wave equation (ILW). 
\item for $\alpha=1$: the symbol $\omega_S(\xi)=\left(\left(1+\xi^2 \right)^{\frac12}-1\right)\xi$ corresponding to the Smith equation (see \cite{Smith}).
\end{enumerate}
\end{remark}

Next we state our main result. 
\begin{theorem}[Local well-posedness]\label{th:main}
  Let $\K\in\{\T, \R\}$, $\alpha\in (0,1]$, $s>s_\alpha$ and $s'>\widetilde{s}_\alpha$ with $s' \le s$. Assume that $L_{\alpha+1}$ satisfies Hypothesis \ref{hyp1}. There exist positive constants $A_0$ and $\beta_0$ such that the following is true. For any $u_0\in H^s(\K)$, let  
  \begin{equation*}
  T:=A_0 (1+\|u_0\|_{H^{s'}})^{-\beta_0} .
  \end{equation*} 
  Then, there exists a unique solution $u$ of the IVP \eqref{fKdV} on the time interval $[0, T]$ satisfying
\begin{equation} \label{th:main.1}
  u\in C([0, T] : H^s(\K))\cap L^2\big((0,T) : W^{(s-(\frac12-\frac \alpha 4))_-,\infty}(\K) \big).
 \end{equation}
 Moreover, for any $0<\widetilde{T}<T$, there exists a neighbourhood $\mathcal{U}$ of $u_0 \in H^s(\mathbb K)$ such that the flow-map data-solution $v_0\mapsto v$ is continuous from $\mathcal{U}$ into $C([0, \widetilde{T}] : H^s(\K))$.
\end{theorem}

\begin{remark}
In the case $\alpha=1$ including the Benjamin-Ono and Intermediate Long Wave equations, we obtain local well-posedness in $H^s(\K)$, $s>\frac14$. This result is similar to our earlier work in \cite{MPV} on the real line and, to our knowledge, is new on the torus without relying on a gauge transformation.  It remains an open question whether the threshold $s=\frac14$ can be lowered for these equations without using a gauge transform. 
\end{remark}

\begin{remark} In the case $0<\alpha<\frac15$, the result is less sharp than in our previous work \cite{MPV} where we worked at the regularity $s>\frac32-\frac{5\alpha}4$. It is likely that this could be improved at the cost of introducing additional technical complications. However, for the sake of simplicity and clarity of presentation, we choose to not pursue these refinements here.
\end{remark}

\begin{corollary}[Unconditional uniqueness] \label{coro:uu}
Let $\K\in\{\T, \R\}$, $\alpha\in (0,1]$ and $s>\max\{\frac12,s_\alpha\}$.  Assume that $L_{\alpha+1}$ satisfies Hypothesis \ref{hyp1}. Let $u_0 \in H^s$. Then there exists a unique solution to \eqref{fKdV} in $C([0,T] : H^s(\mathbb K))$, where $T$ is the existence time in Theorem \ref{th:main}.
\end{corollary}

The quantities \eqref{Mass} and
\begin{equation} \label{Energy-general}
E(u)=\frac12 \int_{\mathbb K} (\Lambda^{\frac{\alpha}2}u)^2(t,x)dx-\frac13 \int_{\mathbb K} u^3(t,x) dx ,
\end{equation}
where $\Lambda^{\frac{\alpha}2}$ is the Fourier multiplier of symbol $\sqrt{\frac{|\omega_{\alpha+1}(\xi)|}{|\xi|}}$, are conserved by the flow of \eqref{fKdV}. Note that the first integral in $E(u)$ is well-defined in low frequency thanks to Remark \ref{Rem:symb}.  By making use of these conservation laws, we can iterate Theorem \ref{th:main} and thereby obtain global well-posedness in $H^s$ for $\alpha>\frac23$ and in the energy space $H^{\frac{\alpha}2}$ for $\alpha > \frac{4}{5}$.

\begin{theorem}[Global well-posedness] \label{coro:gwp}
Let $\K\in\{\T, \R\}$. Assume that $L_{\alpha+1}$ satisfies Hypothesis \ref{hyp1}. 
\begin{itemize}
\item[(i)]  Let $\alpha \in (\frac23,1]$. Then for any $s>s_{\alpha}$, the IVP \eqref{fKdV} is globally well-posed in  $H^s(\mathbb K)$.
\item[(ii)] Let $\alpha\in (\frac45,1]$. Then the IVP \eqref{fKdV} is globally well-posed in the energy space $H^{\frac{\alpha}2}(\mathbb K)$.\end{itemize}
\end{theorem}

\begin{remark} 
The global well-posedness result in Theorem \ref{coro:gwp} $(i)$ follows from the fact that the \emph{a priori} estimate holds at a lower regularity threshold $\widetilde{s}_{\alpha}$ than the local well-posedness threshold $s_{\alpha}$  (see Propositions \ref{prop:Es}). This kind of mechanism, where global control is achieved below the local well-posedness threshold, has been recently exploited by Ifrim and Tataru for general cubic defocusing dispersive equations \cite{IfTa23,IfTa24}.\end{remark}

\begin{remark} As a consequence of Theorem \ref{coro:gwp}, the IVP associated with \eqref{fKdVbis} is globally well-posed in $H^{\frac{\alpha}{2}}(\mathbb R)$ for $\alpha \in (\frac{4}{5}, 1]$. This result completes the proofs of both the orbital stability of solitary waves in \cite{LiPiSa15,Ar16}, and the construction of solutions asymptotically behaving as a finite sum of solitary waves in \cite{Ey23}, for the fKdV equation in the same range of $\alpha$, as both results rely on the global well-posedness in the energy space.
\end{remark}

\subsection{Strategy of the proof} 

The proof of Theorem \ref{th:main} is an improvement of the strategy developed in \cite{MPV}, which combines energy methods for strongly resonant symbols introduced by the first and third authors in \cite{MV}, refined Strichartz estimates, and the use of a modified energy. The main new ingredient is that, in this case, we fully symmetrize  the symbol of the modified energy, both for the \emph{a priori} estimate and for the difference. 

The resolution space $Y^s_T$  consists of functions in $C([0,T] : H^s)$ equipped with both the energy norm $\| u\|_{L^\infty_TH^s_x}$ and the Strichartz norm $\| u\|_{L^2_TL^{\infty}_x}$. Recall that the Strichartz estimate for fKdV equation 
\[ \|e^{tD^{\alpha}\partial_x}u_0\|_{L^4_tL^{\infty}_x} \le c \|D^{\frac{1-\alpha}4}u_0\|_{L^2} \] 
involves a loss of $(1-\alpha)/4$ derivatives. However, by chopping the time interval, applying this estimate on each subinterval and resumming over the whole interval,  as explained above, we can control the Strichartz norm $\| u\|_{L^2_TL^{\infty}_x}$ of a solution at the regularity level $H^s$, for any $s>\frac12-\frac{\alpha}4$, (see Proposition \ref{propse} for a precise statement where we even recover $s-(\frac12-\frac{\alpha}4)$ derivatives when working at the regularity $s>\frac12-\frac{\alpha}4$). 
Moreover, the space $Y^s_T$ is embedded in the Bourgain space $X^{s-1,1}_T$. Indeed, the $X^{s-1,1}_T$-norm can be controlled by applying the classical linear $X^{s,b}$ estimates on the Duhamel formulation associated to \eqref{fKdV}, followed by the fractional Leibniz rule, yielding
  \[\|u\|_{X^{s-1,1}_T} \lesssim \|u_0\|_{H^{s-1}} +\|J_x^{s-1}\partial_x(u^2)\|_{L^2_{x,T}} \lesssim \|u_0\|_{H^s}+\|u\|_{L^2_TL^{\infty}_x}\|u\|_{L^{\infty}_TH^s_x} .\]

As in \cite{MPV}, the resolution space $Z^{\sigma}_T$ for the difference $w$ of two solutions $u$ and $v$ is more involved. First, $w$ can only be controlled at a negative regularity $\sigma$, with  
\[-\frac 1 2 +\frac \alpha 4< \sigma<\min \left\{\frac{\alpha-1}2,s-\frac 32+ \alpha\right\}<0\] and with a weight $|\xi|^{-1}$ applied on the low frequencies. Accordingly, we use the energy norm $\|w\|_{L^{\infty}_T\overline{H^\sigma_x}}$ and the Strichartz norm $\|u\|_{L^2_T\overline{L^{\infty}_x}}$ where the overline indicates the addition of the low-frequency weight. On the other hand, since we are working at negative regularity, it is not possible to close the estimate on the Bourgain norm by applying the fractional Leibniz rule as before. For this reason, we replace the Bourgain space $X^{\sigma-1,1}$  by the \emph{sum space}  $F^{\sigma,\frac12}=X^{\sigma-1,1}+X^{\sigma,\frac12}$
 equipped with  the norm
 \[\|u\|_{F^{\sigma,\frac12}} = \inf \big\{ \|u_1\|_{X^{\sigma-1,1}} + \|u_2\|_{X^{\sigma,(1/2)_+}} \, : \, u=u_1+u_2\big\}. \]
The use of a sum space was introduced by Tao \cite{Tao07} in the study of the quartic gKdV equation, and later applied in \cite{MV, MPV} in this setting.

Therefore, the main part of the paper is devoted to controlling the energy norms, both for a solution $u$ and for the difference $w=u-v$ of two solutions, at the corresponding levels of regularity. To achieve this, we employ modified energies. The main new idea is to fully symmetrise the symbol of these modified energies. In order to preserve as much symmetry as possible, we choose to perform the dyadic frequency decomposition on the functions only after the symmetrisation and cancellation process. 

We illustrate this approach on the energy associated with the difference of two solutions $w=u-v$, which satisfies the equation
\begin{equation} \label{fKdV:diff}
\partial_tw-L_{\alpha+1}w = \partial_x(zw) ,
\end{equation} 
where $z=u+v$. This equation is less symmetric than \eqref{fKdV}, making the analysis more delicate. The modified energy for $w$ is then defined by 

\[ \widetilde{E}^\sigma(t) = \frac 12\||\xi|^{-1}\widehat{w}(t,\xi)\|_{L^2(|\xi|\le 1)}^2 + \frac 12\||\xi|^{\sigma}\widehat{w}(t,\xi)\|_{L^2(|\xi| \ge 1)}^2+ \widetilde{\mathcal{E}}^\sigma(t) \]
where
\[ \widetilde{\mathcal{E}}^\sigma(t) = \int_{\Gamma^3} \widetilde{b}_3(\xi_1,\xi_2) \widehat{w}(t,\xi_1) \widehat{w}(t,\xi_2)\widehat{z}(t,\xi_3)d\xi_1d\xi_2 . \]
Here $\Gamma^3$ denotes the hyperplane of $\mathbb R^3$: $\xi_1+\xi_2+\xi_3=0$, 
\[ \widetilde{b}_3(\xi_1,\xi_2) = \frac{ (\widetilde{\nu}_\sigma(\xi_1) + \widetilde{\nu}_\sigma(\xi_2)){\bf 1}_{|\xi_1+\xi_2|>N_0}}{\Omega_3(\xi_1,\xi_2)}, \ \text{with} \ \widetilde{\nu}_\sigma(\xi)=\xi|\xi|^{2\sigma} {\bf 1}_{|\xi|\ge 1}, \]
and $\Omega_3(\xi_1,\xi_2)$ is the quadratic resonant function defined by 
\[ \Omega_3(\xi_1,\xi_2)=\omega_{\alpha+1}(\xi_1+\xi_2)-\omega_{\alpha+1}(\xi_1)-\omega_{\alpha+1}(\xi_2). \]
It is not difficult to verify that $\widetilde{\mathcal{E}}^\sigma$ is coercive under suitable condition on $s$ and $N_0$ (see Proposition \ref{coer:E_tilde_sigma}). To obtain the energy estimate in Proposition \ref{prop:E_tilde_sigma}, we differentiate $\widetilde{\mathcal{E}}^\sigma(t)$ with respect to time, use the equations \eqref{fKdV} and \eqref{fKdV:diff} and then integrate again from $0$ to $t$. 

The contribution from the quadratic low-frequency part $\||\xi|^{-1}\widehat{w}(t,\xi)\|_{L^2(|\xi|\le 1)}^2$ is readily controlled under the condition $-\sigma<s$. . Meanwhile, the contribution from the quadratic high-frequency part $\||\xi|^{\sigma}\widehat{w}(t',\xi)\|_{L^2(|\xi| \ge 1)}^2$ cancels out with the term arising from the cubic correction $\widetilde{\mathcal{E}}^\sigma(t)$ when inserting the linear contributions of the equations  \eqref{fKdV} and \eqref{fKdV:diff} into the time derivatives. It then remains to control the additional quartic terms produced by substituting the nonlinear terms from \eqref{fKdV} and \eqref{fKdV:diff} into the time derivatives within the cubic correction $\widetilde{\mathcal{E}}^\sigma(t)$.

To do so, we fully symmetrize this term. The most delicate contribution to control is 
\[ \widetilde{\mathcal{K}}(t) = \frac i2\int_0^t\int_{\Gamma^4} \widetilde{a}_4(\xi_1, \xi_2, \xi_3)  \widehat{w}(\xi_1) \widehat{w}(\xi_2) \widehat{z}(\xi_3) \widehat{z}(\xi_4)d\xi_1d\xi_2d\xi_3dt' , \]
where the symbol $\widetilde{a}_4(\xi_1, \xi_2, \xi_3)$ is given by 
\begin{align*}
\widetilde{a}_4(\xi_1, \xi_2, \xi_3) &= (\xi_1+\xi_3)\left[\widetilde{b_3}(\xi_1+\xi_3, \xi_2) - \widetilde{b_3}(\xi_1+\xi_3, -\xi_1)\right] \nonumber \\
&\quad + (\xi_2+\xi_3)\left[\widetilde{b_3}(\xi_2+\xi_3, \xi_1) - \widetilde{b_3}(\xi_2+\xi_3, -\xi_2)\right] \nonumber \\
&\quad - (\xi_1+\xi_2)\widetilde{b_3}(\xi_1, \xi_2) \nonumber 
\end{align*}
We then perform a spatial dyadic decomposition on all the functions involved and carefully analyze the cancellations of $\widetilde{a}_4$  in the different frequency regions: non-resonant, mildly resonant, and resonant. This analysis, carried out in detail in Appendix~\ref{app:symbols}, relies only on elementary calculus arguments. It is worth noting, however, that certain second-order cancellations are exploited in specific regions. The proof of the energy estimate is then completed by combining these symbol estimates with appropriate bounds in the resolution space norms. 

A key point in this analysis is the cubic resonance function
 \[\Omega_4(\xi_1,\xi_2,\xi_3)= \omega_{\alpha+1}(\xi_1+\xi_2+\xi_3)-\omega_{\alpha+1}(\xi_1)-\omega_{\alpha+1}(\xi_2)-\omega_{\alpha+1}(\xi_3)\]
 which factorises as $|\Omega_4| \sim M_{min}M_{med}N_{max}^{\alpha-1}$, using the notation $|\xi_i|\sim N_i$, $M_1\sim|\xi_2+\xi_3|$, $M_2\sim |\xi_1+\xi_3|$, $M_3\sim |\xi_1+\xi_2|$. Both the fully resonant case $N_1 \sim N_2 \sim N_3 \sim N_4=N$ and $M_{min} \sim M_{med} \ll N$ (corresponding to Region A in our analysis), and the mild resonant cases $N_{min} \ll N_{thd} \sim N_{sub} \sim N_{max}$, (corresponding to Region C in our analysis) yield the threshold $s>1-\frac{3\alpha}4$ (when $\frac23 \le \alpha \le 1$).
 
 Finally, with these estimates in hand, we outline the proof Theorem \ref{th:main}. First, we establish a Lipschitz bound at the negative threshold $\sigma$ on the difference of two solutions emanating from initial data at regularity $s>s_{\alpha}$  with the same low-frequency part. This directly yields uniqueness. Next, given an initial datum  $u_0 \in H^s$, $s>s_\alpha$, we consider the sequence of smooth solutions $u_n$ emanating from the truncated data $\Pi_{\le 2^n} u_0$ with frequencies lower than $2^n$. Using the same estimates for smooth solutions (Proposition \ref{prop:Es}), we derive \emph{a priori} $H^s$ estimates on $[0,T]$, where $T = T(\|u_0\|_{H^{s'}})$ for any $s \ge s' > \widetilde{s}_\alpha$. Moreover, we also use a smoothing effect  for these solutions at large frequency $N_0$ (see Corollary \ref{coro:Eshigh}) to conclude that $\{u_n\}$ is a Cauchy sequence in $C([0,T'] : H^s)$ and $T'=T'(\|u_0\|_{H^{s}})$. This yields the existence of a solution $u \in C([0,T] : H^s)$ by iterating the argument a finite number of times as the \emph{a priori} estimate hold on $[0,T]$.  Instead of using the Bona-Smith argument \cite{bonasmith1975} or the frequency envelope method (see \emph{e.g.} \cite{ABITZ24} and the references therein), we rely again on our smoothing effect described above to prove the continuity of the flow map.
   
Throughout the paper, we focus on the real line case. A section addressing the periodic case is included at the end of the paper, where the necessary modifications are discussed. The main new difficulty in the periodic setting is to establish a refined Strichartz estimate analogous to the one available on the real line. To this end, we rely on a crucial Strichartz estimate on short time intervals, originally derived by Kenig, Ponce, and Vega for \eqref{fKdVbis} when $\alpha \in \N_{>1}$ (see Proposition 2.2 in \cite{KPV}), which we adapt to our setting. The remainder of the proof proceeds similarly to the real line case. 
\bigskip
\bigskip

 The paper is organized as follows. In Section~\ref{Sec:2}, we introduce the necessary notation, define the function spaces, and recall the Strichartz and $X^{s,b}$-estimates established in \cite{MPV}. Section~\ref{Sec:3} is devoted to proving the $L^2$-trilinear estimates, which are then applied in Sections~\ref{Sec:4} and~\ref{Sec:5} to derive the energy estimates for solutions and for the difference of two solutions, respectively. In Section~\ref{Sec:6}, we provide the proof of Theorems~\ref{th:main} and~\ref{coro:gwp} and Corollary~\ref{coro:uu}, and we explain the main changes needed to address the periodic case in Section~\ref{Sec:7}. Finally, the symbol estimates required for the analysis are established in Appendix~\ref{app:symbols}.

\section{Notation, function spaces and preliminary estimates} \label{Sec:2}
\subsection{Notation} \label{notation}
\begin{itemize}

\item For any positive numbers $a$ and $b$, the notation $a\lesssim b$ means that there exists a positive constant $C$ such that $a\le Cb$, and we denote $a\sim b$ when $a\lesssim b$ and $b\lesssim a$. We also write $a\ll b$ if the estimate $b\lesssim a$ does not hold. 

\item If $x\in\R$, $x_+$, respectively $x_-$ will denote a number slightly greater, respectively lesser, than $x$. We also set $\cro{x} = (1+x^2)^{\frac12}$.

\item For $u=u(x,t)\in \mathcal{S}'(\R^2)$, $\F u=\hat{u}$ will denote its space-time Fourier transform and inverse Fourier transform, whereas $\F_xu$, respectively $\F_tu$ will denote its Fourier transform in space, respectively in time. When  $u$ depends only on $x$ or $t$ and there is no risk of confusion, we simply denote $\F_tu$ or $\F_x u$ by $\F u=\hat{u}$. Similarly, we denote by $\F^{-1} u=(u)^{\vee}$, $\F_t^{-1}$ and $\F^{-1}_x$ the inverse Fourier transform of $u$.

\item We denote by $U_{\alpha}(t)=e^{tL_{\alpha+1}}$ the unitary group associated to the Fourier multiplier $L_{\alpha+1}$. More precisely, we have 
\[ \left( U_{\alpha}(t)f \right)^{\wedge_x}(\xi)=e^{it \omega_{\alpha+1}(\xi) } \widehat{f}(\xi) . \]

\item For $s\in\R$, we define the classical Bessel and Riesz potentials of order $-s$, $J^s_x$ and $D^x_x$, by
$$
J^s_xu = \F^{-1}_x(\cro{\xi}^s\F_xu) \textrm{ and } D^s_xu = \F^{-1}_x(|\xi|^s \F_xu).
$$
We also define a Bessel potential with a weight in low frequency $\overline{J^s_x}$ by 
$$
\overline{J^s_x}u = \F^{-1}_x(\cro{|\xi|^{-1}}\cro{\xi}^s\F_xu) .
$$

\item Throughout the paper, we fix a smooth cutoff function $\eta$ such that
\begin{equation}\label{eta}
\eta \in C_0^{\infty}(\mathbb R), \quad 0 \le \eta \le 1, \quad
\eta_{|_{[-1,1]}}=1 \quad \mbox{and} \quad  \mbox{supp}(\eta)
\subset [-2,2].
\end{equation}
We set  $ \phi(\xi):=\eta(\xi)-\eta(2\xi)$. Let  $\widetilde{\phi} \in C_0^{\infty}(\mathbb R)$ be such that $\widetilde{\phi}_{|_{\pm [\frac12,2]}} \equiv 1$ and $\text{supp} \, (\widetilde{\phi}) \subset \pm [\frac14,4]$. For $l \in \mathbb Z$, we define
$$
\phi_{2^l}(\xi):=\phi(2^{-l}\xi), \quad \ \widetilde{\phi}_{2^l}(\xi)=\phi_{\sim 2^l}(\xi):=\widetilde{\phi}(2^{-l}\xi) \, ,
$$
and, for $ l\in \N^* $,
$$
\psi_{2^{l}}(\xi,\tau)=\phi_{2^{l}}(\tau-\omega_{\alpha+1}(\xi)).
$$
By convention, we also denote
$$
\psi_{{\le 1}}(\xi,\tau):=\eta(2(\tau-\omega_{\alpha+1}(\xi))).
$$
\item Any summations over capitalized variables such as $N$ or $L$ are presumed to be dyadic. Unless stated otherwise, we work with homogeneous dyadic decomposition for the space frequency variables and non-homogeneous decompositions for modulation variables, i.e. these variables range over numbers of the form $\mathbb D:=\{2^k : k\in\mathbb Z\}$ and $\mathbb D_{\ge 1}:=\{2^k : k\in\mathbb N^{\star}\}$ respectively.  Then, we have that
$$
\sum_{N \in \mathbb D}\phi_N(\xi)=1\quad \forall \xi\in \R^*, \quad \mbox{supp} \, (\phi_N) \subset
\{\frac{N}{2}\le |\xi| \le 2N\}, \ N \in \mathbb D,
$$
and
$$
\psi_{\le 1}(\xi,\tau)+\sum_{L\in \mathbb D_{\ge 1}}\psi_L(\xi,\tau)=1 \quad \forall (\xi,\tau)\in\R^2.
$$

\item Let us define the Littlewood-Paley multipliers by
$$
P_Nu=\mathcal{F}^{-1}_x\big(\phi_N\mathcal{F}_xu\big), \quad P_{ \sim N}u=\mathcal{F}^{-1}_x\big(\widetilde{\phi}_N\mathcal{F}_xu\big) \quad
Q_Lu=\mathcal{F}^{-1}\big(\psi_L\mathcal{F}u\big),
$$ 
 $$P_{\ge N}:=\sum_{K \ge N} P_{K}, \ P_{\gtrsim N}:=\sum_{K \ge N} P_{\sim K}, \ P_{\le N}:=\sum_{K \le N} P_{K}, P_{\lesssim N}:=\sum_{K \le N} P_{\sim K}, $$ 
 and 
 $$Q_{\ge L}:=\sum_{K \ge L} Q_{K}, \    Q_{\le L}:=\sum_{K \le L} Q_{K}. $$ 
 For the sake of brevity we often write $u_N=P_Nu$, $u_{\le N}=P_{\le N}u$, $\cdots$
 
 \item We also define the straight cuts in frequency  $\Pi_{ \le N}u$ and $\Pi_{>N}u$ by 
 $$
\Pi_{ \le N}u=\mathcal{F}^{-1}_x\big({\bf 1}_{|\xi| \le N}\mathcal{F}_xu\big), \quad \text{and} \quad \Pi_{>N}u=\mathcal{F}^{-1}_x\big({\bf 1}_{|\xi| > N}\mathcal{F}_xu\big).
$$

\item 
If $N_1$, $N_2 $ are two dyadic numbers, we denote $N_1 \vee N_2 =\max\{N_1,N_2\}$ and $N_1 \wedge N_2 =\min\{N_1,N_2\}$. Moreover, if $N_1, \, N_2, \, N_3$ are dyadic numbers, we denote by $N_{max} \ge N_{med} \ge N_{min}$ the maximum, sub-maximum and minimum of $\{N_1,N_2,N_3\}$. Finally, if $N_1, \, N_2, \, N_3, \, N_4 \in \mathbb D$, we denote by 
$N_{max} \ge  N_{sub} \ge N_{thd} \ge N_{min}$ the maximum, sub-maximum, third-maximum and minimum of  $\{N_1,N_2,N_3,N_4\}$.
\end{itemize}

\subsection{Function space} \label{FuncSpa}
For $1\le p\le\infty$, $L^p(\R)$ denotes the usual Lebesgue space and for $s\in\R$, $W^{s,p}(\R)$ is the $L^p$-based Sobolev space with norm $\|f\|_{W^{s,p}}=\|J_x^s f\|_{L^p}$, and $H^s(\R)=W^{s,2}(\R)$.
If $B$ is a space of functions on $\R$, $T>0$ and $1\le p\le\infty$, we define the spaces $L^p_TB_x$ and $L^p_tB_x$ by the norms
$$
\|f\|_{L^p_TB_x} = \left\| \|f\|_{B}\right\|_{L^p([0,T])} \ \textrm{ and } \|f\|_{L^p_tB_x} = \left\| \|f\|_B \right\|_{L^p(\R)}.
$$
If $M$ is a normed space of  functions, we will denote $\overline{M}$ its subspace associated with the weighted norm:
$$
\|u\|_{\overline{M}} = \|\F_x^{-1}( \cro{|\xi|^{-1}}\F_xu(\xi))\|_M.
$$

For $s,b\in\R$ we introduce the Bourgain space $X^{s,b}$ associated with the dispersive Burgers' equation as the completion of the Schwartz space $\mathcal{S}(\R^2)$ under the norm
$$
\|u\|_{X^{s,b}} = \|\cro{\xi}^s \cro{\tau-\omega_{\alpha+1}(\xi)}^b \F_{tx}u \|_{L^2}.
$$
We will also work in the sum space $F^{s,b} = X^{s-1,b+\frac12}+X^{s,b_+}$ endowed with the norm
\begin{equation} \label{F}
\|u\|_{F^{s,b}} = \inf \big\{ \|u_1\|_{X^{s-1,b+\frac12}} + \|u_2\|_{X^{s,b_+}} \, : \, u=u_1+u_2\big\}.
\end{equation}
Note in particular that $X^{s-1,b+\frac12} \hookrightarrow F^{s,b}$ is a continuous embedding. 
We will also use the restriction in time versions of these Bourgain spaces. Let $T>0$ be a positive time, we define the restriction spaces $X^{s,b}_T$ and $F^{s,b}_T$ as the spaces of functions $u : \R\times ]0,T[\to\R$ satisfying respectively
\begin{align*}
\|u\|_{X^{s,b}_T} &= \inf \{ \|\widetilde{u}\|_{X^{s,b}} \, : \, \widetilde{u} \, : \, \R\times\R\to\R, \, \widetilde{u}|_{\R\times ]0,T[} = u\} <\infty , \\
\|u\|_{F^{s,b}_T} &= \inf \{ \|\widetilde{u}\|_{F^{s,b}} \, : \, \widetilde{u} \, : \, \R\times\R\to\R, \, \widetilde{u}|_{\R\times ]0,T[} = u\} <\infty .
\end{align*}

For $s\in\R$, we define the spaces $Y^s$ and $Z^s$, respectively by the norms
\begin{equation}\label{Y}
\|u\|_{Y^s} = \|u\|_{L^\infty_t H^s_x} + \|u\|_{X^{s-1,1}} + \|J_x^{(s-(\frac12-\frac \alpha 4))_-}u\|_{L^2_t L^\infty_x}.
\end{equation}
and
$$
\|u\|_{Z^\sigma} = \|u\|_{L^\infty_t H^\sigma_x} + \|u\|_{F^{\sigma,\frac12}} + \|J_x^{(\sigma-(\frac12-\frac \alpha 4))_-}u\|_{L^2_t L^\infty_x}.
$$
Next, we define our resolution space
\[Y^s_T=\left\{ u \in C([0,T] :H^s) \cap L^2((0,T) : W^{(s-(\frac12-\frac{\alpha}4)_-,\infty}) : \|u\|_{Y^s_T} <\infty \right\}, \] where
\begin{equation} \label{def:XT}
\|u\|_{Y^s_T}= \|u\|_{L^{\infty}_TH^s_x}+ \|u\|_{L^2_TW^{s-(\frac12-\frac{\alpha}4)_-,\infty}_x} .
\end{equation}
For convenience, we also introduce the space $\widetilde{Y}^s_T = Y^s_T\cap X^{s-1,1}_T$ equipped with the norm
$$
\|u\|_{\widetilde{Y}^s_T} = \|u\|_{Y^s_T} +\|u\|_{X^{s-1,1}_T}.
$$ 
It is worth noticing that $Y^s_T$ is not the  restriction space related to $ Y^s $. However, in view Proposition \ref{prop:xs-1,1} below, any solution $u $ of \eqref{fKdV} that belongs to $Y^s_T $ also belongs to $\widetilde{Y}^s_T$ and thus according to \eqref{est:rho} possesses an extension in $ Y^s $ with an $ Y^s$-norm that is controlled by its $Y^s_T$-norm. 

We also define the space $Z^\sigma_T$ that will be useful for the difference of $2$ solutions, by its norm 
$$
\|u\|_{Z^\sigma_T} = \|u\|_{L^\infty_T H^\sigma_x} + \|u\|_{F^{\sigma,\frac12}_T} + \|J_x^{(\sigma-(\frac12-\frac \alpha 4))_-}u\|_{L^2_T L^\infty_x}.
$$

Now, we define the extension operator $\rho_T$ which to a function $u$ defined on $[0,T] \times \mathbb R$ maps the function $\rho_T(u)$ defined on $\mathbb R \times \mathbb R$ by 
\begin{equation}\label{defrho}
\rho_T(u)(t):= U_\alpha(t)\eta(t) U_\alpha(-\mu_T(t)) u(\mu_T(t))\; ,
\end{equation}
where $ \eta $ is the smooth cut-off function defined in Section \ref{notation} and $\mu_T $ is the  continuous piecewise affine
 function defined  by
$$
 \mu_T(t)=\left\{\begin{array}{rcl}
 0  &\text{for } &  t<0 \\
 t  &\text {for }& t\in [0,T] \\
  T & \text {for } & t>T
 \end{array}
 \right. .
 $$
 Then, it is proved in Corollary 4.1 in \cite{MPV}  that, for $0<T$ and $s>\frac12-\frac{\alpha}4$,
 \[ \rho_T: C([0,T]:H^s(\R)) \cap X^{s-1,1}_T \cap L^2_TW_x^{(s-(\frac12-\frac{\alpha}4))_-,\infty} \longrightarrow Y^s \]
 is a linear bounded operator. More precisely, there exist $C>0$, independent of $T>0$ and $s>\frac12-\frac{\alpha}4$, such that for any $u \in X^s_T$, 
 \begin{equation} \label{est:rho}
 \|\rho_T(u)\|_{Y^s} \leq C \|u\|_{\widetilde{Y}^s_T}  .
 \end{equation}

\subsection{Strichartz and $X^{s,b}$ estimates}
In this subsection, we recall the crucial estimates on the Strichartz and Bourgain's norms proved in \cite{MPV}.
\begin{proposition}[\cite{MPV}, Proposition 4.1]\label{prop:xs-1,1}
Assume that $0<T\le 1$ and $s\ge 0$. Let $u \in C([0,T] : H^s) \cap L^2((0,T):L^\infty)$ be a solution to \eqref{fKdV}. Then
\begin{equation} \label{be.0}
\|u\|_{X^{s-1,1}_T} \lesssim \|u_0\|_{H^s} + \|u\|_{L^2_TL^\infty_x} \|u\|_{L^\infty_TH^s_x}.
\end{equation}
As a consequence, there exists $\widetilde{C}>0$ independent of $T$ such that
\begin{equation}\label{est:ytilde}
  \|u\|_{\widetilde{Y}^s_T} \le  \widetilde{C}(1+\|u\|_{Y^s_T})\|u\|_{Y^s_T} .
\end{equation}
\end{proposition}

\begin{proposition}[\cite{MPV}, Proposition 4.3]\label{propse}
Let $0<\alpha \le 1$. Assume that $0<T\le 1$ and $s>\frac12-\frac{\alpha}4$. Let $u \in C([0,T] : H^s) \cap L^2((0,T):L^\infty)$ be a solution to \eqref{fKdV}. There exists  $ 0< \kappa_1,\kappa_2<1$ such that  
\begin{equation}\label{se.2}
\|J_x^{(s-(\frac12-\frac \alpha 4))_-}u\|_{L^2_T L^\infty_x}  \le   T^{\kappa_1}\|u\|_{L^\infty_T H^s_x} +T^{\kappa_2}\|u\|_{L^2_TL^\infty_x} \|u\|_{L^\infty_TH^s_x} \; .
\end{equation}
\end{proposition}
\begin{remark}
 It is worth noticing that our critical index $s_\alpha$ is not the same as the one defined in \cite{MPV}, (in \cite{MPV}, $s_\alpha=\frac32-\frac54\alpha$). Then, Proposition \ref{propse} follows directly from Proposition 4.3 in \cite{MPV} by observing that $\frac32-\frac54\alpha-(1-\alpha)=\frac12-\frac{\alpha}4$.
\end{remark}

Propositions \ref{prop:xs-1,1}-\ref{propse} lead to  the following a priori estimate on the $\widetilde{Y}^s_T$-norm  of  solutions to \eqref{fKdV}.
\begin{corollary}[\cite{MPV}, Corollary 4.1]\label{prop:yst}
Let $0<\alpha \le 1$. Assume that $0<T\le 1$ and $s \ge s' >\frac12-\frac{\alpha}4$.  Let $u \in C([0,T] : H^s) \cap L^2((0,T):L^\infty)$ be a solution to \eqref{fKdV}. There exist  $ 0< \kappa<1 $, $a_0>0$ and $ C_0>1 $ such if $ T\le a_0(1+\|u\|_{L^\infty_{T} H^{s'}_x})^{-\frac{1}{\kappa}}$ then 
\begin{equation}\label{est:yst}
\|u\|_{\widetilde{Y}^s_T}   \le C_0   \|u\|_{L^\infty_T H^s_x} \; .
\end{equation}
\end{corollary}

Similarly, we may easily adapt to our framework the needed estimates for the difference of solutions in $Z^\sigma_T$. Let $u,v\in Y^s_T$ be two solutions of \eqref{fKdV} on $[0,T]$. Then the difference $w=u-v$ satisfies
\begin{equation}\label{eqdiff}
  (\partial_t-L_{\alpha+1})w = \partial_x(zw)
\end{equation} 
where $z=u+v$.
\begin{corollary}[\cite{MPV}, Corollary 5.1]\label{prop:zst}
Let $0<\alpha \le 1.$ Assume that $0<T\le 1$, $s>s_\alpha$ and  $ -\frac 1 2 +\frac \alpha 4< \sigma<\min (0,s-\frac 32+ \alpha) $. Let $z\in Y^s_T$ and $w\in \overline{Z}^\sigma_T$ be a solution of \eqref{eqdiff} on $]0,T[$ such that $ w_0\in \overline{H}^\sigma $. Then 
  \begin{equation}\label{est:zst}
    \|w\|_{\overline{Z}^{\sigma}_T} \lesssim (1+\|z\|_{Y^s_T} )^3\Bigl( \|w_0\|_{ \overline{H}^\sigma}+ \|w\|_{L^\infty_T \overline{H}^\sigma}\Bigr) \, .
  \end{equation}
\end{corollary}

\section{$L^2$-trilinear estimates} \label{Sec:3}
We follow the notation in \cite{TaoAJM01}. For any integer $n\ge 1$ and any $\xi\in\R$, the hyperplane $\Gamma^n(\xi)$ of $\mathbb R^n$ is defined by 
$$
\Gamma^n(\xi) = \{(\xi_1,\ldots, \xi_{n})\in \R^{n} : \xi_1+\ldots + \xi_{n}=\xi\}.
$$
When $\xi=0$ we simply denote $\Gamma^n = \Gamma^n(0)$. As usual, if $F$ is a measurable function on $\Gamma^n(\xi)$ we set
$$
\int_{\Gamma^n(\xi)} F(\xi_1,\ldots, \xi_{n}) = \int_{\R^{n-1}} F(\xi_1,\ldots,\xi_{n-1},  \xi - (\xi_1+\ldots+\xi_{n-1})) d\xi_1\ldots d\xi_{n-1}.
$$
Moreover, for fixed $t>0$ we introduce the sets $\R^n_t = [0,t]\times \R^n$ and $\Gamma^n_t = [0,t]\times\Gamma^n$.\\
Now if $\chi$ is a bounded measurable fonction on $\R^n$, we define the multilinear pseudo-product operator $\Pi_\chi$ by its Fourier transform
$$
\mathcal{F}\left(\Pi_\chi(f_1,\ldots, f_n)\right)(\xi) = \int_{\Gamma^{n}(\xi)} \chi(\xi_1,\ldots, \xi_n) \prod_{j=1}^n \widehat{f_j}(\xi_j) .
$$
Sometimes, when there is no risk of confusion, we also denote for $\xi=(\xi_1,\cdots,\xi_n) \in \Gamma^{n}$, 
\[ \chi(\xi_1,\cdots,\xi_{n-1})=\chi(\xi_1,\cdots,\xi_{n-1},-\xi_1-\cdots-\xi_{n-1}) . \]
We easily check that for any permutation $\sigma$ of $\{1,\ldots, n\}$, we have
\begin{equation}\label{Pi-sym}
  \int_\R \Pi_\chi(f_1,\ldots ,f_{n-1}) f_{n} = \int_\R \Pi_{\widetilde{\chi}}(f_{\sigma(1)},\ldots, f_{\sigma(n-1)}) f_{\sigma(n)}
\end{equation}
where $\widetilde{\chi}$ is also bounded and satisfies $\|\widetilde{\chi}\|_{L^\infty} \lesssim \|\chi\|_{L^\infty}$.

We start with a basic $L^2$-trilinear estimate.
\begin{lemma}\label{prod4-est}
 Let $f_j\in L^2(\mathbb R)$, $j=1,...,4$ and $M \in \mathbb D$. Assume that each $f_j$ is supported in $\{|\xi|\sim N_j\}$ for some dyadic numbers $N_j\in \D$. Then it holds that
\begin{equation} \label{prod4-est.1}
\int_{\Gamma^4} \phi_M(\xi_1+\xi_2) \prod_{j=1}^4|f_j(\xi_j)| \lesssim P \prod_{j=1}^4 \|f_j\|_{L^2} \, ,
\end{equation}
where $P=\min\{M, (N_{min}N_{thd})^{\frac 12}\}$. 
\end{lemma}
\begin{proof}
  Estimate \eqref{prod4-est.1} is proved in \cite{MPV2}, Lemma 3.1 with a factor $M$ in the right hand-side. Otherwise, it is a direct consequence of Cauchy-Schwarz inequality.
\end{proof}

It follows from Lemma \ref{prod4-est} that if $\chi$ is a bounded symbol on $\R^3$ and $M_1,M_2,M_3\in \D$ are dyadic numbers, then
\begin{equation}\label{l2tri-basic}
  \left|\int_\R \Pi_{\widetilde{\chi}}(f_1,f_2,f_3)f_4 \right| \lesssim P \|\widetilde{\chi}\|_{L^\infty} \prod_{i=1}^4 \|f_i\|_{L^2}
\end{equation}
where
\begin{equation}\label{def:chitilde}
  \widetilde{\chi}(\xi_1,\xi_2,\xi_3) = \phi_{M_1}(\xi_2+\xi_3)\phi_{M_2}(\xi_1+\xi_3)\phi_{M_3}(\xi_1+\xi_2) \chi(\xi_1,\xi_2,\xi_3)
\end{equation}
and
\begin{equation}\label{def:P}
  P = \min\{M_{min}, (N_{min}N_{thd})^{\frac 12}\} .
\end{equation}
where $M_{min}=\min\{M_1,M_2,M_3\}$ and $N_{min}$, $N_{thd}$ are the third-maximum and the minimum of $\{N_1,N_2,N_3,N_4\}$ as introduced in the notation. 

Below, we recall some important technical lemmas proved in \cite{MV}.

\begin{lemma}[\cite{MV}, Lemma 2.3] \label{QLbound}
Let $L\ge 1$, $1\le p\le\infty$ and $s\in\R$. The operator $Q_{\le L}$ is bounded in $L^p_tH^s_x$ uniformly in $L\ge 1$.
\end{lemma}

For any $T>0$, we consider $1_T$ the characteristic function of the interval $]0,T[$ and use the decomposition
\begin{equation}\label{ind-dec}
1_T = 1_{T,R}^{low}+1_{T,R}^{high},\quad \widehat{1_{T,R}^{low}}(\tau)=\eta(\tau/R)\widehat{1_T}(\tau)
\end{equation}
for some $R>0$ and where $\eta$ is defined in Subsection \ref{notation} and satisfies the additional condition $\int (\eta)^{\vee}d\tau=1$.

\begin{lemma}[\cite{MPV3}, Lemma 3.6]\label{ihigh-lem} For any $ R>0 $ and $ T>0 $ it holds
\begin{equation}\label{high}
\|1_{T,R}^{high}\|_{L^1}\lesssim T\wedge R^{-1}  \textrm{ and } \|1_{T,R}^{high}\|_{L^\infty}\lesssim 1.
\end{equation}
and
\begin{equation}\label{low}
\|1_{T,R}^{low}\|_{L^1} \lesssim T \textrm{ and } \|1_{T,R}^{low}\|_{L^\infty}\lesssim  1.
\end{equation}
\end{lemma}

We are now ready to state our trilinear estimates involving the resolution spaces $Z^0$ and $Y^0$. These estimates will be fundamental to derive the energy estimates in the next sections. 
In the sequel we set 
\begin{equation}\label{defOmega}
\Omega = M_{min}M_{med}N_{max}^{\alpha-1} \; .
\end{equation}
\begin{proposition}\label{prop:l2tri}
  Assume $t\in ]0,1]$ and $u_i\in Z^0$, $i=1,2,3,4$ satisfy $\supp \widehat{u_i}\subset \{|\xi|\sim N_i\}$ for some dyadic numbers $N_1, N_2, N_3, N_4$ such that $N_{max}\gg 1$. Let $\chi\in L^\infty(\R^3)$ and define $\widetilde{\chi}$ as in \eqref{def:chitilde} for some dyadic numbers $M_1,M_2,M_3>0$. If $\Omega \gtrsim 1$, then for any $0\le\theta\le 1$, it holds
  \begin{equation}\label{l2tri.1}
    \left|\int_{\R_t} \Pi_{\widetilde{\chi}}(u_1,u_2,u_3)u_4 dxdt\right| \lesssim P(N_{max}\Omega^{-1})^\theta \|\widetilde{\chi}\|_{L^\infty} \prod_{i=1}^4 \|u_i\|_{Z^0}
  \end{equation}
  where $P$ is defined in \eqref{def:P}.
\end{proposition}
\begin{proof}
  When $\theta = 0$, estimate \eqref{l2tri.1} is directly obtained from \eqref{l2tri-basic}. Thus, by interpolation we may assume $\theta=1$. By \eqref{Pi-sym}, we also may assume by symmetry $N_1\le N_2\le N_3\le N_4$, which enforces $N_3\sim N_4 =: N$. We split the integral $G_t$ in the left hand-side of \eqref{l2tri.1} as $G_t = G_t^{low} + G_t^{high}$ with
  $$
     G_t^{low}(u_1,u_2,u_3,u_4) = \int_{\R^2} 1_{t,R}^{low} \Pi_{\widetilde{\chi}}(u_1,u_2,u_3)u_4 dxdt
  $$
  for $R=N^{-1}\Omega \ll \Omega$ (recall that $N\gg 1$). The high-part is estimated thanks to \eqref{prod4-est.1} and \eqref{high} by
  \begin{equation}
    |G_t^{high}(u_1,u_2,u_3,u_4)| \lesssim PR^{-1} \|\widetilde{\chi}\|_{L^\infty} \prod_{i=1}^4 \|u_i\|_{L^\infty_tL^2_x},
  \end{equation}
  which is acceptable since $PR^{-1} = PN\Omega^{-1}$.
  Next, we decompose $G_t^{low}$ with respect to the modulation variables as follows:
  $$
  G_t^{low}(u_1,u_2,u_3,u_4) = \sum_{L_1,L_2,L_3,L_4} \int_{\R^2} 1_{t,R}^{low} \Pi_{\widetilde{\chi}}(Q_{L_1}u_1,Q_{L_2}u_2,Q_{L_3}u_3)Q_{L_4}u_4 dxdt.
  $$
  Thanks to Proposition \ref{est-Omega4}, combined with the fact that $R\ll \Omega$, this sum vanishes unless $L_{max}\gtrsim \Omega$ where $L_{max}=\max\{L_1,L_2,L_3,L_4\}$, and we are led to evaluate each contribution $G_{t,j}^{low}$, $j=1,2,3,4$ of $G_t^{low}$ to the case $L_{max}=L_j$. By making use of \eqref{low} we bound $G_{t,1}^{low}$ by
  \begin{align*}
    |G_{t,1}^{low}| &\lesssim P \|1_{t,R}^{low}\|_{L^2} \|Q_{\gtrsim \Omega}u_1\|_{L^2_{tx}}  \prod_{i=2}^4 \|u_i\|_{L^\infty_tL^2_x}\\
    &\lesssim P(\cro{N_1}\Omega^{-1}+\Omega^{-(1/2)_+}) \|u_1\|_{F^0} \prod_{i=2}^4 \|u_i\|_{Z^0}.
  \end{align*}
  Observe that since $\Omega=M_{min}M_{med}N^{\alpha-1}\lesssim N^2$ for $\alpha\le 1$, it holds recalling that $\Omega \gtrsim 1$ by hypothesis,
  \begin{equation}\label{l2tri.2}
   \Omega^{-(\frac 12)_+} \lesssim N\Omega^{-1}.
  \end{equation}
  Thus, since $\cro{N_1}\lesssim N$ we get
  $$
  |G_{t,1}^{low}| \lesssim PN\Omega^{-1} \prod_{i=1}^4 \|u_i\|_{Z^0}.
  $$
  To deal with the contribution $G_{t,2}^{low}$ we infer from \eqref{l2tri-basic}, Lemma \ref{QLbound}, \eqref{low} and \eqref{l2tri.2} that
  \begin{align*}
    |G_{t,2}^{low}| &\lesssim P \|1_{t,R}^{low}\|_{L^2} \|Q_{\ll \Omega}u_1\|_{L^\infty_tL^2_x} \|Q_{\gtrsim \Omega}u_2\|_{L^2_{tx}} \|u_3\|_{L^\infty_tL^2_x} \|u_4\|_{L^\infty_tL^2_x}\\
    &\lesssim P(\cro{N_2}\Omega^{-1}+\Omega^{-(\frac 12)_+}) \|u_2\|_{F^0} \prod_{i\neq 2} \|u_i\|_{Z^0}\\
    &\lesssim PN\Omega^{-1} \prod_{i=1}^4 \|u_i\|_{Z^0}.
  \end{align*}
  Finally, by symmetry if suffices to estimate $G_{t,3}^{low}$. We get similarly
  \begin{align*}
    |G_{t,3}^{low}| &\lesssim P \|1_{t,R}^{low}\|_{L^2} \|Q_{\ll \Omega}u_1\|_{L^\infty_tL^2_x} \|Q_{\ll \Omega} u_2\|_{L^\infty_tL^2_x} \|Q_{\gtrsim \Omega}u_3\|_{L^2_{tx}} \|u_4\|_{L^\infty_tL^2_x}\\
    &\lesssim P(N\Omega^{-1}+\Omega^{-(\frac 12)_+}) \|u_2\|_{F^0} \prod_{i\neq 3} \|u_i\|_{Z^0}\\
    &\lesssim PN\Omega^{-1} \prod_{i=1}^4 \|u_i\|_{Z^0}.
  \end{align*}  
\end{proof}
To deal with a particular region in ou energy estimates, we will also need the following result involving the Strichartz part in the $Z^0$ norm.
\begin{proposition}\label{prop:l2tri.2}
 Assume that $t \in ]0,1]$. Let $u_i\in Z^0$, $i=1,2,3,4$ with spatial Fourier support in $\{|\xi_i|\sim N_i\}$ for some $N_i\in\D$ satisfying $N_{min}\ll N_{thd}\ll N_{max}$ and $N_{thd}\gg 1$. Then we have
  \begin{equation}\label{l2tri.3}
    \left|\int_{\R_t} u_1u_2u_3u_4 dxdt\right| \lesssim N_{min}^{\frac 12}N_{thd}^{-\frac \alpha 4} N_{max}^{\frac{1-\alpha}2} \prod_{i=1}^4 \|u_i\|_{Z^0}.
  \end{equation}
\end{proposition}
\begin{proof}
  The proof is similar to the proof of Proposition \ref{prop:l2tri}. By symmetry, it suffices to consider the case $N_1\ll N_2\ll N_3\sim N_4$, so that $\Omega\sim N_2N_4^\alpha\gtrsim 1$. First we decompose the integral $G_t$ in the left hand-side of \eqref{l2tri.3} as
  $$
  G_t = \int_{\R^2}1_{t,R}^{low}u_1u_2u_3u_4dxdt + \int_{\R^2}1_{t,R}^{high}u_1u_2u_3u_4dxdt =: G_t^{low} + G_t^{high}
  $$
  with $R=N_2^{\frac12+\frac{\alpha}4}N_4^{\frac{\alpha-1}2}\ll \Omega$. From Holder inequality and \eqref{high} it holds
  $$
  |G_t^{high}|\lesssim (N_1N_2)^{\frac 12} R^{-1} \prod_{i=1}^4\|u_i\|_{L^\infty_tL^2_x} \lesssim N_1^{\frac 12}N_2^{-\frac \alpha 4} N_4^{\frac{1-\alpha}2} \prod_{i=1}^4 \|u_i\|_{Z^0}.
  $$
  For the low part of $G_t$, we write
  $$
  G_t^{low} = \sum_{L_1,L_2,L_3,L_4  \atop L_{max}\gtrsim N_2N_4^\alpha} \int_{\R^2} 1_{t,R}^{low} Q_{L_1}u_1 Q_{L_2}u_2 Q_{L_3}u_3 Q_{L_4}u_4. 
  $$
  As in the proof of Proposition \ref{prop:l2tri}, let us denote by $G_{t,i}^{low}$ the contribution of the sum over $L_{max}=L_i$ to $G_t^{low}$. Observe that for any function $v \in Z^0$, and for any $N \in \D$, $L \in \D_{\ge 1}$, we have
    \[ \|Q_Lv_N\|_{L^2_{t,x}} \lesssim \|Q_Lu_N\|_{L^2_{t,x}}^{\frac12}(NL^{-1}+L^{-(\frac12)_-})^{\frac12}\|Q_Lu_N\|_{F^0}^{\frac12} \lesssim (NL^{-1}+L^{-(\frac12)_-})^{\frac12} \|v_N\|_{Z^0} . \]
  Then, for $G_{t,1}^{low}$, after resuming over $L_2$, $L_3$, $L_4$, Lemma \ref{ihigh-lem} provides
  \begin{align*}
    |G_{t,1}^{low}| &\lesssim N_1^{\frac 12} \|1_{t,R}^{low}\|_{L^\infty} \|Q_{\gtrsim N_2N_4^\alpha}u_1\|_{L^2_{tx}} \|u_2\|_{L^2_tL^\infty_x} \|u_3\|_{L^\infty_tL^2_x} \|u_4\|_{L^\infty_tL^2_x}\\
    &\lesssim N_1^{\frac 12} (\cro{N_1}(N_2N_4^\alpha)^{-1} + (N_2N_4^\alpha)^{-(\frac 12)_+})^{\frac 12} N_2^{\frac 12-\frac \alpha 4}\prod_{i=1}^4 \|u_i\|_{Z^0}\\
    & \lesssim N_1^{\frac 12}N_2^{-\frac \alpha 4} N_4^{\frac{1-\alpha}2} \prod_{i=1}^4 \|u_i\|_{Z^0},
  \end{align*}
  since we assumed $ N_2=N_{th}\gg 1$.
  In the case $L_{max}=L_2$ we argue similarly and use Lemma \ref{QLbound} to obtain
  \begin{align*}
    |G_{t,2}^{low}| &\lesssim N_1^{\frac 12} \|1_{t,R}^{low}\|_{L^\infty} \|Q_{\ll N_2N_4^\alpha}u_1\|_{L^\infty_tL^2_x} \|Q_{\gtrsim N_2N_4^\alpha}u_2\|_{L^2_{tx}} \|u_3\|_{L^2_tL^\infty_x} \|u_4\|_{L^\infty_tL^2_x}\\
    &\lesssim N_1^{\frac 12} (N_2(N_2N_4^\alpha)^{-1} + (N_2N_4^\alpha)^{-(\frac 12)_+})^{\frac 12} N_4^{\frac 12-\frac \alpha 4} \prod_{i=1}^4 \|u_i\|_{Z^0}\\
    & \lesssim N_1^{\frac 12}N_2^{-\frac \alpha 4} N_4^{\frac{1-\alpha}2} \prod_{i=1}^4 \|u_i\|_{Z^0}.
  \end{align*}
  It remains to consider the contribution of the sum over $L_{max}=L_3$ since the case $L_{max}=L_4$ may be evaluated in the same way. We get 
  \begin{equation}\label{l2tri.4}
  |G_{t,3}^{low}| \lesssim N_1^{\frac 12} N_2^{-\frac 12} N_4^{\frac {1-\alpha}2} \|u_1\|_{L^\infty_tL^2_x} \|Q_{\ll N_2N_4^\alpha}u_2\|_{L^2_tL^\infty_x} \|u_3\|_{F^0} \|u_4\|_{L^\infty_tL^2_x}.
  \end{equation}
  Since it is not clear wether $Q_L$ is bounded or not in $L^2_tL^\infty_x$, we rewrite $Q_{\ll\Omega} = I - Q_{\gtrsim \Omega}$ and obtain thanks to Bernstein inequality that
  \begin{align*}
    \|Q_{\ll N_2N_4^\alpha}u_2\|_{L^2_tL^\infty_x} &\lesssim \|u_2\|_{L^2_tL^\infty_x} + N_2^{\frac 12} \|Q_{\gtrsim N_2N_4^\alpha}u_2\|_{L^2_{tx}} \\
    &\lesssim \left(N_2^{\frac 12-\frac \alpha 4} + N_2^{\frac 12}(N_2(N_2N_4^\alpha)^{-1}+(N_2N_4^\alpha)^{-(\frac 12)_+})\right ) \|u_2\|_{Z^0}\\
    &\lesssim N_2^{\frac 12-\frac \alpha 4} \|u_2\|_{Z^0}.
  \end{align*}
  Inserting this into \eqref{l2tri.4} we conclude
  $$
  |G_{t,3}^{low}| \lesssim N_1^{\frac 12}N_2^{-\frac \alpha 4} N_4^{\frac{1-\alpha}2} \prod_{i=1}^4 \|u_i\|_{Z^0}.
  $$
  
\end{proof}

\section{Energy estimates for solutions} \label{Sec:4}
We provide a priori estimates for solutions of \eqref{fKdV} in $H^s(\R)$ for $s>\widetilde{s}_\alpha$ where $\widetilde{s}_\alpha$ is defined in \eqref{cond-stilde}.

Let $0<T\le 1$, $N_0\gg 1$, and $u\in Y^s_T$ be a solution of \eqref{fKdV} on $[0,T]$. For $t\in (0,T)$ we define the modified energy
\begin{equation}\label{def:Es}
E^s(t) = E^s(t,N_0,u) = \mathcal{E}_2(t,u) + c\mathcal{E}_3(t,N_0,u)
\end{equation}
where $c \in\R$ and $\mathcal{E}_3(t,N_0,u)=\mathcal{E}_3(t)$ will be defined later, and 
$$
\mathcal{E}_2(t) =\mathcal{E}_2(t,u)= \frac 12 \left( \|u(t)\|_{L^2_x}^2 +\|D_x^su(t)\|_{L^2_x}^2 \right).
$$
Differentiating $\mathcal{E}_2(t)$ and integrating on $(0,t)$, we get
\begin{align}
\mathcal{E}_2(t) &= \mathcal{E}_2(0) + \int_{\R_t} D^{2s}_xu \partial_x(u^2) \notag \\
& = \mathcal{E}_2(0) -i \int_{\Gamma^3_t} \xi_1| \xi_1|^{2s} {\bf 1}_{|\xi_1|\le N_0} \widehat{u}(\xi_1) \widehat{u}(\xi_2) \widehat{u}(\xi_3) -i \int_{\Gamma^3_t} \xi_1|\xi_1|^{2s}{\bf 1}_{|\xi_1|\ge N_0} \widehat{u}(\xi_1) \widehat{u}(\xi_2) \widehat{u}(\xi_3) \notag \\
&:= \mathcal{E}_2(0) + \mathcal{I}_0(t) + \mathcal{I}(t). \label{e2tilde}
\end{align}
We easily bound the low frequencies term $\mathcal{I}_0$ thanks to Plancherel and Bernstein inequalities by
\begin{equation}\label{est:I0}
  |\mathcal{I}_0(t)| \lesssim T N_0^{\frac 32}\|u\|_{L^\infty_TL^2_x}\|u\|_{L^\infty_TH^s_x}^2.
\end{equation}
We symmetrize $\mathcal{I}(t)$ to
\begin{equation} \label{def:I}
\mathcal{I}(t) = -\frac i3 \int_{\Gamma^3_t} a_3(\xi_1,\xi_2,\xi_3) \widehat{u}(\xi_1) \widehat{u}(\xi_2) \widehat{u}(\xi_3)
\end{equation}
with
$$
a_3(\xi_1,\xi_2,\xi_3) = \sum_{i=1}^3\xi_i |\xi_i|^{2s}1_{|\xi_i|\ge N_0}.
$$
To remove the contribution of $\mathcal{I}$, we set $c = 1/3$ and
\begin{equation} \label{def:E3}
\mathcal{E}_3(t)= \mathcal{E}_3(t,N_0,u)= \int_{\Gamma^3} b_3(\xi_1,\xi_2,\xi_3) \prod_{i=1}^3 \widehat{u}(\xi_i)
\end{equation}
with a symbol $b_3$ defined on $\Gamma^2$ by
$$
b_3(\xi_1,\xi_2,\xi_3) = \frac{a_3(\xi_1,\xi_2,\xi_3)}{\Omega_3(\xi_1,\xi_2,\xi_3)}.
$$
Note that $b_3$ is symmetric and that we have from \eqref{est-Omega3}, $|b_3(\xi_1,\xi_2,\xi_3)| \sim |\xi_{max}|^{2s-\alpha}$. The following lemma states that our energy is coercive.
\begin{lemma} \label{lemma:coer}
  There exists $C_1>0$ such that for any $\widetilde{s}_\alpha<s \le 3$ and $0<T\le 1$, if $u\in C([0,T] : H^s)$ is a solution of \eqref{fKdV}, then for any $t\in (0,T)$ and $N_0\ge  C_1 (1+\|u\|_{L^\infty_TH^{\widetilde{s}_\alpha}})^{\frac 2 \alpha}$, it holds
  \begin{equation}\label{coesive}
    \left| E^s(t) - \frac 12\|u(t)\|_{H^s}^2\right|\le \frac 14 \|u(t)\|_{H^s}^2.
  \end{equation}
\end{lemma}
\begin{proof}
From the definition in \eqref{def:Es} we infer
$$
\left| E^s(t) - \frac 12\|u(t)\|_{H^s}^2\right|\lesssim |\mathcal{E}_3(t)| \lesssim \sum_{N_1,N_2,N_3 \ge N_0} N_{max}^{2s-\alpha} \int_{\Gamma^3} \prod_{i=1}^3 \left| \widehat{u}_{N_i}(\xi_i) \right|.
$$
It follows then applying Bernstein in space that
$$
\left| E^s(t) - \frac 12\|u(t)\|_{H^s}^2\right| \lesssim (N_0^{-\alpha} + N_0^{\frac 12-\alpha-\widetilde{s}_\alpha} ) \|u\|_{L^\infty_TH^{\widetilde{s}_\alpha}_x} \|u(t)\|_{H^s}^2.
$$
Thus we complete the proof of lemma since $N_0^{-\alpha} + N_0^{\frac 12-\alpha-\widetilde{s}_\alpha} \lesssim N_0^{-\frac{\alpha}2}$.
\end{proof}
\begin{proposition}\label{prop:Es}
  There exists $C_2>0$ and $\gamma>0$ such that the following holds true. Let $\widetilde{s}_\alpha<s'\le s \le 3$, $0<T\le 1$ and $u\in Y^s_T$ be a solution of \eqref{fKdV} on $[0,T]$. Then,  \begin{equation}\label{est:Es}
    \sup_{t\in (0,T)} |E^s(t)| \le |E^s(0)| + C_2TN_0^{\frac 32}\|u\|_{L^\infty_TL^2_x}\|u\|_{L^\infty_TH^s_x}^2 + C_2N_0^{-\gamma} \|u\|_{\widetilde{Y}^{s'}_T}^2\|u\|_{\widetilde{Y}^s_T}^2 .
  \end{equation}
\end{proposition}

\begin{remark}
The index $\widetilde{s}_{\alpha}$ defined in \eqref{cond-stilde} is given by
\[ \widetilde{s}_\alpha = \max\left\{\frac12-\frac{\alpha}4,1-\alpha, \frac32-\frac{\alpha}{1-\alpha}\right\} . \]
The restriction  $\widetilde{s}_\alpha \ge \max\big\{1-\alpha, \frac32-\frac{\alpha}{1-\alpha}\big\}$  is required to control the estimates in the regions $A$, $B$ and $C$ of the proof of Proposition \ref{prop:Es} below.  Meanwhile the restriction on $\widetilde{s}_\alpha \ge \frac12-\frac{\alpha}4$ in Lemma \ref{apriori:smooth} arises from  Corollary \ref{prop:yst} and more precisely from the Strichartz estimates in Proposition \ref{propse}.
\end{remark}

\begin{proof}
We consider the time extensions $\widetilde{u} = \rho_T(u)$ of $u$  (see \eqref{defrho}) and use \eqref{est:rho} and \eqref{est:yst} at the end of each step to estimate $\|\rho_{T}(u)\|_{Y^s} \lesssim\|u\|_{\widetilde{Y}^s_T} $.  For the sake of simplicity, we will denote $\widetilde{u}$ by $u$ in the rest of the proof. 

 Using again equation \eqref{fKdV}, we differentiate in time $\mathcal{E}_3(t)$ and integrate between $0$ and $t$ to find
\begin{align*}
\mathcal{E}_3(t) + \mathcal{I}(t) &= \mathcal{E}_3(0) + \sum_{j=1}^3 \int_{\Gamma^3_t} b_3(\xi_1,\xi_2,\xi_3) \widehat{\partial_x(u^2)}(\xi_j) \prod_{i\neq j}\widehat{u}(\xi_i) \\
&= \mathcal{E}_3(0) + 3\int_{\Gamma^3_t}b_3(\xi_1,\xi_2,\xi_3) \widehat{u}(\xi_1) \widehat{u}(\xi_2) \widehat{\partial_x(u^2)}(\xi_3) \\
&= \mathcal{E}_3(0) -3i\int_{\Gamma^4_t} (\xi_1+\xi_2)b_3(\xi_1,\xi_2) \prod_{i=1}^4 \widehat{u}(\xi_i),
\end{align*}
where in the last inequality we omit for simplicity the third variable in $b_3$ (ie: $b_3(\xi_1,\xi_2) = b_3(\xi_1,\xi_2,-\xi_1-\xi_2)$). Next we symmetrize again to obtain
$$
\mathcal{E}_3(t) + \mathcal{I}(t) = \mathcal{E}_3(0) -\frac i2 \int_{\Gamma^4_t} a_4(\xi_1,\xi_2,\xi_3,\xi_4) \prod_{i=1}^4 \widehat{u}(\xi_i)
$$
where $a_4$ is defined on $\Gamma^4$ by
$$
a_4(\xi_1,\xi_2,\xi_3,\xi_4) = \sum_{1 \le i<j \le 4} (\xi_i+\xi_j)b_3(\xi_i,\xi_j),
$$
so that $a_4$ is symmetric: $a_4(\xi_1,\xi_2,\xi_3,\xi_4) = a_4(\xi_{\sigma(1)}, \xi_{\sigma(2)}, \xi_{\sigma(3)}, \xi_{\sigma(4)})$ for any permutation $\sigma\in \mathcal{S}_4$. 

Gathering this with \eqref{e2tilde} and \eqref{est:I0} we infer
$$
\sup_{t\in[0,T]}|E^s(t)| \lesssim |E^s(0)| + T N_0^{\frac 32}\|u\|_{L^\infty_TL^2_x}\|u\|_{L^\infty_TH^s_x}^2 + \sup_{t\in[0,T]} |\mathcal{J}(t)|
$$
where 
\begin{equation} \label{def:J}
\mathcal{J}(t) =  \int_{\Gamma^4_t} a_4(\xi_1,\xi_2,\xi_3,\xi_4) \prod_{i=1}^4 \widehat{u}(\xi_i)
\end{equation}

We localize the variables $|\xi_i|\sim N_i$ in the definition of $\mathcal{J}$, and we assume $N_1\le N_2\le N_3\le N_4$ by symmetry. From the definition of $a_4$, the integrand in $\mathcal{J}$ vanishes unless $N_4\gtrsim N_0$. Therefore, it suffices to estimate for a fixed $N\gtrsim N_0$
$$
\mathcal{J}_N(t) = \sum_{N_1\le N_2\le N} \sum_{M_1, M_2, M_3} \int_{\R_t} \Pi_\chi(u_{N_1}, u_{N_2}, u_{N}) u_{\sim N} dxdt
$$ 
where 
$$
\chi(\xi_1,\xi_2,\xi_3,\xi_4) = \phi_{M_1}(\xi_2+\xi_3)\phi_{M_2}(\xi_1+\xi_3)\phi_{M_3}(\xi_1+\xi_2) a_4(\xi_1,\xi_2,\xi_3,\xi_4).
$$
To complete the proof of \eqref{est:Es}, it remains to show that for all $t\in (0,T)$, we have
$$
|\mathcal{J}_N(t)| \lesssim N^{-\gamma} \|u\|_{Y^{s'}}^2\|u\|_{Y^s}^2.
$$
Consider the decomposition $\mathcal{J}_N = \mathcal{J}_N^{low} + \mathcal{J}_N^{high}$ where $\mathcal{J}_N^{low}$ denotes the contribution of the low modulation frequencies $\Omega:= M_{min}M_{med}N^{\alpha-1} \ll N^\alpha$, and $\mathcal{J}_N^{high}$ is the high modulation part $\Omega\gtrsim N^\alpha$. Hence, we may use Proposition \ref{prop:l2tri} to evaluate $\mathcal{J}_N^{high}$.
 
From Proposition \ref{est-a4.0} the symbol $\chi$ is bounded by
\begin{equation}\label{est-a4}
  \|\chi\|_{L^\infty} \lesssim M_{min} M_{med}(N_2^{2s-\alpha}\vee N^{2s-\alpha}) N^{-1}.
\end{equation}
Combined with \eqref{l2tri-basic}, we infer, for $s>s'>(1-\alpha)/2 $,
\begin{align*}
  |\mathcal{J}_N^{low}(t)| &\lesssim \sum_{N_1\le N_2\le N} (N_1N_2)^{\frac 12} (N_2^{2s-\alpha} \vee N^{2s-\alpha}) \|u_{N_1}\|_{Y^0} \|u_{N_2}\|_{Y^0} \|u_{\sim N}\|_{Y^0}^2 \\
  &\lesssim \sum_{N_1\le N_2\le N} N_1^{\frac 12}\cro{N_1}^{-s'} \left(N_2^{2s+\frac 12-\alpha}\cro{N_2}^{-s'}N^{-2s} + N_2^{\frac 12}\cro{N_2}^{-s'}N^{-\alpha}\right) \|u\|_{Y^{s'}}^2\|u\|_{Y^s}^2\\
  &\lesssim N^{\max(1-\alpha-s', -\alpha, -2s)}  \|u\|_{Y^{s'}}^2\|u\|_{Y^s}^2.
\end{align*}

To deal with $\mathcal{J}_N^{high}$, let us split the summation in $N_1,N_2$ into 3 regions:
\begin{enumerate}
  \item[$\bullet$] \textbf{Region A}: $N_1\sim N_2\sim N$,  
  \item[$\bullet$] \textbf{Region B}: $N_1\ll N_2\sim N$, 
  \item[$\bullet$] \textbf{Region C}: $N_1\le N_2\ll N$,
\end{enumerate} 
leading to the decomposition $\mathcal{J}_N^{high} = \mathcal{J}_{N,A}^{high} +\mathcal{J}_{N,B}^{high} + \mathcal{J}_{N,C}^{high}$.

\vskip .3cm
\noindent
 \textbf{Region A}: $N_1\sim N_2\sim N$ .\\
 In this region we have the bound $\|\chi\|_{L^\infty} \lesssim M_{min}M_{med}N^{2s-1-\alpha}$. Thus, Proposition \ref{prop:l2tri} leads to
 \begin{align*}
   |\mathcal{J}_{N,A}^{high}(t)| &\lesssim \sum_{M_1,M_2,M_3\lesssim N} M_{min}N(M_{min}M_{med}N^{\alpha-1})^{-1} M_{min}M_{med}N^{2s-1-\alpha} \|u_{\sim N}\|_{Y^0}^4 \\
&\lesssim N^{0_+} N^{-2s'+2-2\alpha} \|u\|_{Y^{s'}}^2\|u\|_{Y^s}^2 \\
&\lesssim N^{0_-}  \|u\|_{Y^{s'}}^2\|u\|_{Y^s}^2
\end{align*}
 as soon as $s'>1-\alpha$.

\vskip .3cm  
\noindent
\textbf{Region B}: $N_1\ll N_2\sim N$.\\
 In this region we have $M_{min}\sim M_{med} \sim M_{max} \sim N$, so that $\|\chi\|_{L^\infty} \lesssim N^{2s+1-\alpha}$ and $\Omega\sim N^{1+\alpha}$. It follows from \eqref{l2tri.1} that
\begin{align*}
  |\mathcal{J}_{N,B}^{high}(t)| &\lesssim \sum_{N_1\ll N} (N_1N)^{\frac 12} N N^{-1-\alpha} N^{2s+1-\alpha} \|u_{N_1}\|_{Y^0} \|u_{\sim N}\|_{Y^0}^3  \\
  &\lesssim  \sum_{N_1\ll N} N_1^{\frac 12}\cro{N_1}^{-s'} N^{-s'+\frac 32-2\alpha}  \|u\|_{Y^{s'}}^2\|u\|_{Y^s}^2,
\end{align*}
which is acceptable for $s'>\max\{1-\alpha,\frac32-2\alpha\}$.

\vskip .3cm
\noindent
 \textbf{Region C}: $N_1\le N_2\ll N$.\\
 From \eqref{est-a4}, it holds that $\|\chi\|_{L^\infty} \lesssim M_3(N_2^{2s-\alpha}\vee N^{2s-\alpha})$. Moreover, $\Omega \sim M_3N^{\alpha}$, and since we are in the high modulation region, we must have $M_3\gtrsim 1$  and thus $ N_2\gtrsim 1$.
 Assume first that $\alpha> \frac 12$. Using \eqref{l2tri.1}, we bound $ |\mathcal{J}_{N,C}^{high}(t)|$ by
 \begin{align*}
   \sum_{N_1\le N_2\ll N}& \sum_{1\lesssim M_3\lesssim N_2} (N_1N_2)^{\frac 12} N(M_3N^\alpha)^{-1} M_3(N_2^{2s-\alpha} + N^{2s-\alpha}) \|u_{N_1}\|_{Y^0} \|u_{N_2}\|_{Y^0} \|u_{\sim N}\|_{Y^0}^2\\
   &\lesssim N^{0_-} \sum_{N_1\le N_2\ll N\atop N_2\gtrsim 1} N_1^{\frac 12}\cro{N_1}^{-s'} N_2^{\frac 12-s'} N_2^{(1-2\alpha)_+}  \|u\|_{Y^{s'}}^2\|u\|_{Y^s}^2\\
   &\lesssim N^{0_-} \|u\|_{Y^{s'}}^2\|u\|_{Y^s}^2, 
 \end{align*} 
 where we used in the last step that $s' >1-\alpha$. Now we turn to the case $\alpha\le \frac 12$. Note that $s'>\frac 32-\frac{\alpha}{1-\alpha} \ge \frac 12\ge \frac{\alpha}2$. Applying \eqref{l2tri.1} with $\theta = \left(\frac{\alpha}{1-\alpha}\right)_-\in (0,1)$ we estimate $ |\mathcal{J}_{N,C}^{high}(t)|$ by
 \begin{align*}
   \sum_{N_1\le N_2\ll N} & \sum_{1\lesssim M_3\lesssim N_2} (N_1N_2)^{\frac 12} (N(M_3N^\alpha)^{-1})^{\theta} M_3N^{2s-\alpha} \|u_{N_1}\|_{Y^0} \|u_{N_2}\|_{Y^0} \|u_{\sim N}\|_{Y^0}^2\\
   &\lesssim N^{0_-} \sum_{N_1\le N_2} \sum_{1\lesssim M_3\lesssim N_2} N_1^{\frac 12}\cro{N_1}^{-s'} N_2^{\frac 12-s'} M_3^{1-\theta}  \|u\|_{Y^{s'}}^2\|u\|_{Y^s}^2\\
   &\lesssim N^{0_-} \sum_{N_1\le N_2\atop N_2\gtrsim 1} N_1^{\frac 12}\cro{N_1}^{-s'} N_2^{\frac 32-\theta-s'}  \|u\|_{Y^{s'}}^2\|u\|_{Y^s}^2,
 \end{align*}
 which is acceptable since $s'>\frac 32-\theta$.
\end{proof}

Now, for a dyadic number $N_0\gg 1$, we define the high frequency energy of a solution $u$ by
\begin{equation}\label{def:Eshigh}
E^s_{>N_0}(t) = E^s_{>N_0}(t,N_0,u) = \mathcal{E}_2(t,\Pi_{>N_0}u) + \frac13\mathcal{E}_3(t,N_0,u) ,
\end{equation}
where 
$$
 \mathcal{E}_2(t,\Pi_{>N_0}u)= \frac 12  \|D_x^s\Pi_{>N_0}u(t)\|_{L^2_x}^2 
$$
and $\mathcal{E}_3(t,N_0,u)$ is defined exactly as in \eqref{def:E3}. Recall here the notation $\Pi_{>N}u=\mathcal{F}^{-1}_x\big({\bf 1}_{|\xi| > N}\mathcal{F}_xu\big)$. As a byproduct of the proofs of Lemma \ref{lemma:coer} and Proposition \ref{prop:Es}, we obtain the following estimates on $\mathcal{E}_3(t,N_0,u)$. 

\begin{corollary} \label{coro:Eshigh}
There exists $C_2>0$ and $\gamma>0$ such that the following holds true. Let $\widetilde{s}_\alpha<s'\le s \le 3$, $0<T\le 1$ and $u\in Y^s_T$ be a solution of \eqref{fKdV} on $[0,T]$. Let $N_0$ be a dyadic number satisfying $N_0\ge  C_1 (1+\|u\|_{L^\infty_TH^{\widetilde{s}_\alpha}})^{\frac 2 \alpha}$, where $C_1$ is defined in Lemma \ref{lemma:coer}. 
\begin{itemize}
\item[(i)] For any $t\in (0,T)$   
\begin{equation}\label{coro:Eshigh.1}
    \left| E^s_{>N_0}(t,N_0,u) - \frac 12\|D^s_x\Pi_{>N_0}u(t)\|_{L^2}^2\right|\le \frac 14\|D^s_x\Pi_{>N_0}u(t)\|_{L^2}^2.
  \end{equation}
  \item[(ii)]   \begin{equation}\label{coro:Eshigh.2}
    \sup_{t\in (0,T)} \left|E^s_{>N_0}(t)-E^s_{>N_0}(0)\right| \leq C_2 N_0^{-\gamma} \|u\|_{\widetilde{Y}^{s'}_T}^2\|u\|_{\widetilde{Y}^s_T}^2 .
  \end{equation}
  \end{itemize}
\end{corollary}

\begin{proof} The proof of \eqref{coro:Eshigh.1} is exactly the same as the one of \eqref{coesive}. 

To prove \eqref{coro:Eshigh.2}, we observe by using \eqref{fKdV} and arguing as above that
\begin{equation*}
\left|\int_0^t\frac{d}{dt}E^s_{>N_0}(t)\right| \lesssim |\mathcal{J}(t)| ,
\end{equation*}
where $\mathcal{J}(t)$ is defined in \eqref{def:J}. Then, we conclude the proof of  \eqref{coro:Eshigh.2} by estimating $\sup_{t \in (0,T)} |\mathcal{J}(t)|$ as in the proof of Proposition \ref{prop:Es}. 
\end{proof}

\section{Energy estimates for the difference of solutions} \label{Sec:5}
Let $0<T\le 1$ and $z, w$ be two solutions of \eqref{eqdiff} such that $z\in Y^s_T$ and $w\in \overline{Z}^\sigma_T$ with $s>s_\alpha$ and 
\begin{equation}\label{cond:sigma} 
-\frac 1 2 +\frac \alpha 4< \sigma<\min \left\{\frac{\alpha-1}2,s-\frac 32+ \alpha\right\}.
\end{equation}
Note that $-\frac 1 2 +\frac \alpha 4 < \min \{\frac{\alpha-1}2,s-\frac 32+ \alpha\}$ as soon as $s>1-\frac{3\alpha}4$. As a consequence, for $s>s_\alpha$ we can always choose $\sigma$ satisfying \eqref{cond:sigma}.

For fixed $N_0\gg 1$ let us define the modified energy $\widetilde{E}^\sigma(t) = \widetilde{E}^\sigma(t,N_0,z,w) $ by
\begin{equation}\label{def:esigma}
  \widetilde{E}^\sigma(t) = \frac 12\||\xi|^{-1}\widehat{w}(t,\xi)\|_{L^2(|\xi|\le 1)}^2 + \frac 12\||\xi|^{\sigma}\widehat{w}(t,\xi)\|_{L^2(|\xi| \ge 1)}^2+\frac 12 \widetilde{\mathcal{E}}^\sigma(t)
\end{equation}
where
$$
\widetilde{\mathcal{E}}^\sigma(t) = \int_{\Gamma^3}\widetilde{b}_3(\xi_1,\xi_2) \widehat{w}(t,\xi_1) \widehat{w}(t,\xi_2)\widehat{z}(t,\xi_3).
$$
To define the symbol $\widetilde{b}_3$ we first set
\begin{equation}\label{def:a3tilde}
\widetilde{a}_3(\xi_1,\xi_2) = (\widetilde{\nu}_\sigma(\xi_1) + \widetilde{\nu}_\sigma(\xi_2)){\bf 1}_{|\xi_1+\xi_2|>N_0},
\end{equation}
where $N_0 \gg \max\{1,\xi_0\}$ is a large dyadic number that will fixed later (here $\xi_0$ is defined in Hypothesis \ref{hyp1} so that Proposition \ref{est-Omega4} holds true),
\begin{equation} \label{def:tilde_nu}\widetilde{\nu}_\sigma(\xi)=\xi|\xi|^{2\sigma} {\bf 1}_{|\xi|\ge 1} . 
\end{equation} 
 and $\widetilde{b}_3$ is given by
$$
\widetilde{b}_3(\xi_1,\xi_2) = \frac{\widetilde{a}_3(\xi_1,\xi_2)}{\Omega_3(\xi_1,\xi_2)}.
$$
\begin{lemma} \label{coer:E_tilde_sigma}
  Assume that $s>s_\alpha$, $-\frac 1 2 +\frac \alpha 4< \sigma<\min \{\frac{\alpha-1}2,s-\frac 32+ \alpha\}$ and $0<T\le 1$. Let $z\in L^\infty_TH^s_x$ and $w\in L^\infty_T\overline{H}^\sigma_x$ be a solution of \eqref{eqdiff}. Then for any $t\in (0,T)$ and $N_0\gg (1+\|z\|_{L^\infty_TH^s})^2$, it holds
  \begin{equation}\label{coesive-diff}
    \left| \widetilde{E}^{\sigma}(t) - \frac 12\|w(t)\|_{\overline{H}^\sigma}^2\right|\le \frac 14 \|w(t)\|_{\overline{H}^\sigma}^2.
  \end{equation}
\end{lemma}
\begin{proof}
  Note that 
  \[ \|w(t)\|_{\overline{H}^\sigma}^2 \sim \||\xi|^{-1}\widehat{w}(t,\xi)\|_{L^2(|\xi|\le 1)}^2 + \||\xi|^{\sigma}\widehat{w}(t,\xi)\|_{L^2(|\xi| \ge 1)}^2.\]
 Then, from the definition of $\widetilde{E}^{\sigma}$ in \eqref{def:esigma} the left-hand side in \eqref{coesive-diff} is bounded by
  $$
  \frac 12|\widetilde{\mathcal{E}}^\sigma| \lesssim \sum_{N_1,N_2>0\atop N_3\gtrsim N_0} \left| \int_{\Gamma^3} \widetilde{b}_3(\xi_1,\xi_2) \widehat{w}_{N_1}(\xi_1) \widehat{w}_{N_2}(\xi_2)\widehat{z}_{N_3}(\xi_3) \right|
  $$
 where, for the sake of simplicity, we have removed the time dependency. By symmetry, we may assume that the sum is reduced to $N_1\le N_2$, so that $N_1\lesssim N_2\sim N_3$  and $N_3 \gtrsim N_0$. For the contribution $N_1\ll N_2$ we get from \eqref{lem:b3tilde.2} that 
   \begin{align*}
     & \sum_{N_1\ll N_2\atop N_2\gtrsim N_0} N_1^{-1}N_2^{2\sigma+1-\alpha} N_1^{\frac 12} \|w_{N_1}\|_{L^2_x} \|w_{N_2}\|_{L^2_x} \|z_{\sim N_2}\|_{L^2_x} \\
    &\lesssim \sum_{N_2\gtrsim N_0} N_2^{\sigma-s+1-\alpha} \|w\|_{\overline{H}^\sigma_x} \|w_{N_2}\|_{H^\sigma_x} \|z_{\sim N_2}\|_{H^s_x}\\
    &\lesssim N_0^{-\frac 12} \|w\|_{\overline{H}^\sigma_x}^2 \|z\|_{H^s_x}, 
  \end{align*}
  where we used $\sigma-s+1-\alpha<-\frac12$. In the same way using now \eqref{lem:b3tilde.1}, we bound the sum over $N_1\sim N_2$ by
  \begin{align*}
     & \sum_{N_1\gtrsim N_0} N_1^{2\sigma-\alpha+\frac 12} N_1^{-2\sigma-s} \|w_{N_1}\|_{H^\sigma_x} \|w_{\sim N_1}\|_{H^\sigma_x} \|z_{\sim N_1}\|_{H^s_x}\\
    &\lesssim N_0^{-\frac 12} \|w\|_{H^\sigma_x}^2 \|z\|_{H^s_x},
  \end{align*}
for $s>1-\alpha$. The claim follows then since $N_0^{-\frac 12}\|z\|_{H^s_x}\ll 1$ by our assumption on $N_0$.
\end{proof}
\begin{proposition} \label{prop:E_tilde_sigma}
  Assume $s>s_\alpha$, $-\frac 1 2 +\frac \alpha 4< \sigma<\min \{\frac{\alpha-1}2,s-\frac 32+ \alpha\}$ and $0<T\le 1$. Let $u, v\in Y^s_T$ be two solutions of \eqref{fKdV} such that $w=u-v\in \overline{Z}^\sigma_T$. Then, setting $z=u+v$, it holds that for any $N_0\gg 1$,
  \begin{equation}\label{est:Esigma}
    \sup_{t\in (0,T)} |\widetilde{E}^\sigma(t)| \lesssim |\widetilde{E}^\sigma(0)| + TN_0^{\frac 32} \|z\|_{Y^s_T} \|w\|_{\overline{Z}^\sigma_T}^2 +  N_0^{-\gamma}(\|u\|_{\widetilde{Y}^s_T}^2 + \|v\|_{\widetilde{Y}^s_T}^2) \|w\|_{\overline{Z}^\sigma_T}^2
  \end{equation}
  where $\gamma>0$ depends only on $\alpha$, $s$ and $\sigma$.
\end{proposition}

\begin{remark}
By looking at  the estimates in the regions $A$, $C_1$ and $C_4$ below, we see that the proof of Proposition \ref{prop:E_tilde_sigma} holds under the condition
\[ s>\max\left\{1-\frac{3\alpha}4,\frac32(1-\alpha), \frac32-\frac{\alpha}{1-\alpha}\right\}, \] 
which corresponds specifically to the index $s_{\alpha}$ defined in \eqref{cond-s}. 
\end{remark}

\begin{proof}
  We consider the time extensions $\widetilde{w} = \rho_T(w)$ and $\widetilde{z}=\rho_T(z)$ (see \eqref{defrho}) that we still denote $w$ and $z$ for simplicity. By using equation \eqref{eqdiff} we infer after differentiating in time and integrating from $0$ to $t$ with $0<t<T$,
\begin{align}
\widetilde{E}^\sigma(t) 
&= \widetilde{E}^\sigma(0) -i\int_{\Gamma^3_t} {\bf 1}_{|\xi_1|\le 1}\xi_1^{-1}\widehat{w}(\xi_1) \widehat{w}(\xi_2)\widehat{z}(\xi_3)  \notag \\
&\quad  -i \int_{\Gamma^3_t}\xi_1|\xi_1|^{2\sigma} {\bf 1}_{|\xi_1|\ge 1} \widehat{w}(\xi_1) \widehat{w}(\xi_2)\widehat{z}(\xi_3) +c\int_0^t \frac{d}{dt}\widetilde{\mathcal{E}}^{\sigma}(t')dt' \notag \\
&=: \widetilde{E}^\sigma(0) + \widetilde{\mathcal{I}}_0(t) + \widetilde{\mathcal{I}}(t) + \widetilde{\mathcal{J}}(t). \label{est:Esigma.1}
\end{align}
The low-frequency term is estimated by using $-\sigma<s$, so that $1 \lesssim \langle \xi_2 \rangle^{\sigma} \langle \xi_3 \rangle^{s} $. Thus,
\begin{equation}\label{est:Esigma.2}
|\widetilde{\mathcal{I}}_0(t)| \lesssim T \|P_{\lesssim 1}w\|_{L^\infty_T\overline{L^2_x}} \|w\|_{L^\infty_TH^\sigma_x} \|z\|_{L^\infty_TH^s_x}
\end{equation}
Next we symmetrize $\widetilde{\mathcal{I}}$ in $(\xi_1, \xi_2)$ by writing
$$
\widetilde{\mathcal{I}}(t) = -\frac i2 \int_{\Gamma^3_t} (\widetilde{\nu}_\sigma(\xi_1)+\widetilde{\nu}_\sigma(\xi_2))  \widehat{w}(\xi_1) \widehat{w}(\xi_2)\widehat{z}(\xi_3) ,
$$
where $\widetilde{\nu}_\sigma$ is defined in \eqref{def:a3tilde}.
We will only cancel out the region where $|\xi_3|>N_0$. Thus we split $\widetilde{\mathcal{I}}=\widetilde{\mathcal{I}}_1+\widetilde{\mathcal{I}}_2$ with
$$
\widetilde{\mathcal{I}}_2(t) = -\frac i2 \int_{\Gamma^3_t} \widetilde{a_3}(\xi_1,\xi_2)  \widehat{w}(\xi_1) \widehat{w}(\xi_2)\widehat{z}(\xi_3)
$$
where $\widetilde{a}_3$ is defined in \eqref{def:a3tilde}. The low $\xi_3$-frequencies term is easily bounded by
\begin{equation}\label{est:Esigma.3}
|\widetilde{\mathcal{I}}_1| \lesssim TN_0^{\frac 32} \|z\|_{Y^s_T} \|w\|_{\overline{Z}^\sigma_T}^2,
\end{equation}
where  we used that, in the region $|\xi_1|\sim |\xi_2|\gg N_0$, the mean value theorem implies
$$
|\widetilde{\nu}_\sigma(\xi_1)+\widetilde{\nu}_\sigma(\xi_2|\sim |\xi_1+\xi_2| |\xi_1|^{2\sigma}.
$$
Using again equation \eqref{eqdiff}, we deduce that $\widetilde{\mathcal{I}}_2(t) + \widetilde{\mathcal{J}}(t)$ is equal to
\begin{align}
  &2\int_{\Gamma^3_t} \widetilde{b_3}(\xi_1, \xi_2) \widehat{\partial_x(zw)}(\xi_1) \widehat{w}(\xi_2) \widehat{z}(\xi_3) + \int_{\Gamma^3_t} \widetilde{b_3}(\xi_1, \xi_2) \widehat{w}(\xi_1) \widehat{w}(\xi_2) \widehat{\partial_x(u^2+v^2)}(\xi_3) \notag \\
  &=2\int_{\Gamma^3_t} \widetilde{b_3}(\xi_1, \xi_2) \widehat{\partial_x(zw)}(\xi_1) \widehat{w}(\xi_2) \widehat{z}(\xi_3) + \frac 12\int_{\Gamma^3_t} \widetilde{b_3}(\xi_1, \xi_2) \widehat{w}(\xi_1) \widehat{w}(\xi_2) \widehat{\partial_x(z^2)}(\xi_3)\notag \\
  &\quad + \frac 12\int_{\Gamma^3_t} \widetilde{b_3}(\xi_1, \xi_2) \widehat{w}(\xi_1) \widehat{w}(\xi_2) \widehat{\partial_x(w^2)}(\xi_3) \notag \\
  &=: \widetilde{\mathcal{K}}_1(t) + \widetilde{\mathcal{K}}_2(t) + \widetilde{\mathcal{L}}(t) \label{est:Esigma.4}
\end{align}
In the rest of the proof, we will show the following estimates:
\begin{equation}\label{est:Esigma.5}
 |\widetilde{\mathcal{K}}_1(t) + \widetilde{\mathcal{K}}_2(t)| \lesssim N_0^{-\gamma} \|w\|_{\overline{Z}^\sigma}^2 \|z\|_{Y^s}^2,
\end{equation}
and
\begin{equation}\label{est:Esigma.6}
 |\widetilde{\mathcal{L}}(t)| \lesssim N_0^{-\gamma} \|w\|_{\overline{Z}^\sigma}^2 (\|u\|_{Y^s}^2 + \|v\|_{Y^s}^2),
\end{equation}
for some $\gamma>0$. Combined with \eqref{est:Esigma.1}--\eqref{est:Esigma.4}, this yields \eqref{est:Esigma}.

\vskip .5cm
\noindent
\textbf{Estimate of $\widetilde{\mathcal{K}} = \widetilde{\mathcal{K}}_1 + \widetilde{\mathcal{K}}_2$.}\\ 
The first term in $\widetilde{\mathcal{K}}$ may rewritten as 
\begin{align*}
  \widetilde{\mathcal{K}}_1(t) &= 2i\int_{\Gamma^4_t} (\xi_1+\xi_3)\widetilde{b_3}(\xi_1+\xi_3, \xi_2) \widehat{w}(\xi_1) \widehat{w}(\xi_2) \widehat{z}(\xi_3) \widehat{z}(\xi_4)\\
  &= i \int_{\Gamma^4_t} (\xi_1+\xi_3)\left[\widetilde{b_3}(\xi_1+\xi_3, \xi_2) - \widetilde{b_3}(\xi_1+\xi_3, -\xi_1)\right] \widehat{w}(\xi_1) \widehat{w}(\xi_2) \widehat{z}(\xi_3) \widehat{z}(\xi_4)
\end{align*}
since $\xi_2+\xi_4=-(\xi_1+\xi_3)$ and $\widetilde{b_3}(\mu,\eta)=\widetilde{b_3}(-\mu, -\eta)$. A second symmetrization leads to
\begin{align*}
  \widetilde{\mathcal{K}}_1(t) &= \frac i2 \int_{\Gamma^4_t} \Big((\xi_1+\xi_3)\left[\widetilde{b_3}(\xi_1+\xi_3, \xi_2) - \widetilde{b_3}(\xi_1+\xi_3, -\xi_1)\right] \\
  &\quad + (\xi_2+\xi_3)\left[\widetilde{b_3}(\xi_2+\xi_3, \xi_1) - \widetilde{b_3}(\xi_2+\xi_3, -\xi_2)\right] \Big) \widehat{w}(\xi_1) \widehat{w}(\xi_2) \widehat{z}(\xi_3) \widehat{z}(\xi_4).
\end{align*}
On the other hand, the second term in $\widetilde{\mathcal{K}}$ is
$$
\widetilde{\mathcal{K}}_2(t) = -\frac i2 \int_{\Gamma^4_t} (\xi_1+\xi_2)\widetilde{b_3}(\xi_1, \xi_2) \widehat{w}(\xi_1) \widehat{w}(\xi_2) \widehat{z}(\xi_3) \widehat{z}(\xi_4)
$$
Combining theses identities we conclude
$$
\widetilde{\mathcal{K}}(t) = \frac i2\int_{\Gamma^4_t} \widetilde{a}_4(\xi_1, \xi_2, \xi_3)  \widehat{w}(\xi_1) \widehat{w}(\xi_2) \widehat{z}(\xi_3) \widehat{z}(\xi_4)
$$
with a symbol
\begin{align}
\widetilde{a}_4(\xi_1, \xi_2, \xi_3) &= (\xi_1+\xi_3)\left[\widetilde{b_3}(\xi_1+\xi_3, \xi_2) - \widetilde{b_3}(\xi_1+\xi_3, -\xi_1)\right] \nonumber \\
&\quad + (\xi_2+\xi_3)\left[\widetilde{b_3}(\xi_2+\xi_3, \xi_1) - \widetilde{b_3}(\xi_2+\xi_3, -\xi_2)\right] \nonumber \\
&\quad - (\xi_1+\xi_2)\widetilde{b_3}(\xi_1, \xi_2) \nonumber \\
&=: (\widetilde{a}_{41} + \widetilde{a}_{42} + \widetilde{a}_{43})(\xi_1,\xi_2,\xi_3). \label{def:a4tilde}
\end{align}
Note that $\widetilde{a_4}$ is symmetric in $(\xi_1,\xi_2)$ and in $(\xi_3,\xi_4)$. We collect all needed estimates on this symbol in Lemma \ref{lem:a4tilde}.
From the localization properties in the definition of $\widetilde{a}_3$, it is clear that $\widetilde{a}_4 = 0$ in the region $\max(|\xi_1|, |\xi_2|, |\xi_3|) \ll N_0$.
Thus we may decompose $\widetilde{\mathcal{K}}$ as $\widetilde{\mathcal{K}} = \sum_{N\gtrsim N_0} \widetilde{\mathcal{K}}_N$ where
$$
\widetilde{\mathcal{K}}_N(t) = \sum_{N_1\lesssim N_2, N_3\lesssim N_4 \atop N_{max}=N} \sum_{M_1,M_2,M_3>0} \int_{\R_t} \Pi_{\widetilde{\chi}}(w_{N_1}, w_{N_2}, z_{N_3})z_{N_4}
$$
with
$$
\widetilde{\chi}(\xi_1,\xi_2,\xi_3) = \frac i2\phi_{M_1}(\xi_2+\xi_3)\phi_{M_2}(\xi_1+\xi_3)\phi_{M_3}(\xi_1+\xi_2) \widetilde{a}_4(\xi_1,\xi_2,\xi_3).
$$
In the sequel, we will split the sum in $N_i$ into several regions, and  we will denote by $\widetilde{\mathcal{K}}_{N,\mathcal{R}}$ the contribution of the region $\mathcal{R}$ to $\widetilde{\mathcal{K}}_N$. Estimate \eqref{est:Esigma.6} follows then if we prove that for any $N\gtrsim N_0$ and all $\mathcal{R}$, it holds
$$
|\widetilde{\mathcal{K}}_{N,\mathcal{R}}(t)| \lesssim  N^{0_-} \|w\|_{\overline{Z}^\sigma}^2 \|z\|_{Y^s}^2.
$$
For each region $\mathcal{R}$, we decompose $\widetilde{\mathcal{K}}_{N,\mathcal{R}} = \widetilde{\mathcal{K}}_{N,\mathcal{R}}^{low} + \widetilde{\mathcal{K}}_{N,\mathcal{R}}^{high}$ where $\widetilde{\mathcal{K}}_{N,\mathcal{R}}^{low}$ denotes the contribution of the sum  over $\Omega:= M_{min}M_{med}N^{\alpha-1}\ll N^\alpha$. This ensures that we may apply Proposition \ref{prop:l2tri} for the high modulation term $\widetilde{\mathcal{K}}_{N,\mathcal{R}}^{high}$.

\vskip .3cm
\noindent
\textbf{Region $A$}: $N_1\sim N_2\sim N_3\sim N_4 \sim N$.\\
In this region, the bound
 $ \|\widetilde{\chi}\|_{L^\infty} \lesssim N^{2\sigma+1-\alpha}$ follows from \eqref{a4tilde-A}. For the contribution of $\widetilde{\mathcal{K}}_{N,A}^{low}$, we deduce from \eqref{l2tri-basic} and $M_{min} \ll N^{\frac12}$, (which follows from $\Omega \ll N^{\alpha}$), that
$$
|\widetilde{\mathcal{K}}_{N,A}^{low}|\lesssim \sum_{M_1,M_2,M_3\atop M_{min}M_{med}\ll N} M_{min} N^{2\sigma+1-\alpha} \|w_{\sim N}\|_{Z^0}^2 \|z_{\sim N}\|_{Y^0}^2 \lesssim  N^{0_-} \|w\|_{Z^\sigma}^2 \|z\|_{Y^s}^2,
$$
since $s>1-\frac{3\alpha}4 \ge \frac 34-\frac \alpha 2$. In the case $\Omega\gtrsim N^\alpha$, the above bound on $\|\widetilde{\chi}\|_{L^\infty} $ and  \eqref{l2tri.1} with $\theta=1/2$ provide
\begin{align*}
  |\widetilde{\mathcal{K}}_{N,A}^{high}(t)| &\lesssim \sum_{1\lesssim M_1,M_2,M_3\lesssim N} M_{min} N^{\frac 12} (M_{min}M_{med}N^{\alpha-1})^{-\frac 12} N^{2\sigma+1-\alpha} \|w_{\sim N}\|_{Z^0}^2 \|z_{\sim N}\|_{Y^0}^2 \\
  &\lesssim N^{0_+} N^{-2s+2-\frac{3\alpha}2} \|w\|_{Z^s}^2 \|z\|_{Y^s}^2\\
  &\lesssim N^{0_-} \|w\|_{Z^\sigma}^2 \|z\|_{Y^s}^2
\end{align*}
where we used that $s>1-\frac{3\alpha}4$.

\vskip .3cm
\noindent
\textbf{Region $B$:} $N_{min}\ll N_{th}\sim N_{sub}\sim N_{max}=N$.\\
Note that in this region we have $M_{min}\sim M_{max}\sim N$, and thus $\Omega \sim N^{\alpha+1}\gtrsim N^\alpha$, so that $\widetilde{\mathcal{K}}_{N,B}^{low} = 0$. 

\vskip .3cm
\noindent
\textbf{Subregion $B_1$:} $N_{min}=N_1$. \\
In this region it holds $\|\widetilde{\chi}\|_{L^\infty} \lesssim N_1^{-1} N^{2\sigma+2-\alpha}$ thanks to \eqref{a4tilde-B1}.
Hence, \eqref{l2tri.1} with $\theta=1$ yields
 \begin{align*}
  |\widetilde{\mathcal{K}}_{N,B_1}^{high}(t)| &\lesssim \sum_{N_1\ll N} \sum_{M_i\sim N} (N_1N)^{\frac 12} NN^{-1-\alpha} N_1^{-1} N^{2\sigma+2-\alpha} \|w_{N_1}\|_{Z^0} \|w_{\sim N}\|_{Z^0} \|z_{\sim N}\|_{Y^0}^2\\
  &\lesssim N^{0_+} N^{\sigma-2s+\frac 52-2\alpha} \|w\|_{\overline{Z}^\sigma} \|w\|_{Z^\sigma} \|z\|_{Y^s}^2,
  \end{align*}
  which is acceptable since $\sigma-2s+\frac 52-2\alpha < -s+1-\alpha<0$.

\vskip .3cm 
\noindent
\textbf{Subregion $B_2$:} $N_{min}=N_3$. \\
On the support of $\widetilde{\chi}$ we have $|\widetilde{a_4}| \lesssim N^{2\sigma+1-\alpha}$. By \eqref{l2tri.1} with $\theta=1$, we get
\begin{align*}
   |\widetilde{\mathcal{K}}_{N,B_2}^{high}(t)| &\lesssim \sum_{N_3\ll N} \sum_{M_i\sim N} (N_3N)^{\frac 12} NN^{-1-\alpha} N^{2\sigma+1-\alpha} \|w_{\sim N}\|_{Z^0}^2 \|z_{N_3}\|_{Y^0} \|z_{\sim N}\|_{Y^0}\\
   &\lesssim \sum_{N_3\ll N} N_3^{\frac 12}\cro{N_3}^{-s} N^{-s+\frac 32-2\alpha} \|w\|_{Z^\sigma}^2 \|z\|_{Y^s}^2, 
 \end{align*}
 which is acceptable since $s_\alpha > 1-\alpha$ and $s_\alpha > \frac 32-2\alpha$. 

\vskip .3cm   
\noindent
\textbf{Region $C$:} $N_{min} \lesssim N_{th}\ll N_{sub}\sim N_{max}=N$. \\
In this region, it holds $M_{min}\ll M_{med}\sim M_{max}\sim N$ and $\Omega\sim M_{min}N^\alpha$, so that 
\[\Omega\ll N^\alpha \Longleftrightarrow M_{min}\ll 1. \]

\noindent
\textbf{Subregion C1:} $N \sim N_1\sim N_2\gg N_4\ge N_3$. \\
Note that in this region $M_3=M_{min}$ and $\|\widetilde{\chi}\|_{L^{\infty}} \lesssim N_4N^{2\sigma-\alpha}$ by $\eqref{a4tilde-C1}$. The term $\widetilde{\mathcal{K}}_{N,C_1}^{low}$ is easily estimated thanks to \eqref{l2tri-basic} as well as \eqref{a4tilde-C1} by
$$
|\widetilde{\mathcal{K}}_{N,C_1}^{low}| \lesssim \sum_{N_3\lesssim N_4\ll N} \sum_{M_3\ll 1} M_3 N_4 N^{-\alpha} \|w_{\sim N}\|_{Z^\sigma}^2 \|z_{N_3}\|_{Y^0} \|z_{N_4}\|_{Y^0}\lesssim N^{0_-} \|w\|_{Z^\sigma}^2 \|z\|_{Y^s}^2.
$$
To deal with the high modulation part, we proceed as in the region $C$ of the a priori estimate, and we first suppose $\alpha>1/2$. Recall that we have $N_4\gtrsim M_3\gtrsim 1$. Then we infer from \eqref{l2tri.1} with $\theta=1$ and \eqref{a4tilde-C1} that $ |\widetilde{\mathcal{K}}_{N,C_1}^{high}(t)|$ is bounded by
 $$
  \sum_{N_3\lesssim N_4\ll N}\sum_{1\lesssim M_3\lesssim N_4} \min\{M_3, (N_3N_4)^{\frac 12}\}M_3^{-1} N^{1-\alpha} N_4 N^{2\sigma-\alpha}  
    \|w_{\sim N}\|_{Z^0}^2 \|z_{N_3}\|_{Y^0} \|z_{N_4}\|_{Y^0}
$$
Moreover, we observe that
$$
\min\{M_3, (N_3N_4)^{\frac 12}\} M_3^{-1} N_4 \lesssim (N_3N_4)^{\frac 12}.
$$  
Hence,
\begin{align*}
  |\widetilde{\mathcal{K}}_{N,C_1}^{high}(t)| &\lesssim  N^{0_-} \sum_{N_3\lesssim N_4} N_3^{\frac 12}\cro{N_3}^{-s} N_4^{\frac 12-s} N^{(1-2\alpha)_+}  \|w\|_{Z^\sigma}^2 \|z\|_{Y^s}^2\\
  &\lesssim N^{0_-} \|w\|_{Z^\sigma}^2 \|z\|_{Y^s}^2
\end{align*}
since $s>\max(1-\alpha, \frac 32-2\alpha)$. In order to handle the case $\alpha\le 1/2$ we apply estimate \eqref{l2tri.1} with $\theta = \left(\frac{\alpha}{1-\alpha}\right)_-$ as well as \eqref{a4tilde-C1} to bound $ |\widetilde{\mathcal{K}}_{N,C_1}^{high}(t)|$ by
\begin{align*}
\sum_{N_3\lesssim N_4\ll N}\sum_{1\lesssim M_3\lesssim N_4} &\min\{M_3, (N_3N_4)^{\frac 12}\} M_3^{-\theta} N^{\theta(1-\alpha)} N_4 N^{2\sigma-\alpha}  
    \|w_{\sim N}\|_{Y^0}^2 \|z_{N_3}\|_{Y^0} \|z_{N_4}\|_{Y^0}\\
    &\lesssim N^{0_-} \sum_{N_3\lesssim N_4} N_3^{\frac 12}\cro{N_3}^{-s} N_4^{\frac 12-s} N_4^{1-\theta} \|w\|_{Z^\sigma}^2 \|z\|_{Y^s}^2,
\end{align*}
which is acceptable since $s>\frac 32-\theta$.
  
\vskip .3cm   
\noindent
\textbf{Subregion $C_2$:} $N_1\le N_2\ll N_3\sim N_4=N$. \\
In this region, $M_{min}=M_3$ and 
$\|\widetilde{\chi}\|_{L^{\infty}} \lesssim N_1^{-1}N^{2\sigma+2-\alpha}$ by $\eqref{a4tilde-C2}$.
The bound for the low modulation part follows by using \eqref{l2tri-basic}  and noticing that if $M_3 \ll 1$, it holds
$$
\min\{M_3, (N_1N_2)^{\frac 12}\}N_1^{-1} \lesssim M_3^{\frac 12} N_1^{-\frac 12} \cro{N_2}^{-\frac 12}.
$$ 
Hence, $ |\widetilde{\mathcal{K}}_{N,C_2}^{low}(t)|$ is bounded by
\begin{align*}
   \sum_{N_1\le N_2\ll N}\sum_{M_3\ll 1}&\min\{M_3,(N_1N_2)^{\frac 12}\}  N_1^{-1} N^{2\sigma-2s+2-\alpha} \|w_{N_1}\|_{Z^0} \|w_{N_2}\|_{Z^0} \|z_{\sim N}\|_{Y^s}^2\\
   &\lesssim N^{0_-} 
    \|w_{N_1}\|_{\overline{Z^\sigma}} \|w_{N_2}\|_{\overline{Z^\sigma}} \|z\|_{Y^s}^2 ,
\end{align*}
where we used the conditions $-\frac12<\sigma$ and $2\sigma-2s+2-\alpha<\alpha-1<0$.

To handle the high modulation term, where $N_2\gtrsim M_3\gtrsim 1$, we use that
$$
\min\{M_3, (N_1N_2)^{\frac 12}\}N_1^{-1}M_3^{-1} \lesssim N_1^{-\frac 12}N_2^{-\frac 12}.
$$
Thus, we infer from \eqref{l2tri.1} with $\theta=1$ that $ |\widetilde{\mathcal{K}}_{N,C_2}^{high}(t)|$ is bounded by
\begin{align*}
  & \sum_{N_1\le N_2\ll N}\sum_{1\lesssim M_3\lesssim N_2}\min\{M_3,(N_1N_2)^{\frac 12}\} M_3^{-1}N^{1-\alpha} N_1^{-1} N^{2\sigma+2-\alpha} \|w_{N_1}\|_{Z^0} \|w_{N_2}\|_{Z^0} \|z_{\sim N}\|_{Y^0}^2\\
   &\lesssim N^{0_-} \sum_{N_1\le N_2\ll N\atop N_2\gtrsim 1} N_1^{-\frac 12}\cro{N_1}^{-\sigma} N_2^{-\frac 12-\sigma}
    N^{2\sigma-2s+3-2\alpha}\|w_{N_1}\|_{Z^\sigma} \|w_{N_2}\|_{Z^\sigma} \|z\|_{Y^s}^2 \\
   &\lesssim N^{0_-}  \|w\|_{\overline{Z}^\sigma}^2 \|z\|_{Y^s}^2,
\end{align*}
where we used the conditions $-\frac12<\sigma$ and $\sigma<s-\frac32+\alpha$.

\vskip .3cm 
\noindent
\textbf{Subregion $C_3$:} $N_1\ll N_3\ll N_2\sim N_4 \sim N$. \\
Note that in this region, it holds $M_{min}=M_2\sim N_3$ and $\Omega\sim N_3N^\alpha$. First, by using \eqref{l2tri-basic} and the basic bound $|\widetilde{a}_4(\xi_1,\xi_2,\xi_3)| \lesssim N_1^{-1} N_3^{2\sigma+2-\alpha}$ (see \eqref{a4tilde-C3b}), we bound $|\widetilde{\mathcal{K}}_{N,C_3}^{low}(t)|$ by
$$
   \sum_{N_1\ll N_3\ll 1} (N_1N_3)^{\frac 12} N_1^{-1} N_3^{2\sigma+2-\alpha}  \|w_{N_1}\|_{Z^0} \|w\|_{Z^\sigma} \|z_{N_3}\|_{Y^0} \|z\|_{Y^0} \lesssim N^{0_-} \|w\|_{\overline{Z}^\sigma}^2 \|z\|_{Y^s}^2 ,
$$
where we used the conditions $2\sigma+\frac52-\alpha$ and $-\sigma<s$.

Estimate \eqref{a4tilde-C3b} is not sufficient to get a suitable bound for $\widetilde{\mathcal{K}}_{N,C_3}^{high}$ without using the Strichartz norm in $Z^\sigma$, which would require a Coifmann-Meyer type estimate on the symbol $\widetilde{\chi}$ (see Theorem 2.2 in \cite{MPV}). To avoid this technical difficulty, we use \eqref{a4tilde-C3} to  decompose 
\begin{equation*}
 \widetilde{\mathcal{K}}_{N,C_3}^{high}(t) = \sum_{j=1}^3 \widetilde{\mathcal{K}}_{N,C_3}^{high,j}(t)
\end{equation*}
with
\begin{align*}
  \widetilde{\mathcal{K}}_{N,C_3}^{high,1}(t) &= \sum_{N_1\ll N_3\ll N\atop N_3 \gtrsim 1} \int_{\R_t} \Pi_{\widetilde{\chi}_1}(w_{N_1}, w_{\sim N}, z_{N_3})z_{\sim N}  \\ 
    \widetilde{\mathcal{K}}_{N,C_3}^{high,2}(t) &=\sum_{N_1\ll N_3\ll N\atop N_3 \gtrsim 1} \int_{\R_t} \Pi_{\widetilde{\chi}_2}(w_{N_1}, w_{\sim N}, z_{N_3})z_{\sim N} \\
   \widetilde{\mathcal{K}}_{N,C_3}^{high,3}(t) &= - \frac 12 \sum_{N_1\ll N_3\ll N\atop N_3 \gtrsim 1} \int_{\R_t} \partial_x^{-1}w_{N_1} w_{\sim N} \Lambda^{2\sigma+2-\alpha}z_{N_3} z_{\sim N} 
  \end{align*}
    and where $\Lambda^{2\sigma+2-\alpha}$ is the Fourier multiplier by $\frac{\xi \widetilde{\nu}_\sigma(\xi)}{\omega_{\alpha+1}'(\xi)}$ and the multipliers $\widetilde{\chi}_j$ satisfy $\|\widetilde{\chi}_1\|_{L^\infty}\lesssim N_1^{2\sigma}N_3^{1-\alpha}$ and $\|\widetilde{\chi}_2\|_{L^\infty}\lesssim N_1^{-1}N_3N^{2\sigma+1-\alpha}$ on the support of the integral in $\widetilde{\mathcal{K}}_{N,C_3}^{high,j}$. To estimate $\widetilde{\mathcal{K}}_{N,C_3}^{high,1}$, we apply \eqref{l2tri.1} with $\theta=1/2$ and deduce that
\begin{align*}
  |\widetilde{\mathcal{K}}_{N,C_3}^{high,1}| &\lesssim \sum_{N_1\ll N_3\ll N \atop N_3 \gtrsim 1} (N_1N_3)^{\frac 12} N_3^{-\frac 12}N^{\frac{1-\alpha}2} N_1^{2\sigma}N_3^{1-\alpha} \|w_{N_1}\|_{Z^0} \|w_{\sim N}\|_{Z^0} \|z_{N_3}\|_{Y^0} \|z_{\sim N}\|_{Y^0}\\
  &\lesssim \sum_{N_1\ll N_3\ll N \atop N_3 \gtrsim 1} \frac{N_1}{\cro{N_1}} N_1^{2\sigma+\frac 12} \cro{N_1}^{-\sigma}N_3^{-s+1-\alpha}N^{-s-\sigma+\frac{1-\alpha}2} \|w_{N_1}\|_{\overline{Z}^\sigma} \|w\|_{Z^\sigma} \|z\|_{Y^s}^2,
\end{align*}
which is acceptable since   $-1/2<\sigma<s-\frac{3}{2}+\alpha $ and $s+\sigma > \frac 12(1-\alpha)$.

To deal with $\widetilde{\mathcal{K}}_{N,C_3}^{high,2}$, we argue as in the region $C$ in the proof of Proposition \ref{prop:l2tri} and apply \eqref{l2tri.1} with $\theta=\min\left\{1, \left(\frac{\alpha}{1-\alpha}\right)_-\right\}\in ]0,1]$, to bound $|\widetilde{\mathcal{K}}_{N,C_3}^{high,2}|$ by
\begin{align*}
 \sum_{N_1\ll N_3 \ll N \atop N_3 \gtrsim 1}& (N_1N_3)^{\frac 12} N_3^{-\theta}N^{\theta(1-\alpha)} N_1^{-1}N_3 N^{2\sigma+1-\alpha} \|w_{N_1}\|_{Z^0} \|w_{\sim N}\|_{Z^0} \|z_{N_3}\|_{Y^0} \|z_{\sim N}\|_{Y^0}\\
  &\lesssim \sum_{1\lesssim N_3\ll N}  N_3^{-s + \frac 32-\theta} N^{\sigma-s+(1+\theta)(1-\alpha)} \|w\|_{\overline{Z}^\sigma}^2 \|z\|_{Y^s}^2\\
  &\lesssim N^{0_-} \|w\|_{\overline{Z}^\sigma}^2 \|z\|_{Y^s}^2 ,
\end{align*}
where we used that $\sigma<s-\frac32 +\alpha$ and $s>1-\alpha$ if $\alpha>\frac 12$ and $s>\frac 32-\theta$ in the case $\alpha< \frac 12$.

Finally, the term $\widetilde{\mathcal{K}}_{N,C_3}^{high,3}$ is bounded with the help of Proposition \ref{prop:l2tri.2} and Bernstein inequality by
\begin{align*}
  |\widetilde{\mathcal{K}}_{N,C_3}^{high,3}| &\lesssim \sum_{N_1\ll N_3\ll N\atop N_3\gtrsim 1} N_1^{\frac 12} N_3^{-\frac \alpha 4} N^{\frac{1-\alpha}2} \|\partial_x^{-1}w_{N_1}\|_{Z^0} \|w_{\sim N}\|_{Z^0} \|\Lambda^{2\sigma+2-\alpha}z_{N_3}\|_{Y^0} \|z_{\sim N}\|_{Y^0} \\
  &\lesssim \sum_{1\lesssim N_3\ll N} N_3^{2\sigma-s+2-\frac{5\alpha}4} N^{-s-\sigma+\frac{1-\alpha}2} \|w\|_{\overline{Z}^\sigma}^2 \|z\|_{Y^s}^2,
\end{align*}
which is acceptable since $s+\sigma > \frac{1-\alpha}2$ and $\sigma-2s+\frac 52-\frac{7\alpha}4 < -s+1-\frac{3\alpha}4<0$.

\vskip .3cm 
\noindent
\textbf{Subregion $C_4$:} $N_3\lesssim N_1 \ll N_2\sim N_4=N$. \\
In this region $M_{min}=M_2$, so that $\Omega =M_2N^{\alpha}$ and $\|\widetilde{\chi}\|_{L^{\infty}}\lesssim N_1^{2\sigma+1-\alpha}$ in view of \eqref{a4tilde-C4}. By using \eqref{l2tri-basic}, we first note that
\begin{align*}
  |\widetilde{\mathcal{K}}_{N,C_4}^{low}| &\lesssim \sum_{N_3\lesssim N_1\ll N}\sum_{M_2\ll 1} M_2 N_1^{2\sigma+1-\alpha} \|w_{N_1}\|_{Z^0} \|w_{\sim N}\|_{Z^0} \|z_{N_3}\|_{Y^0} \|z_{\sim N}\|_{Y^0} \\
  &\lesssim \sum_{N_1\ll N} N_1^{2\sigma+1-\alpha} N^{-s-\sigma} \|w_{N_1}\|_{Z^0} \|w\|_{Z^\sigma} \|z\|_{Y^s}^2\\
  &\lesssim N^{0_-} \|w\|_{\overline{Z}^\sigma}^2 \|z\|_{Y^s}^2,
\end{align*}
where we used that $2\sigma+1-\alpha>-1$ if $N_1\lesssim 1$, and $s>1-\alpha$ otherwise.
Finally, we apply estimate \eqref{l2tri.1} with $\theta=1/2$ combined with \eqref{a4tilde-C4} to bound $|\widetilde{\mathcal{K}}_{N,C_4}^{high}(t)|$ by
\begin{align*}
   \sum_{N_3\le N_1\ll N}&\sum_{1\lesssim M_2\lesssim N_1}  M_2^{\frac 12}N_3^{\frac 12} M_2^{-\frac 12}N^{\frac{1-\alpha}2} N_1^{2\sigma+1-\alpha} \|w_{N_1}\|_{Z^0} \|w_{\sim N}\|_{Z^0} \|z_{N_3}\|_{Y^0} \|z_{\sim N}\|_{Y^0}\\
   &\lesssim N^{0_+} \sum_{N_3\le N_1\ll N\atop N_1\gtrsim 1} N_3^{\frac 12}\cro{N_3}^{-s} N_1^{\sigma+1-\alpha} N^{-s-\sigma+\frac{1-\alpha}2}
    \|w_{N_1}\|_{Z^{\sigma}} \|w\|_{Z^\sigma} \|z\|_{Y^s}^2 .
\end{align*}
Since $s+\sigma> \frac{1-\alpha}2$, this sum is bounded by
$$
    N^{0_+}\sum_{N_3\le N_1,\atop N_1\gtrsim 1} N_3^{\frac 12}\cro{N_3}^{-s} N_1^{-s+\frac 32(1-\alpha)} \|w\|_{Z^\sigma}^2 \|z\|_{Y^s}^2, 
$$
which is acceptable for $s>\max\{\frac 32(1-\alpha), 1-\frac{3\alpha}4\}$. This completes the proof of \eqref{est:Esigma.5}.

\vskip .5cm
\noindent
\textbf{Estimate of $\widetilde{\mathcal{L}}$.}\\ 
Thanks to Plancherel identity, $\widetilde{\mathcal{L}}$ may be written as
$$
\widetilde{\mathcal{L}}(t) = \frac i2 \int_{\Gamma^4_t} \widetilde{\widetilde{a}}_4(\xi_1,\xi_2) \prod_{i=1}^4 \widehat{w}(\xi_i)
$$
with a symbol
\begin{equation}\label{a44}
\widetilde{\widetilde{a}}_4(\xi_1,\xi_2) = (\xi_1+\xi_2)\widetilde{b}_3(\xi_1,\xi_2).
\end{equation}
This term enjoys more symmetries than $\widetilde{\mathcal{K}_1}+\widetilde{\mathcal{K}_2}$, and the corresponding estimates will be very similar to the proof of Proposition \ref{prop:Es}.
Performing a dyadic decomposition, we get $\widetilde{\mathcal{L}} = \sum_{N\gtrsim N_0} \widetilde{\mathcal{L}}_N$ where
$$
\widetilde{\mathcal{L}}_N(t) = \sum_{N_1, N_2, N_3, N_4 \atop N_{max}=N} \sum_{M_1,M_2,M_3>0} \int_{\R_t} \Pi_{\widetilde{\widetilde{\chi}}}(w_{N_1}, w_{N_2}, w_{N_3})w_{N_4}
$$
with
$$
\widetilde{\widetilde{\chi}}(\xi_1,\xi_2,\xi_3) = \frac i2\phi_{M_1}(\xi_2+\xi_3)\phi_{M_2}(\xi_1+\xi_3)\phi_{M_3}(\xi_1+\xi_2) \widetilde{\widetilde{a}}_4(\xi_1,\xi_2).
$$
Moreover, as a consequence of Lemma \ref{lem:b3tilde}, we get that
\begin{equation}\label{chitt}
\|\widetilde{\widetilde{\chi}}\|_{L^\infty} \lesssim N_{min}^{-1}N^{2\sigma+2-\alpha}.
\end{equation}
Again, we separate the low and high modulation frequencies corresponding respectively to $\Omega:= M_{min}M_{med}N^{\alpha-1}\ll N^\alpha$ and $\Omega \gtrsim N^{\alpha}$, which leads to the decomposition $\widetilde{\mathcal{L}}_N = \widetilde{\mathcal{L}}_N^{low} + \widetilde{\mathcal{L}}_N^{high}$.

From estimates \eqref{l2tri-basic} and \eqref{chitt} we infer
$$
|\widetilde{\mathcal{L}}_N^{low}(t)| \lesssim \sum_{N_1, N_2, N_3, N_4 \atop N_{max}=N} \sum_{M_1,M_2,M_3\atop M_{min}M_{med}\ll N} M_{min}^{1_-}(N_{min}N_{thd})^{0_+} N_{min}^{-1} N^{2\sigma+2-\alpha} \prod_{i=1}^4 \|w_{N_i}\|_{Z^0}.
$$
By symmetry, we reduce the sum to $N_1\lesssim N_2\lesssim N_3\sim N_4\sim N$. The contribution of the region $N_1\sim N_2\sim N$ is bounded by
$$
N^{0_+} N^{(\frac 12)_-} N^{2\sigma+1-\alpha} \|w_{\sim N}\|_{Z^0}^4 \lesssim N^{0_-} \|w\|_{Z^\sigma}^2 \|w\|_{Y^s}^2.
$$
In the case $N_1\sim N_2\ll N$, it holds $M_{min}\ll 1$ and thus the contribution is bounded by
$$
\sum_{N_1\sim N_2\ll N} N_1^{(-\frac 12)_+} N_2^{(-\frac 12)_+} N^{(2\sigma-2s+2-\alpha)_+} \|w_{N_1}\|_{Z^0} \|w_{N_2}\|_{Z^0} \|w\|_{Y^s}^2 \lesssim N^{0_-}  \|w\|_{\overline{Z}^\sigma}^2 \|w\|_{Y^s}^2,
$$
since $2\sigma-2s+2-\alpha<0$.
For the region $N_1\ll N_2\ll N$, we have $N_2\sim M_{min}\ll 1$ and the contribution for this case is estimated by
$$
\sum_{N_1\ll N_2\ll 1} N_1^{(-1)_+} N_2^{0_+} \|w_{N_1}\|_{Z^0} \|w_{N_2}\|_{Z^0} \|w\|_{Y^s}^2 \lesssim N^{0_-}  \|w\|_{\overline{Z}^\sigma}^2 \|w\|_{Y^s}^2. 
$$
Gathering the previous estimates we conclude
\begin{equation}\label{est:Esigma.7}
|\widetilde{\mathcal{L}}_N^{low}(t)| \lesssim N^{0_-} \|w\|_{\overline{Z}^\sigma}^2 \|w\|_{Y^s}^2 \lesssim N^{0_-} \|w\|_{\overline{Z}^\sigma}^2 (\|u\|_{Y^s}^2 + \|v\|_{Y^s}^2).
\end{equation}
Let us now tackle the high modulation term. From \eqref{l2tri.2} it holds that $|\widetilde{\mathcal{L}}_N^{high}(t)|$ is bounded by
\begin{equation}\label{est:Esigma.8}
 \sum_{N_1, N_2, N_3, N_4 \atop N_{max}=N} \sum_{M_1,M_2,M_3\atop M_{min}M_{med}\gtrsim N} \inf_{0\le \theta\le 1} P (N\Omega^{-1})^\theta N_{min}^{-1} N^{2\sigma+2-\alpha} \prod_{i=1}^4 \|w_{N_i}\|_{Z^0}
\end{equation}
where $P=\min\{M_{min}, (N_{min}N_{thd})^{\frac 12}\}$. As previously, we may assume $N_1\lesssim N_2\lesssim N_3\sim N_4\sim N$. In the region $N_1\sim N_2\sim N$, we have 
$$
\eqref{est:Esigma.8} \lesssim N^{0_+} N^{-2s+2-\frac{3\alpha}2} \|w_{\sim N}\|_{Z^\sigma}^2 \|w_{\sim N}\|_{Y^0}^2 \lesssim N^{0_-} \|w\|_{Z^\sigma}^2 \|w\|_{Y^s}^2
$$
where we took $\theta=\frac 12$, and we used that $s>1-\frac{3\alpha}4$. 

In the case $N_1\ll N_2$, we have $M_{min}\sim N_2\gtrsim 1$ and, taking $\theta=1$,we infer 
\begin{align*}
\eqref{est:Esigma.8}  & \lesssim \sum_{N_1\ll N_2\lesssim N} (N_1N_2)^{\frac 12} N_2^{-1} N^{1-\alpha} N_1^{-1} N^{2\sigma-2s+2-\alpha} \|w_{N_1}\|_{Z^0} \|w_{N_2}\|_{Z^0} \|w_{\sim N}\|_{Y^s}^2 \\
  & \lesssim \sum_{N_1\ll N_2\lesssim N} N_1^{-\frac 12} N_2^{-\frac 12} N^{2\sigma-2s+3-2\alpha} \|w_{N_1}\|_{Z^0} \|w_{N_2}\|_{Z^0} \|w\|_{Y^s}^2\\
  & \lesssim N^{0_-} \|w\|_{\overline{Z}^\sigma}^2 \|w\|_{Y^s}^2,
\end{align*}
since $2\sigma-2s+3-2\alpha < 0$. 

Finally, the contribution of the region $N_1\sim N_2\ll N$ is bounded by 
\begin{align*}
 \eqref{est:Esigma.8}  & \lesssim  \sum_{N_1\ll N} \sum_{M_3\lesssim N_1} M_3 (M_3^{-1}N^{1-\alpha})^{1_-} N_1^{-1} N^{2\sigma-2s+2-\alpha} \|w_{\sim N_1}\|_{Z^0}^2 \|w_{\sim N}\|_{Y^s}^2 \\
  &\quad \lesssim N^{0_-} \sum_{N_1\ll N} \left(N_1^{(-\frac 12)_+} \|w_{\sim N_1}\|_{Z^0}\right)^2 \|w\|_{Y^s}^2\\
  &\quad \lesssim N^{0_-} \|w\|_{\overline{Z}^\sigma}^2 \|w\|_{Y^s}^2.
\end{align*}
It follows from these estimates that
$$
|\widetilde{\mathcal{L}}_N^{high}(t)| \lesssim N^{0_-} \|w\|_{\overline{Z}^\sigma}^2 \|w\|_{Y^s}^2 \lesssim N^{0_-} \|w\|_{\overline{Z}^\sigma}^2 (\|u\|_{Y^s}^2 + \|v\|_{Y^s}^2).
$$
Combined with \eqref{est:Esigma.7} this proves \eqref{est:Esigma.6}.
\end{proof}

\section{Proof of Theorem \ref{th:main} and Corollary \ref{coro:gwp}} \label{Sec:6}

\subsection{Lipschitz bound, uniqueness and unconditional uniqueness} \label{uniqueness}

Let $s>s_\alpha$, $0<T\le 1$ and $u$, $v \in Y^s_T$ be two solutions of \eqref{fKdV} on $[0,T]$ emanating from the initial data $u_0$, $v_0 \in H^s$ respectively and satisfying $u_0-v_0 \in \overline{L}^2$.  We fix $\sigma$ as in \eqref{cond:sigma} and set $w=u-v$. Thus, $w$ is a solution to \eqref{eqdiff} emanating from $w_0=u_0-v_0 \in \overline{H}^{\sigma}$ and the continuous embedding from $Y^s_T$ into $Z^s_T$ ensures that $w \in Z^\sigma_T$. Moreover, arguing as in Subsection 6.1 in \cite{MPV}, we find that $w\in \overline{Z}^\sigma_T$.

Now, we deduce combining Corollary \ref{prop:zst}, Lemma \ref{coer:E_tilde_sigma}, estimate \eqref{est:ytilde} and Proposition \ref{prop:E_tilde_sigma} that there exists $0<\gamma=\gamma(s,\sigma,\alpha)<1$ such that, for any $N_0 \gg 1$, 
\begin{equation*}
 \|w\|_{L^\infty_T\overline{H}^\sigma_x}^2 \lesssim  \|w_0\|_{\overline{H}^\sigma }+ \left(TN_0^{\frac 32}+N_0^{-\gamma}\right) \left( 1+\|u\|_{Y^s_T}+\|v\|_{Y^s_T}\right)^{16}  \|w\|_{L^\infty_T\overline{H}^\sigma_x}^2  .
\end{equation*}
Therefore, by choosing $N_0\gg  \left( 1+\|u\|_{Y^s_T}+\|v\|_{Y^s_T}\right)^{\frac{16}{\gamma}}$, we deduce that
\begin{equation} \label{Lip}
\|w\|_{L^\infty_{T'}\overline{H}^\sigma_x} \lesssim \|w_0\|_{\overline{H}^\sigma } ,
\end{equation}
for $T'=\min\left\{ T,c\left( 1+\|u\|_{Y^s_T}+\|v\|_{Y^s_T}\right)^{-16(\frac3{2\gamma}+1)}\right\}$
where $c$ is a small positive constant. 

Thus, by choosing $u_0=v_0$, \eqref{Lip} ensures that $u\equiv 0$ on $(0,T')$. Moreover, $u$, $v \in Y^s_T \subseteq C([0,T] :H^s)$,  so that $u(T')=v(T')$. Therefore, we conclude by reapplying the above argument a finite number of times that $u\equiv v$ on $[0,T]$. 

Finally, we prove the unconditional uniqueness result. Observe from Proposition \ref{propse} that 
\begin{equation} \label{embYsT}
C([0,T]:H^s) \cap L^2((0,T) : L^\infty) \subseteq Y^s_T .
\end{equation}
Now, let $s>\max\{\frac12,s_\alpha\}$ and let $u , v \in C([0,T] : H^s)$ be two solution of \eqref{fKdV}. It follows from \eqref{embYsT} and the Sobolev embedding $H^s \subseteq L^\infty$ that $u, v \in Y^s_T$. Thus we conclude as previously that $u \equiv v$ on $[0,T]$, which concludes the proof of Corollary \ref{coro:uu}.

\subsection{\emph{A priori} estimate} In this subsection, we derive \emph{a priori} estimates for smooth solutions in $H^s$, for $s>\widetilde{s}_\alpha$. When $s \ge 3$, the result follows by classical arguments. By arguing as in \cite{ABFS89}, we deduce the following result which will be our starting point. 
\begin{proposition} \label{propo:smooth}
Let $s \ge 3$. Then, for any $u_0 \in H^s$, there exists a positive time $T=T(\|u_0\|_{H^s})$ and a unique solution $u$ of \eqref{fKdV} in $C([0,T] : H^s)$ such that $u(0,\cdot)=u_0$. Moreover, for any fixed $R>0$, the flow map $u_0 \mapsto u$ is continuous form the ball of $H^s$ of radius $R$ centred at the origin into $C([0,T(R)) : H^s)$.  
\end{proposition}

Let $u_0 \in H^{\infty}$. From the above result, $u_0$ gives rise to a solution $u \in C([0,T^\star) : H^\infty)$ of \eqref{fKdV} defined on its maximal time of existence $[0,T^{\star})$ with $T^\star \ge T(\|u_0\|_{H^3})$. Moreover,  the blow-up alternative
\begin{equation} \label{BUA}
\text{if} \ T^\star<+\infty, \quad \text{then} \quad \lim_{t \nearrow T^\star} \|u(t) \|_{H^3} =+\infty ,
\end{equation}
follows from the fact that $T=T(\|u_0\|_{H^s})$ in Proposition \ref{propo:smooth} can be chosen as a non-increasing function of its argument. 

By using a bootstrap argument, we prove that the solution satisfies a suitable \emph{a priori} estimate
on a positive time interval depending only on the $H^s$-norm of the initial datum.

\begin{lemma} \label{apriori:smooth}
 Let $\widetilde{s}_{\alpha}<s' \le s \le 3$. Then, there exist $A_0>0$ and $\beta_0>0$ such that $T^\star \ge A_0 (1+\|u_0\|_{H^{s'}})^{-\beta_0}$ and 
\begin{equation} \label{apriori:smooth.1}
\|u\|_{Y^{s}_T} \le 4C_0\|u_0\|_{H^{s}} \quad \text{with} \quad T= A_0 (1+\|u_0\|_{H^{s'}})^{-\beta_0}. 
\end{equation}
\end{lemma}

\begin{proof} Let $\widetilde{s}_{\alpha}<{s'} \le 3$. We define 
\begin{equation} \label{BS}
T_0=\sup \left\{ T \in (0,T^\star) : \|u\|_{Y^{s'}_T} \le 4C_0\|u_0\|_{H^{s'}} \right\} .
\end{equation}
Observe that the above set is non empty since $u \in  C([0,T^\star) : H^\infty)$, so that $T_0$ is well-defined. We argue by contradiction and assume that 
\begin{equation} \label{def:A0}
0<T_0< A_0 (1+\|u_0\|_{H^{s'}})^{-\beta_0}, 
\end{equation}
where $A_0>0$ and $\beta_0>0$ will be fixed below. 

Let $0<T_1<T_0$. We have from the definition of $T_0$ that $\|u\|_{Y_{T_1}^{s'}} \le 4C_0\|u_0\|_{H^{s'}}$. Then, it follows from Corollary \ref{prop:yst}, Lemma \ref{lemma:coer} and Proposition \ref{prop:Es} that for $s' \le s \le 3$, $N_0 \ge C_1 \left(1+4C_0\|u_0\|_{H^{s'}} \right)^{\frac2{\alpha}}$ and $T_1 \le a_0 (1+4C_0\|u_0\|_{H^{s'}})^{-\frac1\kappa}$,
\begin{equation} \label{aprio:est:YT1}
\|u\|_{Y_{T_1}^{s}}^2 \le 3C_0^2\|u_0\|_{H^{s}}^2+4C_0^2C_2T_1N_0^{\frac32}\|u_0\|_{L^2}\|u\|_{L^\infty_{T_1}H^{s}_x}^2 +
2^6C_2C_0^6N_0^{-\gamma} \|u_0\|_{H^{s'}}^2\|u\|_{Y_{T_1}^{s}}^2 .
\end{equation}
Now, we set
\[N_0=\max\left\{C_1 \left(1+4C_0\|u_0\|_{H^{s'}} \right)^{\frac2{\alpha}},(2^8C_2C_0^6 \|u_0\|_{H^{s'}}^2)^{\frac1{\gamma}} \right\} \]
and we choose $A_0$ and $\beta_0$ in \eqref{def:A0} such that 
\[A_0 (1+\|u_0\|_{H^{s'}})^{-\beta_0} < \min\left\{ a_0 (1+4C_0\|u_0\|_{H^{s'}})^{-\frac1\kappa}, (2^4C_0^2C_2N_0^\frac32 \|u_0\|_{L^2})^{-1}\right\} . \]
Thus, we deduce applying \eqref{aprio:est:YT1} with $s=3$ that 
\[\|u\|_{L^\infty_{T_1}H^{3}_x}^2 \le  \|u\|_{Y_{T_1}^{3}}^2 \le 6C_0^2 \|u_0\|_{H^{3}}^2, \quad 0<T_1<T_0 . \]
This implies in view of the blow-up alternative \eqref{BUA} that $T_0<T^\star$. 

Now, estimate \eqref{aprio:est:YT1} at the level $s=s'$ yields $\|u\|_{Y_{T_1}^{s'}}^2 \le 6C_0^2 \|u_0\|_{H^{s'}}^2$, so that by continuity $\|u\|_{Y_{T_2}^{s'}}^2 \le 8C_0^2 \|u_0\|_{H^{s'}}^2$ for some $T_0<T_2<T^\star$, which contradicts the definition of $T_0$. 

Therefore, $T_0 \ge T:=A_0 (1+\|u_0\|_{H^{s'}})^{-\beta_0}$. Now, let $s' \le s \le 3$. We use the estimate \eqref{aprio:est:YT1} at the regularity $s$ and argue as above to get the bound \eqref{apriori:smooth.1}. 
\end{proof}

As a consequence of Corollary \ref{coro:Eshigh} and Lemma \ref{apriori:smooth}, we deduce a uniform decay of the $H^s$-norm  of the high frequencies of $ u $ on  $ [0,T] $.

\begin{corollary} \label{apriori:high}
Let $\widetilde{s}_{\alpha}<s' \le s \le 3$ and let $T= A_0 (1+\|u_0\|_{H^{s'}})^{-\beta_0}$ as in Lemma \ref{apriori:smooth}. Then, for $N_0\ge  C_1 (1+\|u\|_{L^\infty_TH^{\widetilde{s}_\alpha}})^{\frac 2 \alpha}$,
\begin{equation} \label{apriori:high.1}
\|D^s_x\Pi_{>N_0}u(t)\|_{L^\infty_TL^2_x}^2 \lesssim \|D^s_x\Pi_{>N_0}u_0\|_{L^2}^2+N_0^{-\gamma}\|u_0\|_{H^{s}}^4 ,
\end{equation}
where $\gamma>0$ is defined in Corollary \ref{coro:Eshigh}.
\end{corollary}

\subsection{Local existence in $H^s$} Let $s_\alpha<s <3$, $\widetilde{s}_\alpha<s'\le s$ and let $u_0 \in H^s$. We set $u_{0,n}=\Pi_{\le 2^n}u_0$, where we recall the definition $\Pi_{ \le N}u=\mathcal{F}^{-1}_x\big({\bf 1}_{|\xi| \le N}\mathcal{F}_xu\big)$.  Then, $u_{0,n} \in H^\infty$ and we denote by $u_n \in C([0,T_n^\star): H^\infty)$ the solution of \eqref{fKdV} emanating from $u_{0,n}$ and defined on its maximal time of existence $[0,T_n^\star)$.  

Then, it follows from Lemma \ref{apriori:smooth} that
\begin{equation} \label{def:T}
T_n^\star \ge T:=A_0 (1+\|u_0\|_{H^{s'}})^{-\beta_0}
\end{equation} 
and
\begin{equation} \label{bound:Ys}
\|u_n\|_{Y^{s}_T} \lesssim \|u_{0,n}\|_{H^{s}} \lesssim  \|u_{0}\|_{H^{s}}.
\end{equation}

Now, let $\sigma$ as in \eqref{cond:sigma} and let $n \ge m \ge 0$, so that $u_{0,n}-u_{0,m} \in \overline{L}^2$. Hence, we deduce from \eqref{Lip} that there exist positive constants $A_1$ and $\beta_1$ such that
\begin{equation} \label{Lip:un}
\|u_n-u_m\|_{L^\infty_{T'}\overline{H}^\sigma_x} \lesssim \|u_{0,n}-u_{0,m}\|_{\overline{H}^\sigma } \lesssim \|\Pi_{>2^m}u_0\|_{H^s } ,
\end{equation}
where 
\begin{equation}  \label{def:Tprime}
T':=A_1 (1+\|u_0\|_{H^s})^{-\beta_1}<T .
\end{equation} 
Thus, we infer combining \eqref{apriori:high.1} and \eqref{Lip:un} 
\begin{align} 
\|u_n&-u_m\|_{L^\infty_{T'}H^s_x}  \nonumber
\\ &\lesssim N_0^{s-\sigma}\|\Pi_{\le N_0}(u_n-u_m)\|_{L^\infty_{T'}H^\sigma_x}+\|\Pi_{>N_0}u_n\|_{L^\infty_{T'}H^s_x}+\|\Pi_{>N_0}u_m\|_{L^\infty_{T'}H^s_x} \nonumber
\\ &\lesssim N_0^{s-\sigma}\|\Pi_{>2^m}u_0\|_{H^s }+\|\Pi_{>N_0}u_0\|_{H^s}+N_0^{-\frac{\gamma}2}\|u_0\|_{H^{s}}^2  \label{persistence}
\end{align}
where $N_0\ge  C_1 (1+\|u\|_{L^\infty_TH^{\widetilde{s}_\alpha}})^{\frac 2 \alpha}$. This proves that $\{u_n\}$ is a Cauchy sequence in $C([0,T'] : H^s)$ and thus converges to a function $u$ in $C([0,T'] : H^s)$. Moreover, it follows from \eqref{bound:Ys} that $u \in Y^s_{T'}$ and we deduce that $u$ is a solution of \eqref{fKdV} on $[0,T']$ satisfying $u(\cdot,0)=u_0$. 

Finally, by repeating this argument a finite number of times, we deduce that $u \in Y^s_T$ and  $\{u_n\}$ converges to $u$ in  $C([0,T] : H^s)$ where $T$ is defined in \eqref{def:T}. 

\subsection{Continuity of the flow map} \label{64}
Let $\{u_{0}^j\} \subset B_{H^s}(0,2\|u_0\|_{H^s})$ be a sequence converging to $u_0$ in $H^s$. We denote by $u^j$ and $u_n^j$ the solutions to \eqref{fKdV} emanating respectively from $u_0^j$ and $\Pi_{\le 2^n}u_0^j$. 

On one hand, noticing that 
\begin{equation} \label{unif:conv}
\lim_{n \to +\infty} \sup_{j \in \mathbb N}\|\Pi_{>2^n}(u_0^j)\|_{H^s}=0 ,
\end{equation}
we infer from \eqref{persistence} that
\begin{equation} \label{unif:conv.2}
\lim_{n \to +\infty} \sup_{j \in \mathbb N}\|u_n^j-u^j\|_{L^\infty_{T'}H^s_x}=0 ,
\end{equation}
where $T'$ is defined in \eqref{def:Tprime}. 

Let $\varepsilon>0$. Observe from the triangle inequality that
\begin{equation*} 
\|u^j-u\|_{L^\infty_{T'}H^s_x} \le \|u^j-u_n^j\|_{L^\infty_{T'}H^s_x}+\|u_n^j-u_n\|_{L^\infty_{T'}H^s_x} + \|u_n-u\|_{L^\infty_{T'}H^s_x}. 
\end{equation*}
Thus, we deduce from \eqref{unif:conv.2} that there exists $n_0 \in \mathbb N$ such that, for any $j \in \mathbb N$, 
\begin{equation} \label{triangle}
\|u^j-u\|_{L^\infty_{T'}H^s_x} \leq \epsilon+\|u_{n_0}^j-u_{n_0}\|_{L^\infty_{T'}H^s_x} . 
\end{equation}
 
 Now, using $\|\Pi_{>2^{n_0}}(u_0^j-u_0)\|_{H^3} \lesssim 2^{n_0(3-s)} \|\Pi_{>2^{n_0}}(u_0^j-u_0)\|_{H^s}$, the continuity of the flow in $H^3$  (see Proposition \ref{propo:smooth}) ensures the existence of $j_0 \in \mathbb N$ such that, for any $j \ge j_0$,
 \begin{equation} \label{flowH3}
\|u_{n_0}^j-u_{n_0}\|_{L^\infty_{T'}H^s_x} \leq \epsilon . 
\end{equation}
Note that \eqref{flowH3} holds on the time interval $[0,T']$, since the well-posedness time in Proposition \ref{propo:smooth} only depends on $\|u_0\|_{H^s}$ from the uniqueness result in Subsection \ref{uniqueness} and the \emph{a priori} estimate in Lemma \ref{apriori:smooth}. 

Thus, combining \eqref{unif:conv.2} and \eqref{flowH3}, $\{u^j\}$ converges to $u$ in $C([0,T']:H^s)$. Iterating this process a finite number of time, we conclude that $\{u^j\}$ converges to $u$ in $C([0,\widetilde{T}]:H^s)$ where $\widetilde{T}:=A_0 (1+2\|u_0\|_{H^{s'}})^{-\beta_0}$.

\medskip 
This finishes the proof of Theorem \ref{th:main}. 

\subsection{Global well-posedness} In this subsection we prove Theorem \ref{coro:gwp}. 

Let $\alpha \in (\frac23,1]$ and $s>s_\alpha$. Observe that in this case $\widetilde{s}_\alpha <\frac{\alpha}2$. We deduce by applying Theorem \ref{th:main} with $s'=\frac{\alpha}2$ that there exist $A_0>0$ and $\beta_0>0$ such that the IVP associated to \eqref{fKdV} is locally well-posed in $H^s$ on a time interval $[0,T]$ with $T:=A_0 (1+\|u_0\|_{H^{\frac\alpha2}})^{-\beta_0}$. 

On the other hand, it follows by combining the conserved quantities \eqref{Mass} and \eqref{Energy-general}, with the Galgiardo-Nirenberg inequality that  $\|u(t)\|_{H^{\frac\alpha2}} \le C(\|u_0\|_{H^{\frac\alpha2}})$, for all $t \in [0,T]$. Thus, we conclude by iterating the local well-posedness result an arbitrarily large but finite number of time that the IVP associated to \eqref{fKdV} is globally well-posed in $H^s$, which proves (i). 

The proof of (ii) follows directly from the fact that $s_{\alpha}<\frac{\alpha}2$ for $\alpha \in (\frac45,1]$

\section{The periodic case} \label{Sec:7}
In this section we explain the main changes in the periodic setting. Let us first define our function space in this setting. Let $ \T $ be the one dimensional torus. 
First we define our function spaces in the periodic setting. 
For  a $ 2\pi $-periodic function $ \varphi$, we define its space Fourier transform on
  $ \Z$ by
$$
\hat{\varphi}(k)=\int_{\lambda \T} e^{-i \xi x} \, f(x)
\, dx , \quad \forall k \in \Z \quad .
$$
The Lebesgue spaces $ L^q(\T) $, $ 1\le q \le \infty $,  for $ 2\pi$-periodic functions, will be defined as usually by
$$
\|\varphi\|_{L^q}=\Bigl( \int_{\T} |\varphi(x) |^q \, dx\Bigr)^{1/q}
$$
with the obvious modification for $ q=\infty $. \\
The Sobolev spaces $ H^s(\T)$ for $ 2\pi$-periodic functions are endowed with the norm
$$
\|\varphi\|_{H^s}=\|\langle k \rangle^{s}\widehat{\varphi}(k)
\|_{l^2_k} =\|J^s_x \varphi \|_{L^2} \quad ,
$$
where $ \langle \cdot \rangle = (1+|\cdot|^2)^{1/2} $ and $
\widehat{J^s_x \varphi}(k)=\langle k \rangle^{s}
\widehat{\varphi}(k)
$. \\
In the same way, for a function $ u(t,x) $ on $\R\times \T $, we define its space-time Fourier transform by
$$
\hat{u}(\tau,k)={\mathcal F}_{t,x}(u)(\tau,\xi)=\sum_{k\in\Z}  \int_{ \T} e^{-i (\tau
t+ k x)} \, u(t,x) \, dx dt , \quad \forall (\tau,k) \in
\R\times\Z \quad .
$$
For any $ (s,b) \in \R^2 $, we define the  Bourgain space $ X^{s,b} $,  of $ 2\pi$-periodic (in $x$) functions
  as the completion of $ {\mathcal S}(\T\times \R) $ for the norm
\begin{eqnarray*}
\| u \|_{X^{s,b}}  =
 \| \langle \tau-p_{\alpha+1}(k)\rangle^{b}  \langle k \rangle^s
  \hat{u}\|_{L^2_{\tau} l^2_k} \, .
\end{eqnarray*}
Finally, we define the function $ \phi_N $ and $ \psi_L $
 and  the Fourier multipliers $ P_N $ and $ Q_L $ as in Subsection \ref{notation}.  In the periodic setting it holds
 \begin{equation}\label{mm}
u= P_0 u +\sum_{N>0} P_N u
 \end{equation}
 where $ \widehat{P_0 u}(\xi)=\widehat{u}(0) $.

In order to extend our results to the periodic case, the main difficulty is to obtain suitable refined Strichartz estimates, similar to those of Proposition \ref{propse} in the real line case. Recall that the main tool in the proof of this result is the classical linear Strichartz estimate
\begin{equation}\label{linear-stri}
  \|D_x^{\frac{\alpha-1}4}U_\alpha(t)u_0\|_{L^4_t L^\infty_x} \lesssim \|u_0\|_{L^2_x},
\end{equation}
which turns out to be false in the periodic settings. However, a careful reading of the proof of Proposition \ref{propse} in \cite{MPV} shows that estimate \eqref{linear-stri} is only needed for short time intervals of length $N^{-1}$ when the data $u_0$ have spatial frequencies localized around the dyadic number $N$. Fortunately, such an estimate has been proved in the periodic case by Kenig, Ponce and Vega in Theorem 5.2 of \cite{KPV} for the case where $\alpha \in \mathbb Z_+$, $\alpha>1$. We observed however that their proof adapts well to our general setting.  
\begin{proposition}\label{prop:per-stri}
  Let $\alpha>0$, $u_0\in L^2(\T)$ and $N\ge \xi_0$ be a dyadic number, where $\xi_0$ is defined in Hypothesis \ref{hyp1}. Then for any time interval $I$ satisfying $|I|\lesssim N^{-\alpha}$, it holds
  \begin{equation}\label{per-stri}
    \|P_ND_x^{\frac{\alpha-1}4}U_\alpha(t)u_0\|_{L^4_I L^\infty_x} \lesssim \|u_0\|_{L^2_x}.
  \end{equation}
\end{proposition}

As mentioned above, the proof of Proposition \ref{linear-stri} follows closely the one of Theorem 5.2 in \cite{KPV}.   For the convenience of the reader, we give the details below. 
We will need the following lemma due to Zygmund.
\begin{lemma}[\cite{Z}, p. 198]\label{lem:Zyg}
  Let $f$ be a real-valued function defined on an interval $(a,b)$ such that $f'$ is monotone and $|f'|\le\frac 12$. Then, we have
  $$
  \left| \sum_{k\in (a,b)} e^{2\pi if(k)} - \int_a^b e^{2\pi i f(\xi)}d\xi\right|\lesssim 1,
  $$
  where the implicit constant is independent of $a<b$ and $f$. 
\end{lemma}
\begin{proof}[Proof of Proposition \ref{prop:per-stri}]
  Since $U_\alpha(t)$ is a unitary group in $L^2(\T$), we may always assume that $I=]0,T[$ with $T\lesssim N^{-\alpha}$. Using the classical duality P. Thomas argument, \eqref{per-stri} will follows from estimate
  $$
  \|P_ND_x^{\frac{\alpha-1}4}U_\alpha(t)u_0\|_{L^\infty_x} \lesssim t^{-\frac 12} \|u_0\|_{L^1}
  $$
  for all $t\in I$. From Young and Bernstein inequalities, it will suffice to prove that 
  $$
  \left|\sum_{k\sim N} e^{i(t\omega_{\alpha+1}(k)+kx)}\right| \lesssim N^{\frac{1-\alpha}2} t^{-\frac 12} 
  $$
 for any $t\in I$ and $x\in\T$. Setting $\varphi(\xi) = \frac 1{2\pi}(t\omega_{\alpha+1}(\xi)+x\xi)$, we get from Hypothesis \ref{hyp1} that $\varphi'$ is a continuous monotonic function on $D_N:=\{\xi\sim N\}$. Moreover, it holds that $|\varphi'(\xi)|\lesssim 1+tN^\alpha\lesssim 1$ for all $t\in I$, where we use the assumption that $0<t<T \lesssim N^{-\alpha}$. Therefore, there exist a subdivision $D_N=\cup_{j=1}^n [\xi_j, \xi_{j+1}]$ with $n \lesssim 1$  and integers $|p_1|,\ldots, |p_n|\lesssim 1$ such that $|\varphi'(\xi)-p_j|\le \frac 12$ for all $\xi \in  [\xi_j, \xi_{j+1}]$, $j=1,\ldots, n$. It follows then that
  $$
  \left|\sum_{k\sim N} e^{2\pi i \varphi(k)}\right| = \left| \sum_{j=1}^n \sum_{k\in (\xi_j,\xi_{j+1})} e^{2\pi i\varphi(k)}\right| \le n \sup_{1\le j\le n} \left| \sum_{k\in(\xi_j,\xi_{j+1})} e^{2\pi i(\varphi(k)-kp_j)}\right| .
  $$
  By virtue of Lemma \ref{lem:Zyg}, we infer
  $$
  \left|\sum_{k\sim N} e^{2\pi i \varphi(k)}\right| \lesssim 1+ \sup_{1\le j\le n} \left|\int_{\xi_j}^{\xi_{j+1}} e^{2\pi if(\xi)}d\xi\right|
  $$
  where $f(\xi) = \varphi(\xi)-p_j\xi$. Observing that $|f''(\xi)|\sim tN^{\alpha-1}$, we conclude thanks to Van der Corput lemma that
  $$
  \left|\sum_{k\sim N} e^{2\pi i \varphi(k)}\right| \lesssim 1 + (tN^{\alpha-1})^{-\frac 12} \lesssim N^{\frac{1-\alpha}2} t^{-\frac 12},
  $$
which finishes the proof of Proposition \ref{prop:per-stri}.   
\end{proof}

With this proposition in hand, the rest of the proof is similar to, and even simpler than, the case of the real line, since we do not need to worry about the low-frequency contributions. Indeed, as usually, we  notice that $ u $ is solution to \eqref{fKdV}  then its spatial mean value $ \int_{\T}\hspace*{-4mm} - \, u $ is a conserved quantity. Therefore $ u$ is a solution to \eqref{fKdV} if and only if $ \tilde{u}(t,x)=u(t,x- t \int_{\T}\hspace*{-4mm} - \,  u_0) - \int_{\T}\hspace*{-4mm} - \,  u_0  $ is a solution to  \eqref{fKdV}  with $ \Lambda_{\alpha+1} $ replaced by $ \tilde{\Lambda}_{\alpha+1}= \Lambda_{\alpha+1} -(\int_{\T}\hspace*{-4mm} - \, u_0) \partial_x $. Since $ \tilde{\Lambda}_{\alpha+1} $ satisfies Hypothesis \ref{hyp1} as soon as $\Lambda_{\alpha+1} $ does (with possibly a different $ \xi_0$), by this change of unknown, we can always reduce to initial data with mean value zero. We can thus avoid  to add a weight on the spatial low frequencies in the function space where we put the difference of two solutions. In this framework, outside resonances of order $ 4$ (points of cancellation of the resonance function $\Omega_4 $), we get the same estimates than those on the solutions and on the difference of solutions, having the same mean value,   obtained in Section \ref{Sec:4} and \ref{Sec:5}. 

 Now, according to Proposition \ref{est-Omega4} or direct calculations, for any $(k_1,k_2,k_3)\in \Z^3 $, 
$$
\Omega_4(k_1,k_2,k_3)=0 \Leftrightarrow (k_1+k_2)(k_1+k_3)(k_2+k_3)=0 \; .
$$
We thus only have to take care of these resonance points in the estimates on $ {\mathcal J}_N $ in Proposition \ref{prop:Es} and $\tilde{\mathcal K}_N,\; \tilde{\mathcal L}_N $ in Proposition \ref{prop:E_tilde_sigma}. It is worth noticing that, in the crucial estimates 
\eqref{l2tri-basic} and \eqref{l2tri.1}, it holds $P=1 $ on the set of resonance points.
 \vspace*{2mm} \\
{\bf $\bullet $ Estimate on $ {\mathcal J}_N $.} In view of the decomposition \eqref{a4}, $ a_4 $ does vanish as soon as $\Omega_4(k_1,k_2,k_3)=0$
 and thus the resonance points do not contribute to $ {\mathcal J}_N $.\\
 {\bf $\bullet $ Estimate on $ {\mathcal K}_N $.} In view of the decomposition \eqref{def:a4tilde}, $ \tilde{a}_4 $ does vanish as soon as $
 k_1+k_2=0 $. Since $k_1$ and $ k_2$ play a symmetric role in $ \tilde{a}_4 $, it thus suffices to consider  the set 
 $$
 {\mathcal R}:=\{(k_1,k_2,k_3)\in \Z^3, \;  k_1+k_3=0, \; k_1+k_2\neq 0 \}  \; .
 $$ According to  \eqref{def:a4tilde}, on ${\mathcal R}$ it holds 
 \begin{equation}\label{a4R}
\tilde{a}_4(k_1,k_2,k_3)=(k_2+k_3)\Bigl(\tilde{b}_3(k_2+k_3,k_1)-\tilde{b}_3(k_2+k_3,-k_2)\Bigr) -(k_1+k_2)\tilde{b}_3(k_1,k_2) \; . 
 \end{equation}
 Let us come back to  the restriction on the set ${\mathcal R} $ of the estimates on $ {\mathcal K}_N $ in the proof of Proposition \ref{prop:E_tilde_sigma}. We review one by one the different regions introduced in the proof of  Proposition \ref{prop:E_tilde_sigma}.\\
{\bf Region} A : Then, we proceed exactly as in    Proposition \ref{prop:E_tilde_sigma}.  We have $ \|\widetilde{\chi}\|_{L^\infty} \lesssim N^{2\sigma+1-\alpha}$ and thus it follows from \eqref{l2tri-basic}  (with $ P=1$) that
$$
|\widetilde{\mathcal{K}}_{N,A}^{\mathcal R}|\lesssim \sum_{0\le M_1,M_2\lesssim N \atop 0\le M_{min}\lesssim 1} N^{2\sigma+1-\alpha} \|w_{\sim N}\|_{Z^0}^2 \|z_{\sim N}\|_{Y^0}^2 \lesssim  N^{-2s+1-\alpha+0_+} \|w\|_{Z^\sigma}^2 \|z\|_{Y^s}^2,
$$
that is acceptable for $ s>\frac{1-\alpha}{2} $. \\
{\bf Region} B : There is no resonances in this region since $M_{min} \sim N $.\\
{\bf Region} C : In the subregions $ C_1 $ and $C_2 $, $\Omega_4=0 \Rightarrow k_1+k_2=0 $. Moreover, in 
 the subregion $ C_3 $, it holds $M_{min}\sim N_3\ge 1 $. Hence, these three subregions do not contribute. Now in the subregion $ C_4 $ where $ N_3\lesssim N_1\ll N_2 \sim N_4=N $, \eqref{a4R} ensures that $|\tilde{a}_4| \lesssim N^{2\sigma+1-\alpha} $.
 Therefore, by \eqref{l2tri-basic}  (with $ P=1$), we get 
 \begin{align*}
  |\widetilde{\mathcal{K}}_{N,C_4}^{\mathcal R}| &\lesssim \sum_{N_3\lesssim N_1\ll N}\sum_{M_2\lesssim  1} N^{2\sigma+1-\alpha} \|w_{N_1}\|_{Z^0} \|w_{\sim N}\|_{Z^0} \|z_{N_3}\|_{Y^0} \|z_{\sim N}\|_{Y^0} \\
  &\lesssim N^{-s+1-\alpha+0_+}\|w\|_{Z^\sigma}^2 \|z\|_{Y^s}^2,
\end{align*}
that is acceptable for $s>1-\alpha $.\\
{\bf $\bullet $ Estimate on $ \widetilde{{\mathcal L}}_N $.}
 We notice that, according to \eqref{a44}, $ \tilde{\tilde{a}}_4(k_1,k_2,k_3) =0 $ whenever $ k_1+k_2=0$ and thus, as above, we only have to care about the contribution of ${\mathcal R} $. Note that on ${\mathcal R} $ we must have $ N_1 \sim N_3 $ and $
 N_2 \sim N_4$. Therefore, in view of \eqref{chitt},  \eqref{l2tri-basic}  (with $ P=1$) leads to 
  \begin{align*}
  |\widetilde{\mathcal{L}}_{N}^{\mathcal R}| &\lesssim\sum_{N_1\sim N_3, N_2\sim N_4 \atop N_{max}=N} \sum_{1\le M_1\lesssim N, M_2\lesssim 1} N_{min}^{-1} N^{2\sigma+2-\alpha} \prod_{i=1}^4 \|w_{N_i}\|_{L^\infty_T L^2_x} \\
  &\lesssim\sum_{N_1\sim N_3, N_2\sim N_4 \atop N_{max}=N} \sum_{1\le M_1\lesssim N, M_2\lesssim 1} N_{min}^{-1-2\sigma} N^{2(\sigma-s)+2-\alpha}\|w\|_{Z^\sigma}^2 \|w\|_{Y^s}^2 \\
  &\lesssim N^{2(\sigma-s)+2-\alpha} \|w\|_{Z^\sigma}^2 \|w\|_{Y^s}^2,
\end{align*}
where we used that $ \sigma>-1/2$. This is acceptable for $ \sigma<s-1+\alpha/2 $ that is satisfied for $ \sigma<s-\frac32+\alpha $ as soon as $ \alpha\le 1$. This completes the proof of Propositions \ref{prop:Es} and  \ref{prop:E_tilde_sigma} in the periodic case.

Finally the proofs in  Section \ref{Sec:6} are exactly the same. We only have to care about the fact that we can not assume the sequence $\{u_0^j\} $ to be mean-value zero in Subsection \ref{64} since the mean-value of $
 u_0^j $ may depend on $ j$. However,  for any $j\in \N$, $u_0^j$ and $\Pi_{\le 2^n}u_0^j$ have the same mean value and thus \eqref{triangle} still holds. Then \eqref{flowH3} follows as in the real line case since Proposition \ref{propo:smooth} is still valid  in the periodic setting (cf \cite{ABFS89}).
\appendix 

\section{Symbol estimates} \label{app:symbols}
\subsection{Estimates for the resonance function}
For $(\xi_1,\ldots, \xi_n)\in \Gamma^{n}$, the resonance function of order $n$ associated with \eqref{fKdV} is defined by
$$
\Omega_n(\xi_1,\ldots, \xi_n) = \sum_{i=1}^n \omega_{\alpha+1}(\xi_i)
$$
where $\omega_{\alpha+1}$  satisfies Hypothesis \ref{hyp1}.
 Note that $\Omega_n$ is symmetric:
\begin{equation}\label{sym-Omega}
  \Omega_n(\xi_1,\ldots,\xi_n) = \Omega_n(\xi_{\sigma(1)},\ldots,\xi_{\sigma(n)}) \quad \textrm{ for any permutation }\sigma\in \mathcal{S}_n.
\end{equation}
Moreover, it satisfies
\begin{equation}\label{prop-Omega}
  \Omega_n(-(\xi_1,\cdots,\xi_n)) = -\Omega_n(\xi_1,\cdots,\xi_n).
\end{equation}
\subsubsection{The case $n=3$}
We recall the well-known result:
\begin{proposition}\label{est-Omega3}[Lemma 2.1 in \cite{MV}]
  Let $\alpha>0$,  $(\xi_1,\xi_2, \xi_3)\in \Gamma^3$ and $N_1,N_2,N_3$ be dyadic numbers such that $N_i\sim |\xi_i|$ for $i=1,2,3$. Assume moreover $N_{max} \gg \xi_0$. Then 
  \begin{equation}\label{est-Omega3.0}
    |\Omega_3(\xi_1,\xi_2,\xi_3)| \sim N_{min}N_{max}^\alpha
  \end{equation}
  where $N_{min}=\min_{i=1,2,3}N_i$ and $N_{max}=\max_{i=1,2,3}N_i$.
\end{proposition}

\subsubsection{The case $n=4$}
We will prove the following estimate:
\begin{proposition}\label{est-Omega4}
  Let  $\alpha>0$,  $(\xi_1,\xi_2,\xi_3,\xi_4)\in \Gamma^4$ and $N_1,N_2,N_3,N_4$ and $M_1,M_2,M_3$ be dyadic numbers such that $N_i\sim |\xi_i|$ for $i=1,2,3,4$ and $M_1\sim|\xi_2+\xi_3|$, $M_2\sim |\xi_1+\xi_3|$, $M_3\sim |\xi_1+\xi_2|$. Assume moreover $N_{max}  \gg \xi_0$. Then we have
  \begin{equation}\label{est-Omega4.0}
    |\Omega_4(\xi_1,\xi_2,\xi_3,\xi_4)| \sim M_{min}M_{med}N_{max}^{\alpha-1}
  \end{equation}
  with the usual notation.
\end{proposition}
\begin{proof}
  We use the notation $N_{min}\le N_{thd}\le N_{sub}\le N_{max}$ to denote $N_1,\ldots,N_4$ and recall that $N_{sub}\sim N_{max}\sim M_{max}$. We split the proof into three parts.
  
  \noindent
  \underline{Case $N_{thd} \ll N_{max}$}. By \eqref{sym-Omega}, we may assume
  $$
  N_1, N_2 \ll N_3\sim N_4,
  $$
  which implies $\xi_3\xi_4 < 0$. In this case we have $M_{min}=M_3$ and $M_{med}\sim N_{max}$.\\
  From the mean value theorem, we infer
  $$
  |\omega_{\alpha+1}(\xi_1)+\omega_{\alpha+1}(\xi_2)| = |\omega_{\alpha+1}(\xi_1)-\omega_{\alpha+1}(-\xi_2)| = |(\xi_1+\xi_2)\omega_{\alpha+1}'(\xi^\ast)|\sim M_3 |\omega_{\alpha+1}'(\xi^\ast)| 
  $$
  with $|\xi^\ast|\lesssim N_1\vee N_2\ll N_{max}$. Therefore,
  \begin{equation}\label{est-Omega4.1}
  |\omega_{\alpha+1}(\xi_1)+\omega_{\alpha+1}(\xi_2)| \ll M_3N_{max}^\alpha.
  \end{equation}
  In the same way, since $\xi_3\xi_4<0$ we can find $\xi^\ast$ with $|\xi^\ast|\sim N_{max}$ such that
  \begin{equation}\label{est-Omega4.2}
  |\omega_{\alpha+1}(\xi_3)+\omega_{\alpha+1}(\xi_4)| = |(\xi_3+\xi_4)\omega_{\alpha+1}'(\xi^\ast)| \sim M_3 N_{max}^\alpha.
  \end{equation}
  Hence, \eqref{est-Omega4.0} is deduced by combining \eqref{est-Omega4.1} and \eqref{est-Omega4.2}.
  
  \noindent
  \underline{Case $N_{min}\ll N_{thd} \sim N_{max}$}. By \eqref{sym-Omega}, we may assume
  $$
  N_1\ll N_2\sim N_3\sim N_4.
  $$
  It follows that $M_{min}\sim M_{max}\sim N_4$. Then, rewrite $\Omega_4$ as 
  $$
  \Omega_4(\xi_1,\xi_2,\xi_3,\xi_4) = \Omega_3(\xi_1,\xi_2) + \Omega_3(\xi_3,\xi_4).
  $$
  From Proposition \ref{est-Omega3}, we get $|\Omega_3(\xi_1,\xi_2)| \sim N_1 N_2^\alpha \ll N_4^{\alpha+1}$ and $|\Omega_3(\xi_3,\xi_4)|\sim N_4^{\alpha+1}$. Estimate \eqref{est-Omega4.0} follows immediately.
  
  \noindent
  \underline{Case $N_{min} \sim N_{max}$}. We separate 2 subcases.
  
  \noindent
  \textit{$\bullet$ 3 of $\xi_1,\ldots ,\xi_4$ are of the same sign}. From \eqref{sym-Omega}-\eqref{prop-Omega}, we may assume $\xi_1,\xi_2,\xi_3>0$ and $\xi_4<0$. Then rewrite $\Omega_4$ as
  \begin{equation}\label{est-Omega4.3}
  \Omega_4(\xi_1,\xi_2,\xi_3,\xi_4) = \Omega_3(\xi_1,\xi_2) + \Omega_3(\xi_3,\xi_1+\xi_2).
  \end{equation}
  At this point, we remark that both terms in the above sum have the same sign. Since they both are of size $N_4^{\alpha+1}$, we conclude 
   $$
  |\Omega_4(\xi_1,\xi_2,\xi_3,\xi_4)| \sim N_{max}^{\alpha+1},
  $$
  
  To show that both terms in the RHS of \eqref{est-Omega4.3} have the same sign, it suffices to prove that
 \begin{equation} \label{sign:Omega}
  \forall \xi_1,\xi_2>\xi_0, \quad  \sign(\Omega_3(\xi_1,\xi_2)) = -\sign(\omega_{\alpha+1}''(\xi_0)).
 \end{equation}
  Indeed, we rewrite $-\Omega_3(\xi_1,\xi_2)$ as 
  \begin{equation} \label{sign.Omega.1}
  \int_{\xi_0}^{\xi_2}\int_0^{\xi_1}\omega_{\alpha+1}''( \theta+\mu) d\mu d\theta +\left(\omega_{\alpha+1}(\xi_1+\xi_0)-\omega_{\alpha+1}(\xi_1)-\omega_{\alpha+1}(\xi_0)\right).
  \end{equation}
  On one hand, by the mean value theorem and Hypothesis \ref{hyp1},
  \[ \left|\omega_{\alpha+1}(\xi_1+\xi_0)-\omega_{\alpha+1}(\xi_1)-\omega_{\alpha+1}(\xi_0) \right| \le N_{max}^{\alpha} |\xi_0| \ll N_{max}^{\alpha+1} .\]
 On the other hand, by Hypothesis \ref{hyp1}, $\omega_{\alpha+1}'' $ is continuous outside of $0$ and $|\omega_{\alpha+1}''(\xi) | \sim |\xi|^{\alpha-1}$, so that $\omega_{\alpha+1}'' (\xi)$ does not change sign for $\xi>\xi_0$. Without loss of generality, we assume for instance $\omega_{\alpha+1}'' (\xi)>0$. Thus, 
 \[ \int_{\xi_0}^{\xi_2}\int_0^{\xi_1}\omega_{\alpha+1}''( \theta+\mu) d\mu d\theta \sim \int_{\xi_0}^{\xi_2}\int_0^{\xi_1}( \theta+\mu)^{\alpha-1} d\mu d\theta \sim N_{max}^{\alpha+1} ,\]
 which, combined with \eqref{sign.Omega.1}, proves \eqref{sign:Omega}.
  
  \noindent
  \textit{$\bullet$ 2 of $\xi_1,\ldots ,\xi_4$ are of the same sign, and the 2 others are of the opposite sign}. From \eqref{sym-Omega}-\eqref{prop-Omega}, we may assume $\xi_1,\xi_2>0$ and $\xi_3,\xi_4<0$. This implies that $\{M_{min},M_{med}\} = \{M_1,M_2\}$. To prove \eqref{est-Omega4.0}, we notice that
  $$
  \Omega_4(\xi_1,\xi_2,\xi_3,\xi_4) = -\int_0^{\xi_2+\xi_3}\int_{-\xi_3}^{\xi_1} \omega_{\alpha+1}''(\theta+\mu)d\mu d\theta
  $$
  and it suffices to show that $\theta+\mu \sim N_{max}$.\\
  If $0 < \xi_2+\xi_3$ and $-\xi_3<\xi_1$, then $0<-\xi_3<\theta+\mu <-\xi_4$ , thus $\theta+\mu \sim N_{max}$.\\
  If $0 < \xi_2+\xi_3$ and $-\xi_3>\xi_1$, then $0<\xi_1 < \theta+\mu < \xi_2$, thus $\theta+\mu \sim N_{max}$.\\
  The other cases follow by symmetry since $\int_a^b\int_c^d f = \int_b^a\int_d^c f$.  
\end{proof}

\subsection{Estimates for $a_k$}
\subsubsection{The case $k=3$}
For any $s>0$, the first symbol $a_3$ is defined on $\Gamma^3$ by $a_3(\xi_1,\xi_2,\xi_3) = \sum_{i=1}^3 \nu_s(\xi_i)$ where
$$
\nu_s(\xi) = \xi|\xi|^{2s}{ \bf 1}_{|\xi|\ge N_0}.
$$
We easily check arguing as in the proof of Proposition \ref{est-Omega3} that 
$$
|a_3(\xi_1,\xi_2,\xi_3)| \sim N_{min} N_{max}^{2s}.
$$
Next, we define $b_3 = a_3/\Omega_3$ on $\Gamma^3$. It follows from the above estimate that
\begin{equation}\label{est-b3}
|b_3(\xi_1,\xi_2,\xi_3)| \sim  N_{max}^{2s-\alpha}.
\end{equation}
Note also that $b_3$ is symmetric and satisfies
$$
b_3(-\xi_1,-\xi_2,-\xi_3) = b_3(\xi_1,\xi_2,\xi_3).
$$
Moreover, we record the following result for $b_3$. As previously explained, we often denote $b_3(\xi_1,\xi_2)=b_3(\xi_1,\xi_2,-(\xi_1+\xi_2))$ to simplify the notation. 
\begin{lemma}
  Assume that $|\xi_1|\ll |\xi_2|$ and $\beta=(\beta_1,\beta_2)\in\N^2$ be such that $|\beta|\le 2$. Then it holds that
  \begin{equation}\label{est-b3.1}
 \ |\partial^\beta b_3(\xi_1,\xi_2)|\lesssim |\xi_1|^{-\beta_1} |\xi_2|^{-\beta_2} |b_3(\xi_1,\xi_2)|.
\end{equation}
\end{lemma}
\begin{proof}
  This is a direct consequence of Lemma 2.6 in \cite{MPV}. It is worth noticing that with Hypothesis \ref{hyp1} on the symbol regularity, estimate \eqref{est-b3.1} only holds for $|\beta|\le 2$.
\end{proof}

\subsubsection{The case $k=4$}

For $(\xi_1,\xi_2,\xi_3,\xi_4) \in \Gamma^4$, let us set
$$
a_4(\xi_1,\xi_2,\xi_3,\xi_4) = \sum_{1 \le i<j \le 4}(\xi_i+\xi_j)b_3(\xi_i,\xi_j).
$$
\begin{proposition}\label{est-a4.0}
 With the notation of Proposition \ref{est-Omega4}, we have for all $s,\alpha>0$,
 \begin{equation}\label{est-a4.1}
  |a_4(\xi_1,\xi_2,\xi_3,\xi_4)| \lesssim M_{min} M_{med}(N_{thd}^{2s-\alpha}\vee N_{max}^{2s-\alpha}) N_{max}^{-1}.
 \end{equation}
\end{proposition}
\begin{proof}
  \noindent
  \underline{Case $N_{thd} \ll N_{max}$}.
  By \eqref{sym-Omega}, we may assume
  $$
  N_1\lesssim N_2 \ll N_3\sim N_4,
  $$
  which implies $\xi_3\xi_4 < 0$ and $M_{min}=M_3$ and $M_1\sim M_2\sim N_{max}$. Using the properties of $b_3$, we rewrite $a_4$ as follows:
  \begin{align*}
   a_4(\xi_1,\xi_2,\xi_3,\xi_4) &= (\xi_1+\xi_2)[b_3(\xi_1,\xi_2) - b_3(\xi_3,\xi_4)]\\
   &\quad + (\xi_2+\xi_3)[b_3(\xi_2,\xi_3) - b_3(\xi_1,\xi_4)] \\
   &\quad + (\xi_1+\xi_3)[b_3(\xi_1,\xi_3) - b_3(\xi_2,\xi_4)] \\
   &= (\xi_1+\xi_2)[b_3(\xi_1,\xi_2) - b_3(\xi_3,\xi_4)] \\
   &\quad + \xi_2[b_3(\xi_2,\xi_3)-b_3(-\xi_1,-\xi_4)] + \xi_1[b_3(\xi_1,\xi_3)-b_3(-\xi_2,-\xi_4)]\\
   &\quad + \xi_3[b_3(\xi_2,\xi_3)-b_3(\xi_2,\xi_1+\xi_3)] - \xi_3[b_3(\xi_1,\xi_4)-b_3(\xi_1,\xi_2+\xi_4)]\\
   & =: a_{41}(\xi_1,\xi_2,\xi_3,\xi_4) +a_{42}(\xi_1,\xi_2,\xi_3,\xi_4) +a_{43}(\xi_1,\xi_2,\xi_3,\xi_4).
  \end{align*}
    The first term $a_{41}$ is easily bounded thanks to \eqref{est-b3}:
    $$
    |a_{41}(\xi_1,\xi_2,\xi_3,\xi_4)| \lesssim M_3 (|b_3(\xi_1,\xi_2)| + |b_3(\xi_3,\xi_4)|)\lesssim M_3(N_2^{2s-\alpha} + N_4^{2s-\alpha}).
    $$
    Now we turn to $a_{42}$. First if $N_2\sim M_3$, then the needed bound is obtained as for $a_{41}$, thus it suffices to deal with the case $M_3\ll N_2$, which enforce $N_1\sim N_2$ and $\xi_1\xi_2<0$. Moreover, from the symmetry in $(\xi_1, \xi_2)$, we only need to consider the first term in $a_{42}$. By making use of the mean value theorem, as well as estimate \eqref{est-b3.1} we obtain
    \begin{align*}
      |b_3(\xi_2,\xi_3)& - b_3(-\xi_1,-\xi_4)| \\ &= |(b_3(\xi_2,\xi_3) - b_3(-\xi_1,\xi_3)) + (b_3(-\xi_1,\xi_3) - b_3(-\xi_1,-\xi_4)|\\
      &\lesssim |\xi_1+\xi_2| \sup_{|\xi^\ast|\sim N_2}|\partial^{(1,0)}b_3(\xi^\ast, \xi_3)| + |\xi_3+\xi_4| \sup_{|\xi^\ast|\sim N_4}|\partial^{(0,1)}b_3(-\xi_1,\xi^\ast)|\\
      &\lesssim M_3(N_2^{-1}N_4^{2s-\alpha} + N_4^{-1}N_4^{2s-\alpha})\\
      &\lesssim M_3N_2^{-1}N_4^{2s-\alpha}.
    \end{align*}
    This provides the desired bound for $a_{42}$.
    Now for the contribution of $a_{43}$, the Taylor formula gives
    $$
    b_3(\xi_2,\xi_3)-b_3(\xi_2,\xi_1+\xi_3) = -\xi_1 \int_0^1 \partial^{(0,1)}b_3(\xi_2, \theta\xi_3+(1-\theta)(\xi_1+\xi_3))d\theta
    $$
    and
    $$
    b_3(\xi_1,\xi_4)-b_3(\xi_1,\xi_2+\xi_4) = -\xi_2\int_0^1\partial^{(0,1)}b_3(\xi_1, \theta\xi_4+(1-\theta)(\xi_2+\xi_4))d\theta
    $$
    with $|\theta\xi_3+(1-\theta)(\xi_1+\xi_3)|\sim |\theta\xi_4+(1-\theta)(\xi_2+\xi_4)|\sim N_4$ for $\theta\in(0,1)$. The desired result follows easily from this and \eqref{est-b3.1} in the case $N_2\sim M_3$. Thus we assume now $M_3\ll N_2$ so that $N_1\sim N_2$ and $\xi_1\xi_2<0$. Splitting $\xi_1=(\xi_1+\xi_2)-\xi_2$, we decompose 
   \begin{align*}
     a_{43}(\xi_1,\xi_2,\xi_3,\xi_4)=a_{431}(\xi_1,\xi_2,\xi_3,\xi_4) + a_{432}(\xi_1,\xi_2,\xi_3,\xi_4),
    \end{align*}  
 where   
      \begin{align*}
    a_{431}(\xi_1,\xi_2,\xi_3,\xi_4) = -\xi_3(\xi_1+\xi_2) \int_0^1 \partial^{(0,1)}b_3(\xi_2, \theta\xi_3+(1-\theta)(\xi_1+\xi_3))d\theta
    \end{align*}
    and 
     \begin{align*}
    a_{432}(\xi_1,\xi_2,\xi_3,\xi_4)& =\xi_2\xi_3 \int_0^1\partial^{(0,1)}b_3(\xi_1, \theta\xi_4+(1-\theta)(\xi_2+\xi_4))d\theta  \\ & \quad -\xi_2\xi_3  \int_0^1 \partial^{(0,1)}b_3(-\xi_2, -\theta\xi_3-(1-\theta)(\xi_1+\xi_3))d\theta  .
    \end{align*} 
    
    The symbol $a_{431}$ may be bounded as previously. To deal with $a_{432}$, we first observe that setting $\Psi= \partial^{(0,1)}b_3$, we have
    \begin{align*}
      \Psi(\xi_1, &\xi_4+(1-\theta)\xi_2) - \Psi(-\xi_2, -\xi_3-(1-\theta)\xi_1) \\
      &=
       [ \Psi(\xi_1, \xi_4+(1-\theta)\xi_2) - \Psi(-\xi_2, \xi_4+(1-\theta)\xi_2) ]\\
      &\quad + [ \Psi(-\xi_2, \xi_4+(1-\theta)\xi_2) - \Psi(-\xi_2, -\xi_3-(1-\theta)\xi_1) ]\\
      &= (\xi_1+\xi_2)\int_0^1 \partial^{(1,0)}\Psi(\mu\xi_1-(1-\mu)\xi_2, \xi_4+(1-\theta)\xi_2)d\mu \\
      &\quad - (\xi_1+\xi_2)\int_0^1 \theta \partial^{(0,1)}\Psi(-\xi_2, \mu(\xi_4+(1-\theta)\xi_2) - (1-\mu)(\xi_3+(1-\theta)\xi_1)) d\mu .
    \end{align*} 
    Hence, $a_{432}$ may be rewritten as    
    \begin{align*}
     & \xi_2\xi_3(\xi_1+\xi_2) \int_0^1\int_0^1 \partial^{(1,1)}b_3(\mu\xi_1-(1-\mu)\xi_2, \xi_4+(1-\theta)\xi_2) d\mu d\theta \\
     &\ -\xi_2\xi_3(\xi_1+\xi_2)\int_0^1\int_0^1 \theta\partial^{(0,2)}b_3(-\xi_2, \mu(\xi_4+(1-\theta)\xi_2) - (1-\mu)(\xi_3+(1-\theta)\xi_1)) d\mu d\theta. 
    \end{align*}
    Since $\xi_1\xi_2<0$ and $\xi_3\xi_4<0$, it holds that 
    $$
    |\mu\xi_1-(1-\mu)\xi_2| \sim N_2
    $$
    and 
    $$
    |\xi_4+(1-\theta)\xi_2| \sim |\mu(\xi_4+(1-\theta)\xi_2) - (1-\mu)(\xi_3+(1-\theta)\xi_1)| \sim N_4
    $$
    for all $\theta,\mu\in (0,1)$ and the claim follows thanks to \eqref{est-b3.1}.    
    \noindent
    
  \underline{Case $N_{min}\ll N_{thd} \sim N_{max}$}. We have $M_{min}\sim M_{max}\sim N_{max}$. It is clear then that for any $1 \le i<j \le 4$, $|(\xi_i+\xi_j)b_3(\xi_i,\xi_j)|\lesssim N_{max}^{1+2s-\alpha}$, which by recalling the definition of $a_4$ yields estimate \eqref{est-a4.1} in this case. 
  
  \noindent
  \underline{Case $N_{min} \sim N_{max}:= N$}.
  
 \noindent
  \textit{$\bullet$ 3 of $\xi_1,\ldots \xi_4$ are of the same sign}. Again, we must have $M_{min}\sim M_{max}\sim N_{max}$ and the bound \eqref{est-a4.1} follows arguing as in the previous case.
  
  \noindent
  \textit{$\bullet$ 2 of $\xi_1,\ldots \xi_4$ are of the same sign, and the 2 others are of the opposite sign}. We may assume $\xi_1,\xi_2>0$ and $\xi_3,\xi_4<0$. This implies that $\{M_{min},M_{med}\} = \{M_1,M_2\}$. We will use the following decomposition of $a_4$:
  \begin{align}
   a_4(\xi_1,\xi_2,\xi_3,\xi_4) &= (\xi_1+\xi_2)[b_3(\xi_1,\xi_2) - b_3(-\xi_3,-\xi_4)]\nonumber \\
   &\quad + (\xi_2+\xi_3)[b_3(\xi_2,\xi_3) - b_3(-\xi_1,-\xi_4)] \nonumber\\
   &\quad + (\xi_1+\xi_3)[b_3(\xi_1,\xi_3) - b_3(-\xi_2,-\xi_4)] \nonumber\\   
   & =: a_{44}(\xi_1,\xi_2,\xi_3,\xi_4) +a_{45}(\xi_1,\xi_2,\xi_3,\xi_4) +a_{46}(\xi_1,\xi_2,\xi_3,\xi_4).\label{a4}
  \end{align}
  From the symmetry in $\xi_1,\xi_2$, it suffices to estimate $a_{44}$ and $a_{45}$. Using the definition of $b_3$ we may decompose
  $$
  b_3(\xi_2, \xi_3) - b_3(-\xi_1, -\xi_4) = a_3(\xi_2,\xi_3)\left(\frac{1}{\Omega_3(\xi_2,\xi_3)} - \frac{1}{\Omega_3(-\xi_1,-\xi_4)}\right) + \frac{a_3(\xi_2,\xi_3)-a_3(-\xi_1,\xi_4)}{\Omega_3(-\xi_1,-\xi_4)}
  $$
  Then, notice that $\Omega_3(\xi_2,\xi_3)-\Omega_3(-\xi_1,-\xi_4) = \Omega_4(\xi_1,\xi_2,\xi_3)$, therefore we may take advantage of Lemma \ref{est-Omega4} to infer
  $$
  \left| a_3(\xi_2,\xi_3)\left(\frac{1}{\Omega_3(\xi_2,\xi_3)} - \frac{1}{\Omega_3(-\xi_1,-\xi_4)}\right) \right| \lesssim M_1N^{2s}\frac{M_1M_2N^{\alpha-1}}{(M_1N^\alpha)^2} \lesssim M_2N^{2s-1-\alpha}.
  $$ 
  Similarly, since the bounds for $a_3$ are the exact same than those for $\Omega_3$ with $\alpha$ replaced with $2s$, we get
  $$
  \left| \frac{a_3(\xi_2,\xi_3)-a_3(-\xi_1,\xi_4)}{\Omega_3(-\xi_1,-\xi_4)} \right| \lesssim \frac{M_1M_2N^{2s-1}}{M_1N^\alpha} \lesssim M_2N^{2s-1-\alpha}.
  $$
  Gathering the above estimates we deduce $|a_{45}|\lesssim M_1M_2N^{2s-1-\alpha}$.
 Finally, the bound for $a_{44}$ is obtained in the same way:
 \begin{align*}
   |a_{44}| &\lesssim N \left( \left|a_3(\xi_1,\xi_2)\frac{\Omega_4(\xi_1,\xi_2,\xi_3)}{\Omega_3(\xi_1,\xi_2)\Omega_3(\xi_3,\xi_4)}\right| + \left|\frac{a_3(\xi_1,\xi_2)+a_3(\xi_3,\xi_4)}{\Omega_3(\xi_3,\xi_4)}\right|  \right)\\
   &\lesssim N\left( N^{2s+1}\frac{M_1M_2N^{\alpha-1}}{N^{2(1+\alpha)}} + \frac{M_1M_2N^{2s-1}}{N^{1+\alpha}} \right)\\
   &\lesssim M_1M_2N^{2s-\alpha-1}.
 \end{align*} 
\end{proof}

\subsection{Estimates for $\widetilde{a}_k$}
\subsubsection{Estimates for $\widetilde{b}_3$}
Recall that for  $\sigma>-\frac 12$ and $N_0\gg 1$, $\widetilde{a}_3$ is defined by
$$
\widetilde{a_3}(\xi_1, \xi_2) = (\widetilde{\nu}_\sigma(\xi_1) + \widetilde{\nu}_\sigma(\xi_2)){\bf 1}_{|\xi_1+\xi_2|>N_0},
$$
where $\widetilde{\nu}_\sigma(\xi)=\xi|\xi|^{2\sigma}{\bf 1}_{|\xi|>1}$ and that $\widetilde{b}_3 = \widetilde{a}_3 / \Omega_3$.  It is worth noticing that $\widetilde{a}_3$ vanishes unless $\max(|\xi_1|, |\xi_2|)\gg 1$.

\begin{lemma}\label{lem:b3tilde}
  Let $(\xi_1, \xi_2, \xi_3)\in \Gamma^3$ and $N_1,N_2,N_3$ be dyadic numbers such that $N_i\sim |\xi_i|$ for $i=1,2,3$. 
  \begin{enumerate}
    \item If $N_1\sim N_2$, one has
    \begin{equation}\label{lem:b3tilde.1}
    |\widetilde{b_3}(\xi_1, \xi_2)| \lesssim N_{max}^{2\sigma-\alpha}.
    \end{equation}
    \item If $N_1\ll N_2$ or $N_2\ll N_1$, one has
    \begin{equation}\label{lem:b3tilde.2}
    |\widetilde{b_3}(\xi_1, \xi_2)| \lesssim N_{min}^{-1} N_{max}^{2\sigma+1-\alpha}.
    \end{equation}
  \end{enumerate}
\end{lemma}
\begin{proof}
We may always assume $N_{max}\gg 1$. In the case $N_3\sim N_{min}$ then this forces $N_1\sim N_2\sim N_{max}\gg 1$. Splitting the cases $\xi_1\xi_2>0$ and $\xi_1\xi_2<0$, and using the mean value theorem in the latter case, we find that $|\widetilde{a}_3(\xi_1, \xi_2)|\sim N_3 N^{2\sigma}$. Estimate \eqref{lem:b3tilde.1} follows from this and \eqref{est-Omega3.0}. In the same way, for $N_3\sim N_{max}$ we easily get $|\widetilde{a}_3(\xi_1, \xi_2)|\sim N_{max}^{2\sigma+1}$ and we may conclude \eqref{lem:b3tilde.2} with \eqref{est-Omega3.0}.
\end{proof}

The following lemma will also be useful.
\begin{lemma}\label{lem:diff-b3tilde}
  With the notations of Lemma \ref{lem:b3tilde}, if $N_1\ll N_2$, then
  \begin{equation}\label{lem:diff-b3tilde.0}
    |\partial^{(0,1)}\widetilde{b}_3(\xi_1, \xi_2)| \lesssim N_1^{-1}N_2^{2\sigma-\alpha}.
  \end{equation}
\end{lemma}

\begin{proof}
  Noticing that $|\widetilde{\nu}_\sigma'(\xi_2)|\sim N_2^{2\sigma}$ and $|\partial^{(0,1)}\Omega_3(\xi_1,\xi_2)|\sim N_1N_2^{\alpha-1}$ thanks to the mean value theorem, we get using the classical derivative rules that
  \begin{align*}
    |\partial^{(0,1)}\widetilde{b}_3(\xi_1, \xi_2)| &= \left|\frac{\widetilde{\nu}_\sigma'(\xi_2)}{\Omega_3(\xi_1,\xi_2)} - \frac{\left(\partial^{(0,1)}\Omega_3(\xi_1, \xi_2)\right) \widetilde{a}_3(\xi_1,\xi_2)}{(\Omega_3(\xi_1,\xi_2))^2}\right| \\
    &\lesssim \frac{N_2^{2\sigma}}{N_1N_2^\alpha} + \frac{N_1N_2^{\alpha-1} N_2^{2\sigma+1}}{(N_1N_2^\alpha)^2}\\
    &\lesssim N_1^{-1}N_2^{2\sigma-\alpha}.
  \end{align*}
\end{proof}

\subsubsection{Estimates for $\widetilde{a}_4$}
In this section we prove the following lemma.
\begin{lemma}\label{lem:a4tilde}
  Let $(\xi_1,\xi_2,\xi_3,\xi_4) \in \Gamma^4$ and $N_1, N_2, N_3, N_4$ be dyadic numbers such that $N_i\sim |\xi_i|$ for $i=1,2,3,4$. From the symmetry for $\widetilde{a}_4$ we assume $|\xi_1|\le |\xi_2|$ and $|\xi_3|\le |\xi_4|$. Then, for any  $\alpha>0$ and $\sigma\in \left(-\frac 12, \min\left\{0, -\frac 12+\frac{\alpha}2\right\} \right)$, the following estimates for $\widetilde{a}_4$ hold true.
  \begin{enumerate}
    \item {\bf Region $A$}: If $N_{min} \sim N_{max}$,
    \begin{equation}\label{a4tilde-A}
      |\widetilde{a}_4(\xi_1,\xi_2,\xi_3)| \lesssim N_{max}^{2\sigma+1-\alpha} .
    \end{equation}
    \item {\bf Region $B$}: $N_{min}\ll N_{thd}\sim N_{max}$.
    \begin{enumerate}
      \item {\bf Sub-region $B_1$}: If $N_1\ll N_2\sim N_3\sim N_4$,
      \begin{equation}\label{a4tilde-B1}
      |\widetilde{a}_4(\xi_1,\xi_2,\xi_3)| \lesssim N_1^{-1} N_{max}^{2\sigma+2-\alpha} .
    \end{equation}
    \item {\bf Sub-region $B_2$}: If $N_3\ll N_1\sim N_2\sim N_4$,
      \begin{equation}\label{a4tilde-B2}
      |\widetilde{a}_4(\xi_1,\xi_2,\xi_3)| \lesssim N_{max}^{2\sigma+1-\alpha} .
    \end{equation}
    \end{enumerate}
    \item {\bf Region $C$}: $N_{thd} \ll N_{max}$.
    \begin{enumerate}
      \item {\bf Sub-region $C_1$}: If $N_1\sim N_2\gg N_4\gtrsim N_3$,
       \begin{equation}\label{a4tilde-C1}
      |\widetilde{a}_4(\xi_1,\xi_2,\xi_3)| \lesssim N_4 N_{max}^{2\sigma-\alpha} .
    \end{equation}
     \item {\bf Sub-region $C_2$}: If $N_1\lesssim N_2\ll N_3\sim N_4$,
       \begin{equation}\label{a4tilde-C2}
      |\widetilde{a}_4(\xi_1,\xi_2,\xi_3)| \lesssim N_1^{-1} N_{max}^{2\sigma+2-\alpha} .
    \end{equation}
    \item {\bf Sub-region $C_3$}: If $N_1\ll N_3\ll N_2\sim N_4$,
    \begin{equation}\label{a4tilde-C3b}
      |\widetilde{a}_4(\xi_1,\xi_2,\xi_3)| \lesssim N_1^{-1} N_3^{2\sigma+2-\alpha} .
    \end{equation}
    Moreover, we have 
   \begin{equation}\label{a4tilde-C3}
      \left|\widetilde{a}_4(\xi_1,\xi_2,\xi_3) + \frac{\xi_3\nu_\sigma(\xi_3)}{\xi_1\omega'(\xi_3)} \right| \lesssim N_1^{2\sigma}N_3^{1-\alpha} + N_1^{-1}N_3N_{max}^{2\sigma+1-\alpha} .
    \end{equation}
    
    \item {\bf Sub-region $C_4$}: If $N_3\lesssim N_1\ll N_2\sim N_4$,
       \begin{equation}\label{a4tilde-C4}
      |\widetilde{a}_4(\xi_1,\xi_2,\xi_3)| \lesssim N_1^{2\sigma+1-\alpha} .
    \end{equation}
    \end{enumerate}
  \end{enumerate}
\end{lemma}

\begin{remark}
\begin{enumerate}
\item[$(i)$]
 For the region $C_2$ in the particular case $M_3:=|\xi_1+\xi_2|\ll N_1\sim N_2$,  we have the simpler bound
 $$
 |\widetilde{a}_4(\xi_1,\xi_2,\xi_3)| \lesssim N_1^{-1}N_2^{-1}M_3 N_{max}^{2\sigma+2-\alpha}.
 $$
 However, estimate \eqref{a4tilde-C2} is good enough for our purpose.
\item[$(ii)$]
In the resonant region $A$, we have the better bound 
\begin{equation} \label{bound:a4tilde-resonant}
|\widetilde{a}_{4}| \lesssim M_3N_{max}^{2\sigma-\alpha} .
\end{equation}
However, we will not use this bound in our energy estimate.
\end{enumerate}
\end{remark}
\begin{proof}
We will extensively use Lemma \ref{lem:b3tilde} to bound $\widetilde{a}_4$.
Recalling the definition of $\widetilde{a}_4$ in \eqref{def:a4tilde}, we also decompose 
\begin{equation}  \label{def:a4tilde.2}
\widetilde{a}_{41}=\widetilde{a}_{411}+\widetilde{a}_{412} \quad \text{and} \quad \widetilde{a}_{42}=\widetilde{a}_{421}+\widetilde{a}_{423}
\end{equation}
where 
\begin{align*}
\widetilde{a}_{411} & = (\xi_1+\xi_3) \widetilde{b_3}(\xi_1+\xi_3, \xi_2); \\
\widetilde{a}_{412}& =-(\xi_1+\xi_3) \widetilde{b_3}(\xi_1+\xi_3, -\xi_1); \\
\widetilde{a}_{421}& =(\xi_2+\xi_3)\widetilde{b_3}(\xi_2+\xi_3, \xi_1); \\
\widetilde{a}_{422}& = - (\xi_2+\xi_3)\widetilde{b_3}(\xi_2+\xi_3, -\xi_2) .
\end{align*}
We will denote by $M_1, M_2, M_3$ some dyadic numbers satisfying $M_1\sim|\xi_2+\xi_3|$, $M_2\sim |\xi_1+\xi_3|$ and, $M_3\sim |\xi_1+\xi_2|$.  
\vskip .3cm
\noindent
{\bf Region $A$}: $N_1\sim N_2\sim N_3\sim N_4 =: N$\\
Since $|\xi_1+\xi_3|\lesssim |\xi_1|\sim |\xi_2|$ Lemma \ref{lem:b3tilde} provides
\begin{align*}
|\widetilde{a}_{41}| &\le|(\xi_1+\xi_3)\widetilde{b}_3(\xi_1+\xi_3, \xi_2)| + |(\xi_1+\xi_3)\widetilde{b}_3(\xi_1+\xi_3, -\xi_1)|\\
&\lesssim   M_2(M_2^{-1}N^{2\sigma+1-\alpha} + M_2^{-1}N^{2\sigma+1-\alpha}) \\
&\lesssim N^{2\sigma+1-\alpha}.
\end{align*}
Similarly, we obtain the same bound for $|\widetilde{a}_{42}| + |\widetilde{a}_{43}|\lesssim N^{2\sigma+1-\alpha}$.

\vskip .3cm
\noindent
{\bf Region $B_1$}: $N_1\ll N_2\sim N_3\sim N_4=:N$\\
In this region, we have $M_1\sim M_2\sim M_3\sim N$. Therefore we get on one hand
$$
|\widetilde{a}_{411}| + |\widetilde{a}_{422}| + |\widetilde{a}_{43}| \lesssim N^{2\sigma+1-\alpha},
$$
and on the other hand
$$
|\widetilde{a}_{412}| + |\widetilde{a}_{421}| \lesssim N_1^{-1}N^{2\sigma+2-\alpha}.
$$
Then, \eqref{a4tilde-B1} follows immediately.

\vskip .3cm
\noindent
{\bf Region $B_2$}: $N_3\ll N_1\sim N_2\sim N_4 =: N$\\ 
Again, we have $M_1\sim M_2\sim M_3\sim N$ and every term $\widetilde{a}_{4ij}$, $i,j\in\{1,2\}$, as well as $\widetilde{a}_{43}$ is bounded by $N^{2\sigma+1-\alpha}$.

\vskip .3cm
\noindent
{\bf Region $C_1$}: $N_3\lesssim N_4 \ll N_1\sim N_2=: N$\\ 
In $C_1$ it holds that $M_3\lesssim N_4 \ll M_1\sim M_2\sim N$. First, we deduce from \eqref{lem:b3tilde.1} that
\begin{equation}\label{lem:a4tilde.1}
|\widetilde{a}_{43}| = |(\xi_1+\xi_2)\widetilde{b}_3(\xi_1,\xi_2)|  \lesssim M_3 N^{2\sigma-\alpha}.
\end{equation}
Using the symmetry in $(\xi_1,\xi_2)$ it remains to estimate 
\begin{align*}
\widetilde{a}_{411}+\widetilde{a}_{421} &= \xi_1[\widetilde{b}_3(\xi_1+\xi_3,\xi_2) - \widetilde{b}_3(\xi_2+\xi_3, \xi_1)] \\
&\quad +\xi_3 \widetilde{b}_3(\xi_1+\xi_3,\xi_2) - \xi_4 \widetilde{b}_3(\xi_2+\xi_3, \xi_1)\\
& := \widetilde{a}_{411}^{(1)} + \widetilde{a}_{411}^{(2)} + \widetilde{a}_{411}^{(3)}.
\end{align*}
From \eqref{lem:b3tilde.1} we get
\begin{equation}\label{lem:a4tilde.2}
|\widetilde{a}_{411}^{(2)}| + |\widetilde{a}_{411}^{(3)}| \lesssim (N_3+N_4)N^{2\sigma-\alpha} \lesssim N_4 N^{2\sigma-\alpha}.
\end{equation}
Then, we may rewrite $\widetilde{a}_{411}^{(1)}$ as
\begin{align*}
  \widetilde{a}_{411}^{(1)} &= \xi_1\left( \frac{\widetilde{a}_3(\xi_1+\xi_3,\xi_2)}{\Omega_3(\xi_1+\xi_3, \xi_2)} - \frac{\widetilde{a}_3(\xi_2+\xi_3, \xi_1)}{\Omega_3(\xi_2+\xi_3, \xi_1)} \right)\\
  &= \xi_1\Big( \widetilde{a}_3(\xi_1+\xi_3,\xi_2)\left(\frac{1}{\Omega_3(\xi_1+\xi_3, \xi_2)} - \frac{1}{\Omega_3(\xi_2+\xi_3, \xi_1)}\right) \\
  &\quad + \frac{\widetilde{a}_3(\xi_1+\xi_3,\xi_2) -\widetilde{a}_3(\xi_2+\xi_3, \xi_1)}{\Omega_3(\xi_2+\xi_3, \xi_1)}  \Big)\\
  &:= \widetilde{a}_{411}^{(11)} + \widetilde{a}_{411}^{(12)}.
\end{align*}
Using twice Taylor formula we deduce
\begin{align*}
 |\Omega_3(\xi_2+\xi_3, \xi_1) - \Omega_3(\xi_1+\xi_3, \xi_2)| &= |\omega_{\alpha+1}(\xi_2+\xi_3)-\omega_{\alpha+1}(\xi_2) + \omega_{\alpha+1}(\xi_1) - \omega_{\alpha+1}(\xi_1+\xi_3)|\\
 &\lesssim |\xi_3\omega_{\alpha+1}'(\xi_2) - \xi_3\omega_{\alpha+1}'(\xi_1)| + N_3^2N^{\alpha-1}\\
 &\lesssim N_3M_3N^{\alpha-1} + N_3^2N^{\alpha-1}\\
 &\lesssim N_4^2N^{\alpha-1}.
\end{align*}
It follows that
\begin{equation}\label{lem:a4tilde.3}
|\widetilde{a}_{411}^{(11)}| \lesssim N M_3N^{2\sigma} \frac{N_4^2N^{\alpha-1}}{(N_4N^\alpha)^2} \lesssim N_4N^{2\sigma-\alpha}.
\end{equation}
Similarly, recalling that $\widetilde{\nu}_\sigma(\xi)=\xi_|\xi|^{2\sigma}{\bf 1}_{|\xi|\ge 1}$ and $2\sigma+1>0$ we obtain
\begin{align*}
 |\widetilde{a}_3(\xi_1+\xi_3,\xi_2) -\widetilde{a}_3(\xi_2+\xi_3, \xi_1)| &\lesssim |\xi_4\widetilde{\nu}_\sigma'(\xi_1) - \xi_4\widetilde{\nu}_\sigma'(\xi_2)| + N_4^2N^{2\sigma-1}\\
 &\lesssim N_4M_3N^{2\sigma-1} + N_4^2N^{2\sigma-1}\\
 &\lesssim N_4^2N^{2\sigma-1},
\end{align*}
from which we conclude
\begin{equation}\label{lem:a4tilde.4}
|\widetilde{a}_{411}^{(12)}| \lesssim N\frac{N_4^2N^{2\sigma-1}}{N_4N^\alpha} \lesssim N_4N^{2\sigma-\alpha}.
\end{equation}
Gathering \eqref{lem:a4tilde.1}-\eqref{lem:a4tilde.2}-\eqref{lem:a4tilde.3}-\eqref{lem:a4tilde.4} we get the desired estimate for the region $C_1$.

\vskip .3cm
\noindent
{\bf Region $C_2$}: $N_1\lesssim N_2 \ll N_3\sim N_4=: N$\\ 
Thanks to \eqref{lem:b3tilde.2}, we directly obtain the rough bounds
$$
|\widetilde{a}_{411}| + |\widetilde{a}_{422}| \lesssim N_2^{-1}N^{2\sigma+2-\alpha} \lesssim N_1^{-1}N^{2\sigma+2-\alpha}
$$
and
$$
|\widetilde{a}_{412}| + |\widetilde{a}_{421}| \lesssim N_1^{-1}N^{2\sigma+2-\alpha}.
$$
In the same way, using \eqref{lem:b3tilde.1} or \eqref{lem:b3tilde.2} we always have
$$
|\widetilde{a}_{43}| \lesssim M_3 N_1^{-1}N_2^{-1} N_2^{2\sigma+2-\alpha}\lesssim N_1^{-1} N^{2\sigma+2-\alpha}.
$$

\vskip .3cm
\noindent
{\bf Region $C_3$}: $N_1\ll N_3 \ll N_2\sim N_4=: N$\\ 
We only show the most difficult estimate \eqref{a4tilde-C3}. Note that in $C_3$, we have $M_2\sim N_3$. According to Lemma \ref{lem:b3tilde} we get
$$
|\widetilde{a}_{411}| + |\widetilde{a}_{422}| \lesssim N^{2\sigma+1-\alpha} \lesssim N_1^{-1}N_3N^{2\sigma+1-\alpha}.
$$
For $\widetilde{a}_{412}$, we split
\begin{align}
\widetilde{a}_{412} &= -\xi_3\widetilde{b}_3(\xi_1+\xi_3,-\xi_1) + \mathcal{O}(N_3^{2\sigma+1-\alpha}) \notag \\
&= -\xi_3\frac{\nu_\sigma(\xi_1+\xi_3) - \nu_\sigma(\xi_1)}{\Omega_3(\xi_1+\xi_3, -\xi_1)} + \mathcal{O}(N_1^{2\sigma}N_3^{1-\alpha}). \label{lem:a4tilde.5}
\end{align}
Next, using a Taylor expansion of $\omega_{\alpha+1}$, we rewrite $$ \Omega_3(\xi_1+\xi_3, -\xi_1) = \xi_1\omega_{\alpha+1}'(\xi_3) + R(\xi_1,\xi_3)$$ with $|R(\xi_1,\xi_3)| \lesssim N_1^{\alpha+1}$. It follows that
\begin{equation}\label{lem:a4tilde.6}
\frac 1{\Omega_3(\xi_1+\xi_3, -\xi_1)} = \frac{1}{\xi_1\omega'_{\alpha+1}(\xi_3)} \frac{1}{1+\frac{R(\xi_1,\xi_3)}{\xi_1\omega'(\xi_3)}} = \frac{1}{\xi_1\omega'_{\alpha+1}(\xi_3)} + \mathcal{O}(N_1^{\alpha-1}N_3^{-2\alpha}) .
\end{equation}
Moreover, the mean value theorem provides
\begin{equation}\label{lem:a4tilde.7}
\widetilde{\nu}_\sigma(\xi_1+\xi_3) - \widetilde{\nu}_\sigma(\xi_1) = \widetilde{\nu}_\sigma(\xi_3) - \widetilde{\nu}_\sigma(\xi_1) + \mathcal{O}(N_1N_3^{2\sigma}) = \widetilde{\nu}_\sigma(\xi_3) + \mathcal{O}(N_1^{2\sigma+1}).
\end{equation}
Combining \eqref{lem:a4tilde.5}-\eqref{lem:a4tilde.6}-\eqref{lem:a4tilde.7} we infer
$$
\left|\widetilde{a}_{412} + \frac{\xi_3\widetilde{\nu}_\sigma(\xi_3)}{\xi_1\omega'(\xi_3)} \right| \lesssim N_1^{2\sigma}N_3^{1-\alpha}+N_3^{2+2\sigma-2\alpha}N_1^{\alpha-1} \lesssim N_1^{2\sigma}N_3^{1-\alpha},
$$
since $2\sigma+1-\alpha<0$.
To conclude the proof of \eqref{a4tilde-C3} we bound the term $\widetilde{a}_{421}+\widetilde{a}_{43}$ thanks to \eqref{lem:diff-b3tilde.0} by
\begin{align*}
  \widetilde{a}_{421}+\widetilde{a}_{43} &= \xi_2[\widetilde{b}_3(\xi_2+\xi_3,\xi_1) - \widetilde{b}_3(\xi_1,\xi_2)] + \mathcal{O}(N^{2\sigma+1-\alpha}) + \mathcal{O}(N_1^{-1}N_3N^{2\sigma+1-\alpha})\\
  &= \xi_2\xi_3\partial^{(0,1)}\widetilde{b}_3(\xi_1, \xi^\ast) + \mathcal{O}(N_1^{-1}N_3N^{2\sigma+1-\alpha}),\quad |\xi^\ast|\sim N\\
  &= \mathcal{O}(N_1^{-1}N_3N^{2\sigma+1-\alpha}).
\end{align*}

\vskip .3cm
\noindent
{\bf Region $C_4$}: $N_3\lesssim N_1 \ll N_2\sim N_4=: N$\\ 
From Lemma \ref{lem:b3tilde}, we easily check, by using $2\sigma+1-\alpha<0$, that
$$
|\widetilde{a}_{411}| + |\widetilde{a}_{422}| \lesssim M_2M_2^{-1}N^{2\sigma+1-\alpha} + N^{2\sigma+1-\alpha} \ll N_1^{2\sigma+1-\alpha}.
$$
Next, we claim that
$$
|\widetilde{a}_{412}| \lesssim N_1^{2\sigma+1-\alpha}.
$$
Indeed, if $N_3\ll N_1$, one has $M_2\sim N_1$ and $|\widetilde{b}_3(\xi_1+\xi_3, -\xi_1)| \lesssim N_1^{2\sigma-\alpha}$. In the case $N_3\sim N_1$, it holds $|\widetilde{b}_3(\xi_1+\xi_3, -\xi_1)| \lesssim M_2^{-1}N_1^{2\sigma+1-\alpha}$.\\
It remains to estimate
\begin{align*}
\widetilde{a}_{421} + \widetilde{a}_{43} &= 
\xi_2[\widetilde{b}_3(\xi_2+\xi_3,\xi_1) - \widetilde{b}_3(\xi_2, \xi_1)]  + [\xi_3\widetilde{b}_3(\xi_2+\xi_3,\xi_1) - \xi_1\widetilde{b}_3(\xi_2,\xi_1)]\\
&= \xi_2 \xi_3 \partial^{(1,0)} \widetilde{b}_3(\xi^\ast, \xi_1) + \mathcal{O}(N^{2\sigma+1-\alpha}), \quad |\xi^\ast|\sim N,\\
&= \mathcal{O}(N_3 N_1^{-1} N^{2\sigma+1-\alpha}) + \mathcal{O}(N_1^{2\sigma+1-\alpha})\\
&= \mathcal{O}(N_1^{2\sigma+1-\alpha}),
\end{align*}
where we used Lemma \ref{lem:diff-b3tilde}. This concludes the proof of \eqref{a4tilde-C4}.
\end{proof}

\section*{Acknowledgements} 
The authors would like to thank Mats Ehrnstr\"om for his valuable comments on an earlier version of this work.
D.P. was partially supported by the European research Council (ERC) under
the European Union's Horizon 2020 research and innovation programme (Grant agreement 101097172-GEOEDP). He also would like to thank the University Paris-Saclay for the kind hospitality during the redaction of this article. 

\bibliographystyle{amsplain}

\end{document}